\documentclass[reqno,11pt]{amsart}
\usepackage{amssymb, amsmath,latexsym,amsfonts,amsbsy, amsthm}
\usepackage{color}

\let\pa=\partial

\let\e=\varepsilon

\let\f=\frac

\let\ve=\varepsilon
\let\na=\nabla

\let\tre=\triangleq

\def\R{\mathbb{R}}

\def\no{\noindent}

\def\eqdef{\buildrel\hbox{\footnotesize def}\over =}

\newcommand{\beq}{\begin{equation}}
\newcommand{\eeq}{\end{equation}}
\newcommand{\ben}{\begin{eqnarray}}
\newcommand{\een}{\end{eqnarray}}
\newcommand{\beno}{\begin{eqnarray*}}
\newcommand{\eeno}{\end{eqnarray*}}

\renewcommand{\theequation}{\thesection.\arabic{equation}}

\newtheorem{Theorem}{Theorem}[section]
\newtheorem{Definition}[Theorem]{Definition}
\newtheorem{Proposition}[Theorem]{Proposition}
\newtheorem{Lemma}[Theorem]{Lemma}

\newtheorem{Remark}[Theorem]{Remark}

\begin{document}

\title[Zero-viscosity limit of the Navier-Stokes equations]
{On the zero-viscosity limit of the Navier-Stokes equations in the half-space}

\author{Mingwen Fei}
\address{School of  Mathematics and Computer Sciences, Anhui Normal University, Wuhu, China}
\email{ahnufmwen@126.com}

\author{Tao Tao}
\address{School of  Mathematical Sciences, Peking University, Beijing 100871, China}
\email{taotao@amss.ac.cn}

\author{Zhifei Zhang}
\address{School of  Mathematical Sciences, Peking University, Beijing 100871, China}
\email{zfzhang@math.pku.edu.cn}

\date{\today}
\maketitle

\renewcommand{\theequation}{\thesection.\arabic{equation}}
\setcounter{equation}{0}

\begin{abstract}
We consider the zero viscosity limit of the incompressible Navier-Stokes equations with non-slip boundary condition in the half-space
for the initial vorticity located away from the boundary. By using the vorticity formulation and Cauchy-Kowaleskaya theorem, Maekawa proved the local in time convergence of the Navier-Stokes equations in the half- plane to the Euler equations outside a boundary layer and to the Prandtl equations in the boundary layer. In this paper, we develop the direct energy method to generalize Maekawa's result to the half-space.
\end{abstract}

\numberwithin{equation}{section}

\section{Introduction}

In this paper, we are concerned with the zero-viscosity limit of the incompressible Navier-Stokes equations in the half-space $\R^3_+$:
\begin{eqnarray}
\left \{
\begin {array}{ll}
\partial_tu^{\varepsilon}+u^{\varepsilon}\cdot\na_xu^{\varepsilon}+v^{\varepsilon}\partial_yu^{\varepsilon}+\na_xp^{\varepsilon}=\varepsilon^2
\Delta u^{\varepsilon},\\[3pt]
 \partial_tv^{\varepsilon}+u^{\varepsilon}\cdot\na_xv^{\varepsilon}+v^{\varepsilon}\partial_yv^{\varepsilon}+\partial_yp^{\varepsilon}=\varepsilon^2
\Delta v^{\varepsilon},\\[3pt]
 \na_x\cdot u^{\varepsilon}+ \partial_yv^{\varepsilon}=0,\\[3pt]
  (u^{\varepsilon}, v^{\varepsilon})(t,x,0)=(0,0).\label{ins}
\end{array}
\right.
\end{eqnarray}
Here and in what follows, $(x,y)\in\mathbb{R}^{2}\times\mathbb{R}_+$ and $\na_x=(\partial_{x_1},\partial_{x_2})$, $u^{\varepsilon}=(u^{\varepsilon}_1,u^{\varepsilon}_2)$, $\varepsilon^2$ is the viscosity coefficient, $( u^{\varepsilon},v^{\varepsilon})$ and $p^{\varepsilon}$ denote the velocity field and the pressure respectively.

In the absence of the boundary, the Navier-Stokes equations indeed converge to the Euler equations
\begin{align}\label{equ:Euler}
\left\{
\begin{aligned}
&\pa_t u^{e}+ u^{e}\cdot\na_x u^{e}+v^{e}\pa_y u^{e}+\na_x p^{e}=0,\\
&\pa_t v^{e}+u^{e}\cdot\na_x v^{e}+v^{e}\pa_y v^{e}+\pa_y p^{e} =0,\\
&\na_x\cdot u^{e}+\pa_y v^{e}=0.
\end{aligned}
\right.
\end{align}
This problem has been well studied in various functional settings \cite{K1, Swann, CW, Mar, Mas}.

In the presence of the boundary, the zero-viscosity limit will become very complicated due to the possible appearance of boundary layer. For the Navier slip boundary condition
\beno
v^{\varepsilon}=0,\quad \pa_yu^{\varepsilon}=0\quad\text{on}\quad y=0,
\eeno
the boundary layer is weak. In such case, the limit  from the Navier-Stokes equations to the Euler equations
was justified by Xiao and Xin for the half-space \cite{XX} and by Rousset and Masmoudi \cite{MR} for general domain,
see \cite{IP, IS, WXZ} and references therein for more relevant results. 
For the non-slip boundary condition, the boundary layer is strong. In 1904, Prandtl introduced the boundary layer theory in \cite{Pr}. Using a formal boundary layer expansion
\ben\label{formal expan}
 \left\{
 \begin{array}{l}
 u^{\e}(t,x,y) =u^{e}(t,x,y)+ u^{p}(t,x,\f{y}{\e})+O(\e),\\
 v^{\e}(t,x,y)= v^{e}(t,x,y)+\e v^{p}(t,x,\f{y}{\e})+O(\e),
 \end{array}\right.
 \een
  he derived the Prandtl boundary layer equation
 \begin{eqnarray}\label{equ:P}
\left\{\begin{aligned}
&\partial_t u+u\cdot\na_x u+v\pa_y u+\pa_xp=\pa^2_{y}u,\\
&\na_x\cdot u+\pa_y v=0,\\
&u|_{y=0}=v|_{y=0}=0,\quad\lim_{y\rightarrow+\infty}u(t,x,y)=U(t,x).
\end{aligned}\right.
\end{eqnarray}
Up to now, the justification of this formal boundary expansion is still a challenging problem.

The first step toward this problem is to deal with the well-posedness of the Prandtl equation. Initiated by Oleinik and Samokhin \cite{OS}, the well-posed problem was well understood for the monotonic data in Sobolev spaces \cite{XZ, AW, MW, GD} and general analytic data  \cite{SC1,LCS,GM,ZZ}.

The first rigorous verification of the Prandtl boundary layer theory was achieved in the analytic setting by Sammartino and Caflisch \cite{SC2}(see also \cite{WWZ} for a proof based on direct energy method). In the case when the domain and the initial data have a circular symmetry,
the convergence was justified in \cite{LMT, MT}. Guo and Nguyen \cite{GN} justify the zero-viscosity limit of steady Navier-Stokes equations over a moving plate.
Initiated by Kato \cite{K2},  there are many works devoted to the conditional convergence \cite{Ke1, Ke2, WX, CKV}.

Recently, Maekawa \cite{M}  justified the zero-viscosity limit for the initial vorticity located away from the boundary in the half-plane.  A very interesting point is that this kind of data is only analytic near the boundary.
Intuitively, this seems enough to exclude the instability of  boundary layer.
However,  the  proof in \cite{M} used another mechanism in a crucial way
: weak interaction between the outer vorticity and the inner vorticity.
\smallskip

The goal of this paper is two fold. The first one is to generalize Maekawa's result to the half space $\R^3_+$.  In $\R^3_+$,  the data with vorticity located away from the boundary is not analytic near the boundary.
However, we find that the data is still analytic in the tangential direction near the boundary. Indeed, the tangential analyticity is enough to ensure the well-posedness of the Prandtl equation \cite{LCS,ZZ}. The second one is to develop a direct energy method for the zero-viscosity limit problem. The proof in \cite{SC2, M} is based on the Cauchy-Kowaleskaya theorem, where the representation formula of the solution was used in a crucial way. In particular, the representation formula of the vorticity is used in \cite{M}. So, this method seems difficult to apply to the zero-viscosity limit problem in general physical domain. While, energy method may be applicable for the case of general domain.

For the simplicity, we consider the initial data of  the form
\beno
u^{\e}(0,x,y)=u_0(x,y),\quad v^\e(0,x,y)=v_0(x,y),
\eeno
which satisfies
\ben\label{ass:initial}
\nabla_x\cdot u_0+\pa_yv_0=0,\quad u_0(x,0)=0, \quad v_0(x,0)=0,
\een
and  the initial vorticity $\omega_0=\text{curl}(u_0,v_0)$ satisfies
\begin{align}\label{e:initial vorticity assumption}
 2d_0\tre\text{dist}\big(\text{supp}\omega_0, \{y=0\}\big)>0.
\end{align}
Without loss of generality, we take $d_0=1$.

To state our main result, we introduce the following Prandtl system
\begin{align}\label{e:prandtl equation}
\left\{
\begin{aligned}
&\partial_tu^p-\partial_{zz}u^p+u^p\cdot\na_xu^e(t,x,0)+\big(u^e(t,x,0)+u^p\big)\cdot\na_xu^p\\
&\qquad\qquad\qquad\qquad+\Big(v^p-\int_0^\infty\partial_xu^p(t,x,z)dz+ z\partial_yv^e(t,x,0)\Big)\partial_zu^p=0,\\[3pt]
&\na_x\cdot u^p+\partial_zv^p=0,\\[3pt]
&u^p(0,x,y)=0,\\[3pt]
&\lim\limits_{z\rightarrow \infty}(u^p, v^p)(t,x,z)=0,\quad u^p(t,x,0)=-u^e(t,x,0),
\end{aligned}
\right.
\end{align}
where $(u^e, v^e, p^e)$ is the solution of  the Euler equations (\ref{equ:Euler}).
\smallskip

Now, our main result is stated as follows.

\begin{Theorem}\label{thm:main}
There exist $T> 0$ and $C> 0$ independent of $\varepsilon$ such that for any $(u_0,v_0)\in H^{30}(\R^3_+)$ satisfying \eqref{ass:initial}-(\ref{e:initial vorticity assumption}), there exists a unique solution $(u^\varepsilon, v^\varepsilon)$ of the Navier-Stokes equations (\ref{ins}) in $[0,T]$, which satisfies
\begin{align}
\sup_{0\leq t\leq T}\big\|u^\varepsilon(t,x,y)-u^e(t,x,y)-u^p(t,x,\frac{y}{\varepsilon})\big\|_{L^2\cap L^\infty(\R^3_+)}\leq C\varepsilon,\nonumber\\
\sup_{0\leq t\leq T}\big\|v^\varepsilon(t,x,y)-v^e(t,x,y)-\varepsilon v^p(t,x,\frac{y}{\varepsilon})\big\|_{L^2\cap L^\infty(\R^3_+)}\leq C\varepsilon.\nonumber
\end{align}
\end{Theorem}

Let us present a sketch of the proof and ideas.

\begin{itemize}

\item[1.] Construction of the approximate solution $U^a=(u^a, v^a)$ of the system (\ref{ins}) by using the asymptotic matched expansion method.

\item[2.] The error $U^R$ between the solution and the approximate solution satisfies
$$\partial_tU^R-\varepsilon^2\triangle U^R+U^R\cdot\nabla U^a+U^a\cdot\nabla U^R+U^R\cdot\nabla U^R+\nabla P^R=R.$$
In this equation, the main trouble is to control the linear terms $U^R\cdot\nabla U^a$ and $U^a\cdot\nabla U^R$, while the nonlinear term $U^R\cdot\nabla U^R$ is in fact easy to control. One of the most singular terms is
$$v^R\partial_y(u^p(t,x,\frac{y}{\varepsilon}))=\frac{v^R}{y}z\partial_zu^p(t,x,\frac {y}{\varepsilon})\sim -\na_x\cdot u^R\partial_zu^p(t,x,\frac {y}{\varepsilon}).
$$
So, this term will lead to the loss of one horizontal derivative in the process of energy estimates. To remedy the loss of the derivative, it is natural to work in the analytic setting.
In our case, we will use the tangential analyticity to recover one derivative loss near the boundary, and use the exponential decay in $z$ of $u^p(t,x,z)$ away from the boundary.

\item[4.] In order to avoid the singularity like $\pa_yu^p(t,x,\f y \varepsilon)\sim \f 1 \varepsilon$, it is better to work in the conormal Sobolev spaces with $\pa_y$ replaced by  conormal derivative $y\pa_y$. The disadvantage is that we have no control on the regularity in $y$ variable near the boundary.

\item[5.] To gain one derivative in $y$ variable, we need to use the vorticity formulation of the error equation, which takes the form
\begin{eqnarray}
&\partial_tw-\varepsilon^2\triangle w+\widetilde{U}_a\cdot\nabla w+\widetilde{U}\cdot\nabla w_a+\widetilde{U}\cdot\nabla w
-w_{a}\cdot\nabla U-w\cdot\nabla U_a-w\cdot\nabla U\nonumber\\
&=-{\rm curl}(R_h,R_v)-M,\nonumber
\end{eqnarray}
where the boundary condition of the vorticity can be determined by introducing the Dirichlet-Neumann operator.

\item[6.] In the vorticity formulation, one of trouble terms is $\widetilde{v}_a\pa_yw.$ To handle it, we need to decompose the vorticity into two parts: Euler part $w_e$ and Prandtl part $w_p$. The Euler part has the exponential decay in $\varepsilon$ near the boundary, and  the Prandtl part has the exponential decay in $\f y \varepsilon$.

\item[7.] The third component of the vorticity has better behaviour. This observation is crucial to close our estimate.

\item[8.] The most subtle task is to construct a suitable energy functional to reveal all mechanism such as the analyticity near the boundary,
the exponential decay in $\varepsilon$ of $w_e$ near the boundary and
the exponential decay in $\f y \varepsilon$ of $w_p$. Some ideas are motivated by \cite{CGP, PZ}.
\end{itemize}

The paper is organized as follows. In Section 2, we construct the approximate solution by using the matched asymptotic expansion method. In Section 3, we derive the error equation and give a decomposition of vorticity formulation of the error equation. Section 4 is devoted to the functional framework and some product estimates. In Section 5, we construct the energy functional and prove Theorem \ref{thm:main} under some assumptions. Section 6-Section 12 is devoted to the key energy estimates in the analytic setting and Sobolev setting for the velocity and the vorticity. Finally, we present the well-posedness of the Euler system and the Prandtl system in the appendix.

\section{Construction of the approximate solution}

In this section, we use the asymptotic matched method to construct the approximate solution.

\subsection{Outer(Euler) expansions }

Away from the boundary, we construct the approximate solution by the following expansions
\begin{align*}
&u^{\varepsilon}(t,x,y)=u_e^{(0)}(t,x,y)+\varepsilon u_e^{(1)}(t,x,y)+\cdots,\\&
v^{\varepsilon}(t,x,y)=v_e^{(0)}(t,x,y)+\varepsilon v_e^{(1)}(t,x,y)+\cdots,\\&
p^{\varepsilon}(t,x,y)=p_e^{(0)}(t,x,y)+\varepsilon p_e^{(1)}(t,x,y)+\cdots.
\end{align*}
By substituting the above expansions into (\ref{ins}) and matching the (leading) zeroth order terms, we find that $(u_e^{(0)},v_e^{(0)},p_e^{(0)})$ should satisfy the Euler equations
\begin{eqnarray}\label{e:Euler equation}
\left \{
\begin {array}{ll}
\partial_tu_e^{(0)}+u_e^{(0)}\cdot\na_xu_e^{(0)}+v_e^{(0)}\partial_yu_e^{(0)}+\na_xp_e^{(0)}=0,\\[3pt]
 \partial_tv_e^{(0)}+u_e^{(0)}\cdot\na_xv_e^{(0)}+v_e^{(0)}\partial_yv_e^{(0)}+\partial_yp_e^{(0)}=0,\\[3pt]
 \na_x\cdot u_e^{(0)}+ \partial_yv_e^{(0)}=0,
\end{array}
\right.
\end{eqnarray}
which will be equipped with the boundary condition
\begin{align*}
v_e^{(0)}(t,x,0)=0,\ \ t\geq0, \ x\in\mathbb{ R}^2,
\end{align*}
and the initial condition
\begin{align*}
u_e^{(0)}(0,x,y)=u_0(x,y),  \ v_e^{(0)}(0,x,y)=v_0(x,y),\ (x,y)\in\mathbb{R}^2\times \mathbb{R}_+.
\end{align*}

By matching the $\varepsilon$-order terms, we find that $(u_e^{(1)},v_e^{(1)},p_e^{(1)})$ should satisfy the linearized Euler equations
\begin{eqnarray}\label{e:linearized Euler equation}
\left \{
\begin {array}{ll}
\partial_tu_e^{(1)}+u_e^{(1)}\cdot\na_xu_e^{(0)}+v_e^{(1)}\partial_yu_e^{(0)}+u_e^{(0)}\cdot\na_xu_e^{(1)}+v_e^{(0)}\partial_yu_e^{(1)}
+\na_xp_e^{(1)}=0,\\[3pt]
 \partial_tv_e^{(1)}+u_e^{(1)}\cdot\na_xv_e^{(0)}+v_e^{(1)}\partial_yv_e^{(0)}+u_e^{(0)}\cdot\na_xv_e^{(1)}+v_e^{(0)}\partial_yv_e^{(1)}
 +\partial_yp_e^{(1)}=0,\\[3pt]
 \na_x\cdot u_e^{(1)}+ \partial_yv_e^{(1)}=0,\\[3pt]
 (u_e^{(1)}, v_e^{(1)})(0,x,y)=(0, 0).
\end{array}
\right.
\end{eqnarray}
Here the boundary condition on $v_e^{(1)}$ is determined by $v_p^{(1)}$
$$v_e^{(1)}|_{y=0}=-v_p^{(1)}|_{y=0}.$$

\subsection{Boundary(Prandtl) layer  expansions }

Near the boundary,  we will make the boundary layer(Prandtl layer) expansion. For this,  we introduce the scaled(Prandtl) variable  $z=\f y \ve\in[0,+\infty)$ and  write
\begin{align}
&u^{\varepsilon}(t,x,y)=u_P^{(0)}(t,x,y,z)+\varepsilon u_P^{(1)}(t,x,y,z)+\cdots,\nonumber\\&
v^{\varepsilon}(t,x,y)=v_P^{(0)}(t,x,y,z)+\varepsilon v_P^{(1)}(t,x,y,z)+\cdots,\nonumber\\&
p^{\varepsilon}(t,x,y)=p_P^{(0)}(t,x,y,z)+\varepsilon p_P^{(1)}(t,x,y,z)+\cdots,\nonumber
\end{align}
where for every $i\in\{1,2,\cdots\}$,
\begin{align}
&u_P^{(i)}(t,x,y,z)=u_e^{(i)}(t,x,y)+u_p^{(i)}(t,x,z),
\nonumber\\& v_P^{(i)}(t,x,y,z)=v_e^{(i)}(t,x,y)+v_p^{(i)}(t,x,z),
\nonumber\\& p_P^{(i)}(t,x,y,z)=p_e^{(i)}(t,x,y)+p_p^{(i)}(t,x,z).\nonumber
\end{align}
The matched boundary condition requires that
\begin{align}\label{e:matched boundary}
u_p^{(i)}(t,x,z)\rightarrow 0, \ \ v_p^{(i)}(t,x,z)\rightarrow 0, \ \ p_p^{(i)}(t,x,z)\rightarrow 0,\ \ \ \text{as} \ z\rightarrow+\infty.
\end{align}
While, the boundary condition of $(u^\varepsilon, v^\varepsilon)$ on $y=0$ requires that
\begin{align}
u_p^{(i)}(t,x,0)=-u_e^{(i)}(t,x,0),\ \ v_e^{(i)}(t,x,0)=-v_p^{(i)}(t,x,0),\ i=0,1,\cdots.\label{ep:boundary}
\end{align}

To derive the equation of $(u_p^{(i)}, v_p^{(i)}, p_p^{(i)})$, we put the expansions into (\ref{ins}) and then put the terms with the same order in $\varepsilon$ together.

First of all,  we deduce from the $\varepsilon^{-1}$-order terms  and boundary condition (\ref{e:matched boundary}) that
\begin{align}
v_p^{(0)}=0,\quad p_p^{(0)}=0.\nonumber
\end{align}
Then, collecting $\varepsilon^0$-th term of the $u^\varepsilon$ equation, divergence free condition and boundary condition, we obtain
\begin{eqnarray}\label{e:prandtl equation}
\left \{
\begin {array}{ll}
&\partial_tu_p^{(0)}-\partial_{zz}u_p^{(0)}+u_p^{(0)}\cdot\na_xu_e^{(0)}(t,x,0)+\big(u_p^{(0)}+u_e^{(0)}(t,x,0)\big)\cdot\na_xu_p^{(0)}\\
&\qquad\qquad\qquad\qquad+\big(v_p^{(1)}+v_e^{(1)}(t,x,0)+ z\partial_yv_e^{(0)}(t,x,0)\big)\partial_zu_p^{(0)}=0,\\[3pt]
&\na_x\cdot u_p^{(0)}+\partial_zv_p^{(1)}=0,\\[3pt]
&u_p^{(0)}(0,x,y)=0,\\[3pt]
&\lim\limits_{z\rightarrow +\infty}(u_p^{(0)}, v_p^{(1)})(t,x,z)=0,\\[3pt]
&u_p^{(0)}(t,x,0)=-u_e^{(0)}(t,x,0).
\end{array}
\right.
\end{eqnarray}
While, using $v_p^{(0)}=0$ and boundary condition (\ref{e:matched boundary}),  collecting $\varepsilon^0$-order term of the $v^\varepsilon$ equation, we obtain
$$p_p^{(1)}=0.$$
\begin{Remark}
We set
\begin{align}
&\widetilde{u}_p^{(0)}(t,x,z)\tre u_p^{(0)}(t,x,z)+u_e^{(0)}(t,x,0),\nonumber
&\widetilde{v}_p^{(1)}(t,x,z)\tre v_p^{(1)}(t,x,z)+v_e^{(1)}(t,x,0)+ z\partial_yv_e^{(0)}(t,x,0).\nonumber
\end{align}
Then, by Bernoulli law,
\begin{align}
\partial_tu_e^{(0)}(t,x,0)+u_e^{(0)}(t,x,0)\cdot\na_xu_e^{(0)}(t,x,0)+\na_x p_e^{(0)}(t,x,0)=0,\nonumber
\end{align}
we arrive at
\begin{eqnarray}
\left \{
\begin{array}{lll}
&\partial_t\widetilde{u}_p^{(0)}-\partial_{zz}\widetilde{u}_p^{(0)}+\widetilde{u}_p^{(0)}\cdot\na_x\widetilde{u}_p^{(0)}+\widetilde{v}_p^{(1)}\partial_z\widetilde{u}_p^{(0)}
+\na_x p_e^{(0)}(t,x,0)=0,\\[3pt]
&\na_x\cdot\widetilde{u}_p^{(0)}+\partial_z\widetilde{v}_p^{(1)}=0,\\[3pt]
&\widetilde{u}_p^{(0)}(0,x,z)=u_e^{(0)}(0,x,0),\\[3pt]
&\lim\limits_{z\rightarrow +\infty}\widetilde{u}_p^{(0)}(t,x,z)=u_e^{(0)}(t,x,0),\\[3pt]
&\widetilde{u}_p^{(0)}(t,x,0)=0,\ \ \widetilde{v}_p^{(1)}(t,x,0)=0.\nonumber
\end{array}
\right.
\end{eqnarray}
This is just the Prandtl equation.
\end{Remark}

Finally, we collect the $\varepsilon$-order terms that
\begin{align}\label{e:linearized prandtl equation}
&\partial_tu_p^{(1)}-\partial_{zz}u_p^{(1)}+\big(u_p^{(0)}+u_e^{(0)}(t,x,0)\big)\cdot\na_xu_p^{(1)}+\big(v_p^{(1)}+v_e^{(1)}(t,x,0)+z\partial_yv_e^{(0)}(t,x,0)\big)\partial_zu_p^{(1)}
\nonumber\\&\qquad+u_p^{(1)}\cdot\na_x\big(u_p^{(0)}+u_e^{(0)}(t,x,0)\big)+\big(u_e^{(1)}(t,x,0)+z\partial_yu_e^{(0)}(t,x,0)\big)\cdot\na_xu_p^0\nonumber\\
&\qquad +\Big(v_p^{(2)}-v_p^{(2)}(t,x,0)+z\partial_yv_e^{(1)}(t,x,0)+\frac12z^2\partial_{yy}v_e^{(0)}(t,x,0)\Big)\partial_zu_p^{(0)}\nonumber\\
&=-\Big(u_p^{(0)}z\cdot\na_x\partial_{y}u_e^{(0)}(t,x,0)+u_p^{(0)}\cdot\na_xu_e^{(1)}(t,x,0)+v_p^{(1)}\partial_yu_e^{(0)}(t,x,0)\Big),
\end{align}
with the initial condition $u_p^{(1)}(0,x,z)=0$ and the boundary condition
$$u_p^{(1)}(t,x,0)=-u_e^{(1)}(t,x,0),\quad \lim\limits_{z\rightarrow+\infty} u_p^{(1)}(t,x,z)=0.$$
Here $v_p^{(2)}$ is determined by
$$v_p^{(2)}(t,x,z)=\int_z^{+\infty} \na_x\cdot u_p^{(1)}(t,x,z')dz'.$$\\
Moreover, the pressure $p_p^{(2)}$ is determined by
$$p_p^{(2)}(t,x,z)=-\int_z^{+\infty}\mathcal{P}_2(t,x,z')dz',$$
where
\begin{align}
&\mathcal{P}_2= \partial_{zz}v_p^{(1)}-\partial_tv_p^{(1)}-u_e^{(0)}(t,x,0)\cdot\na_xv_p^{(1)}-u_p^{(0)}\cdot\big(\na_xv_e^{(1)}(t,x,0)
+\na_xv_p^{(1)}\big)-v_p^{(1)}\partial_yv_e^{(0)}(t,x,0)
\nonumber\\&\qquad-\big(v_e^{(1)}(t,x,0)+v_p^{(1)}\big)\partial_zv_p^{(1)}
-z\partial_{y}\na_xv_e^{(0)}(t,x,0)\cdot u_p^{(0)}-z\partial_{y}v_e^{(0)}(t,x,0)\partial_zv_p^{(1)}.\nonumber
\end{align}

\begin{Remark}
These equations can be solved in the following way
$$(u_e^{(0)},v_e^{(0)},p_e^{(0)})\rightarrow (u_p^{(0)},v_p^{(1)})\rightarrow (u_e^{(1)},v_e^{(1)},p_e^{(1)})\rightarrow (u_p^{(1)},v_p^{(2)},p_p^{(2)}). $$
\end{Remark}

\subsection{Construction of the approximate solution}

Let us define the approximate solution ($u_a^{\varepsilon},v_a^{\varepsilon},p_a^{\varepsilon}$) as follows
\begin{align*}
&u_a^{\varepsilon}(t,x,y)\eqdef\sum\limits_{i=0}^{1}\varepsilon^iu_e^{(i)}(t,x,y)+\sum\limits_{i=0}^{1}\varepsilon^iu_p^{(i)}\Big(t,x,\frac{y}{\varepsilon}\Big),\\&
v_a^{\varepsilon}(t,x,y)\eqdef\sum\limits_{i=0}^{1}\varepsilon^iv_e^{(i)}(t,x,y)+\sum\limits_{i=1}^{2}\varepsilon^iv_p^{(i)}\Big(t,x,\frac{y}{\varepsilon}\Big),\\&
p_a^{\varepsilon}(t,x,y)\eqdef\sum\limits_{i=0}^{1}\varepsilon^ip_e^{(i)}(t,x,y)+\varepsilon^2p_p^{(2)}\Big(t,x,\frac{y}{\varepsilon}\Big).
\end{align*}
We set
\ben\label{def:f}
f(t,x)\tre\int_{0}^\infty\partial_xu_p^{(1)}(t,x,z)dz.
\een
In view of the process of the asymptotic expansions, a straightforward computation gives
\begin{eqnarray}
\left \{
\begin {array}{ll}
\partial_tu_a^{\varepsilon}+u_a^{\varepsilon}\cdot\na_xu_a^{\varepsilon}+(v_a^{\varepsilon}-\varepsilon^2f(t,x)e^{-y})\partial_yu_a^{\varepsilon}
+\na_xp_a^{\varepsilon}-\varepsilon^2
\Delta u_a^{\varepsilon}
=-R_{e,h}-R_{p,h}\tre-R_h,\\[3pt] \partial_tv_a^{\varepsilon}+u_a^{\varepsilon}\cdot\na_xv_a^{\varepsilon}+(v_a^{\varepsilon}-\varepsilon^2f(t,x)e^{-y})\partial_yv_a^{\varepsilon}
+\partial_yp_a^{\varepsilon}-\varepsilon^2
\Delta v_a^{\varepsilon}
=-R_{e,v}-R_{p,v}\tre -R_v,\\[3pt]
 \na_x\cdot u_a^{\varepsilon}+\partial_yv_a^{\varepsilon}=0,\\[3pt]
(u_a^\varepsilon, v_a^\varepsilon)(0,x,y)=(u_0(x,y),v_0(x,y)),\\[3pt]
(u_a^\varepsilon, v_a^\varepsilon)(t,x,0)=(0,\varepsilon^2f(t,x)),\label{approximate-equation}
\end{array}
\right.
\end{eqnarray}
where
\begin{align}\label{d:difinition on error term}
-R_{e,h}&=\varepsilon^2\Big(u^{(1)}_e\cdot\na_xu^{(1)}_e+v_e^{(1)}\partial_yu^{(1)}_e-fe^{-y}(\partial_yu^{(0)}_e+\varepsilon\partial_yu^{(1)}_e)\Big)
-\varepsilon^2\triangle (u^{(0)}_e+\varepsilon u^{(1)}_e),
\nonumber\\
-R_{e,v}&=\varepsilon^2\Big(u^{(1)}_e\cdot\na_xv^{(1)}_e+v^{(1)}_e\partial_yv^{(1)}_e-fe^{-y}(\partial_yv^{(0)}_e+\varepsilon\partial_yv^{(1)}_e)\Big)
-\varepsilon^2\triangle (v^{(0)}_e+\varepsilon v^{(1)}_e),\nonumber
\end{align}
and
\begin{align}
-R_{p,h}&=\varepsilon^2\Big(u_e^{(1)}\cdot\na_xu_p^{(1)}+u_p^{(1)}\cdot\na_xu_e^{(1)}+u_p^{(1)}\cdot\na_xu_p^{(1)}
+v_p^{(1)}\partial_yu_e^{(1)}+v_p^{(2)}(\partial_yu_e^{(0)}+\varepsilon \partial_yu_e^{(1)})\nonumber\\
&\qquad-fe^{-y}\partial_zu_p^{(1)}+v_p^{(2)}\partial_zu_p^{(1)}\Big)+\varepsilon\Big((u_e^{(0)}-u_e^{(0)}(t,x,0))\cdot\na_xu_p^{(1)}+(\na_xu_e^{(1)}
-\na_xu_e^{(1)}(t,x,0))\cdot u_p^{(0)}\nonumber\\
&\qquad+\big(u_e^{(1)}-u_e^{(1)}(t,x,0)\big)\cdot\na_xu_p^{(0)}+\big(\na_xu_e^{(0)}
-\na_xu_e^{(0)}(t,x,0)\big)u_p^{(1)}+\big(v_e^{(1)}-v_e^{(1)}(t,x,0)\big)\partial_zu_p^{(1)}\nonumber\\
&\qquad+(\partial_yu_e^{(0)}-\partial_yu_e^{(0)}(t,x,0))v_p^{(1)}+f(1-e^{-y})\partial_zu_p^{(0)}\Big)\nonumber\\
&\quad+\frac{1}{\varepsilon}\Big(v_e^{(0)}-y\partial_yv_e^{(0)}(t,x,0)-\frac{y^2}{2}\partial_{yy}v_e^{(0)}(t,x,0)\Big)\partial_zu_p^{(0)}+
(v_e^{(0)}-y\partial_yv_e^{(0)}(t,x,0))\partial_zu_p^{(1)}\nonumber\\
&\quad+\big(v_e^{(1)}-v_e^{(1)}(t,x,0)-y\partial_yv_e^{(1)}(t,x,0)\big)\partial_zu_p^{(0)}
+\big(u_e^{(0)}-u_e^{(0)}(t,x,0)-y\partial_yu_e^{(0)}(t,x,0)\big)\cdot\na_xu_p^{(0)}\nonumber\\
&\quad+\big(\na_xu_e^{(0)}
-\na_xu_e^{(0)}(t,x,0)-y\partial_{y}\na_xu_e^{(0)}(t,x,0)\big)\cdot u_p^{(0)}-\varepsilon^2\Delta_{x}u_{a,p}+\varepsilon^2\partial_xp_p^{(2)},
\nonumber\\
-R_{p,v}&=\varepsilon^2\Big(\partial_tv_p^{(2)}+(u_e^{(0)}+u_p^{(0)})\cdot\na_xv_p^{(2)}
+(v_e^{(1)}+v_p^{(1)})\partial_zv_p^{(2)}+(u_e^{(1)}+u_p^{(1)})\cdot\na_xv_p^{(1)}
\nonumber\\
&\quad+v_p^{(2)}\partial_yv_{a,e}+(v_p^{(2)}-fe^{-y})\partial_zv_{a,p}+u_p^{(1)}\cdot\na_xv_e^{(1)}+v_p^{(1)}\partial_yv_e^{(1)}-\partial_{zz}v_p^{(2)}\Big)
\nonumber\\
&\quad+\varepsilon\Big((u_e^{(0)}-u_e^{(0)}(t,x,0))\cdot\na_xu_p^{(1)}+(\na_xv_e^{(1)}
-\na_xv_e^{(1)}(t,x,0))\cdot u_p^{(0)}\nonumber\\
&\qquad+(v_e^{(1)}-v_e^{(1)}(t,x,0))\partial_zv_p^{(1)}
+(\partial_yv_e^{(0)}-\partial_yv_e^{(0)}(t,x,0))v_p^{(1)}+ u_p^{(1)}\cdot\na_xv_e^{(0)}+ v_e^{(0)}\partial_zv_p^{(2)}\Big) \nonumber\\
&\quad+(v_e^{(0)}-y\partial_yv_e^{(0)}(t,x,0)\big)\partial_zv_p^{(1)}+\big(\na_xv_e^{(0)}
-y\partial_{y}\na_xv_e^{(0)}(t,x,0)\big)\cdot u_p^{(0)}
\nonumber\\
&\quad+\varepsilon^3\big(u_e^{(1)}+u_p^{(1)}\big)\cdot\na_xv_p^{(2)} -\varepsilon^3\Delta_{x}v_{a,p}.\nonumber
\end{align}
Here and in what follows, $u_{a,e}, v_{a,e}, u_{a,p}, v_{a,p}$ are defined by
\begin{align}
u_{a,e}(t,x,y)&=u_e^{(0)}(t,x,y)+\varepsilon u_e^{(1)}(t,x,y),\quad v_{a,e}(t,x,y)=v_e^{(0)}(t,x,y)+\varepsilon v_e^{(1)}(t,x,y),\nonumber\\
u_{a,p}(t,x,z)&=u_p^{(0)}(t,x,z)+\varepsilon u_p^{(1)}(t,x,z),\quad v_{a,p}(t,x,z)=v_p^{(1)}(t,x,z)+\varepsilon v_p^{(2)}(t,x,z).\nonumber
\end{align}
Formally,  it holds that
\begin{align}
R_{e,h}=O(\varepsilon^{2}),\ \ R_{p,h}=O(\varepsilon^{2}),\ \ R_{e,v}=O(\varepsilon^{2}),\ \ R_{p,v}=O(\varepsilon^{2}).\nonumber
\end{align}
Later on, we will make them precise.

\section{The error system and vorticity formulation}

To justify the boundary layer expansion, the most key ingredient is to show that the remainder is
uniformly small in a suitable sense.

\subsection{The error system}
We introduce the error between the solution and the approximate solution
\begin{align}
u_R^{\varepsilon}\eqdef u^{\varepsilon}-u_a^{\varepsilon},\quad v_R^{\varepsilon}\eqdef v^{\varepsilon}-v_a^{\varepsilon},\quad p_R^{\varepsilon}\eqdef p^{\varepsilon}-p_a^{\varepsilon}.\nonumber
\end{align}
From (\ref{ins}) and (\ref{approximate-equation}), we deduce that $(u_R^\varepsilon, v_R^\varepsilon, p_R^\varepsilon)$ satisfies the following system
\begin{eqnarray}\label{error-equation}
\left \{
\begin {array}{ll}
\partial_tu_R^{\varepsilon}+u_a^{\varepsilon}\cdot\na_xu_R^{\varepsilon}+(v_a^{\varepsilon}-\varepsilon^2f(t,x)e^{-y})\partial_yu_R^{\varepsilon}
+u_R^{\varepsilon}\cdot\na_xu_R^{\varepsilon}+(v_R^{\varepsilon}+\varepsilon^2f(t,x)e^{-y})\partial_yu_R^{\varepsilon}
+u_R^{\varepsilon}\cdot\na_xu_a^{\varepsilon}\\[3pt]
\qquad \qquad \qquad +(v_R^{\varepsilon}+\varepsilon^2f(t,x)e^{-y})\partial_yu_a^{\varepsilon}
+\na_xp_R^{\varepsilon}-\varepsilon^2
\Delta u_R^{\varepsilon}
=R_h,\\ [3pt] \partial_tv_R^{\varepsilon}+u_a^{\varepsilon}\cdot\na_xv_R^{\varepsilon}+(v_a^{\varepsilon}-\varepsilon^2f(t,x)e^{-y})\partial_yv_R^{\varepsilon}
+u_R^{\varepsilon}\cdot\na_xv_R^{\varepsilon}+(v_R^{\varepsilon}+\varepsilon^2f(t,x)e^{-y})\partial_yv_R^{\varepsilon}
+u_R^{\varepsilon}\cdot\na_xv_a^{\varepsilon}\\[3pt]
 \qquad\qquad\qquad+(v_R^{\varepsilon}+\varepsilon^2f(t,x)e^{-y})\partial_yv_a^{\varepsilon}
+\partial_yp_R^{\varepsilon}-\varepsilon^2
\Delta v_R^{\varepsilon}
=R_v,\\ [3pt]
\na_x\cdot u_R^{\varepsilon}+\partial_yv_R^{\varepsilon}=0,\\[3pt]
(u_R^{\varepsilon}, v_R^{\varepsilon})(t,x,0)=(0, -\varepsilon^2f(t,x)),\\[3pt]
(u_R^{\varepsilon}, v_R^{\varepsilon})(0,x,y)=(0, 0).
\end{array}
\right.
\end{eqnarray}

For simplicity of notations, we will omit the superscript $\varepsilon$ and set $u:=u_R^\varepsilon,~~~v:=v_R^\varepsilon,~~~ p:=p_R^\varepsilon, ~~~u_a:=u_a^\varepsilon,~~~v_a:=v_a^\varepsilon$ and introduce
\begin{align}
&U_a=(u_{a},v_a),\quad \widetilde{U}_a=(u_{a},v_a-\varepsilon^2fe^{-y}),\nonumber\\
&U=(u,v),\quad \widetilde{U}=(u,v-\varepsilon^2fe^{-y}),\quad R=(R_h, R_v),\nonumber
\end{align}
then the error system (\ref{error-equation}) reads
\begin{eqnarray}\label{scaled-error-equation-2}
\left \{
\begin {array}{ll}
\partial_tU-\varepsilon^2
\Delta U+\widetilde{U}_a\cdot\nabla U+\widetilde{U}\cdot\nabla U_a+\widetilde{U}\cdot\nabla U
+\nabla p=R,\\[3pt]
\na_{x,y}\cdot U=0,\\[3pt]
U(t,x,0)=(0,-\varepsilon^2f(t,x)),\\[3pt]
U(0,x,y)=(0,0).
\end{array}
\right.
\end{eqnarray}

\subsection{The vorticity formulation of error equation}

To gain the derivative in $y$ variable, we need to use the vorticity formulation. One can check \cite{KM} for the derivation in 3-D.

Let us first introduce the Dirichlet-Neumann map and Neumann-Dirichlet map.
 \begin{Definition}
Let $f\in C_0^\infty(\R^2)$. We denote by $E_Df$ and $E_Nf$, respectively, the solution to the Dirichlet problem
\begin{align}
\left\{
\begin{array}{ll}
\triangle E_Df=0,\\
E_Df|_{\partial \R^3_+}=f,
\end{array}
\right.\nonumber
\end{align}
and the solution to the Neumman problem
\begin{align}
\left\{
\begin{array}{ll}
\triangle E_Nf=0,\\
-\partial_3E_Nf|_{\partial \R^3_+}=f.
\end{array}
\right.\nonumber
\end{align}
Then the Dirichlet-Neumann map $\Lambda_{DN}$ and Neumann-Dirichlet map $\Lambda_{ND}$ are respectively defined by
\beno
\Lambda_{DN}f=-\gamma\partial_3E_Df,\quad \Lambda_{ND}f=\gamma E_Nf,
\eeno
where $\gamma$ is the trace operator.
\end{Definition}

Next, we introduce
 $$w=(w_h,w_3)={\rm curl}(u,v),\quad w_a=(w_{a,h},w_{a,3})={\rm curl}(u_a,v_a).$$
 Then taking ${\rm curl}$ on both sides of
 (\ref{scaled-error-equation-2}), we arrive at
\begin{eqnarray}
\left \{
\begin {array}{ll}
\partial_tw-\varepsilon^2\triangle w+\widetilde{U}_a\cdot\nabla w+\widetilde{U}\cdot\nabla w_a+\widetilde{U}\cdot\nabla w
-w_{a}\cdot\nabla U-w\cdot\nabla U_a-w\cdot\nabla U\\\quad={\rm curl}(R_h,R_v)-M,\\[3pt]
w(0,x,z)=0,\\[3pt]
-\varepsilon^2(\partial_y+|D_x|)w_h(t,x,0)-\varepsilon^2\partial_x\Lambda_{ND}(\gamma\nabla_x\cdot w_h)(t,x,0)=-\partial_y(-\triangle_D)^{-1}J_h+\partial_x(-\triangle_N)^{-1}J_3 ,\\[3pt]
w_3(t,x,0)=0.
\end{array}
\right.\nonumber
\end{eqnarray}
Here
\begin{align}
M=\left(
\begin{array}{lll}
\varepsilon^2\partial_{x_2}fe^{-y}\partial_yv_a+\varepsilon^2fe^{-y}\partial_yu_{a,2}\\
-\varepsilon^2\partial_{x_1}fe^{-y}\partial_yv_a-\varepsilon^2fe^{-y}\partial_yu_{a,1}\\
\varepsilon^2\partial_{x_1}fe^{-y}\partial_yu_{a,2}-\varepsilon^2\partial_{x_2}fe^{-y}\partial_yu_{a,1}
\end{array}
\right),\nonumber
\end{align}
and
\begin{align}
J={\rm curl}\left(
\begin{array}{lll}
-\widetilde{U}_a\cdot\nabla u-\widetilde{U}\cdot\nabla u_a-\widetilde{U}\cdot\nabla u+R_{h}\\[3pt]
-\widetilde{U}_a\cdot\nabla v-\widetilde{U}\cdot\nabla v_a-\widetilde{U}\cdot\nabla v+R_{v}
\end{array}
\right)\tre (J_h,J_3).\nonumber
\end{align}

In the following, let us explain the derivation of boundary condition of the vorticity.
First, we have the following Biot-Savart law
\beno
u={\rm curl}(\Psi(w)), \quad
\Psi(w)=\Big(
\begin{array}{ll}
(-\triangle_D)^{-1}w_h\\
(-\triangle_N)^{-1}w_3
\end{array}
\Big).
\eeno
Therefore,
\begin{align*}
0=\partial_tu_1|_{y=0}=&\partial_t(\partial_{x_2}(-\triangle_N)^{-1}w_3-\partial_y(-\triangle_D)^{-1}w_2)|_{y=0}\\
=&\partial_{x_2}(-\triangle_N)^{-1}(\varepsilon^2\triangle w_3+J_3)|_{y=0}
-\partial_y(-\triangle_D)^{-1}(\varepsilon^2\triangle w_2+J_2)|_{y=0}\\
=&-\varepsilon^2(\gamma\partial_{x_2}w_3+\partial_{x_2}\Lambda_{ND}\gamma \partial_yw_3)+\partial_{x_2}(-\triangle_N)^{-1}J_3|_{y=0}\\
&+\varepsilon^2(\gamma\partial_yw_2+\Lambda_{DN}\gamma w_2)-\partial_y(-\triangle_D)^{-1}J_2|_{y=0},
\end{align*}
which gives, by $w_3|_{y=0}=0$ and divergence free of vorticity,  that
\ben\label{e:boundary condition of vorticity 1}
-\varepsilon^2(\gamma\partial_yw_2+\Lambda_{DN}\gamma w_2)-\varepsilon^2\partial_{x_2}\Lambda_{ND}\gamma (\nabla_x\cdot w_h)
=\partial_{x_2}(-\triangle_N)^{-1}J_3|_{y=0}
-\partial_y(-\triangle_D)^{-1}J_2|_{y=0}.\nonumber
\een
Similarly, there holds
\begin{align*}
0=\partial_tu_2|_{y=0}=&\partial_t(-\partial_{x_1}(-\triangle_N)^{-1}w_3+\partial_y(-\triangle_D)^{-1}w_1)|_{y=0}\\
=&-\partial_{x_1}(-\triangle_N)^{-1}(\varepsilon^2\triangle w_3+J_3)|_{y=0}
+\partial_y(-\triangle_D)^{-1}(\varepsilon^2\triangle w_1+J_1)|_{y=0}\\
=&\varepsilon^2(\gamma\partial_{x_1}w_3+\partial_{x_1}\Lambda_{ND}\gamma \partial_yw_3)-\partial_{x_1}(-\triangle_N)^{-1}J_3|_{y=0}\\
&-\varepsilon^2(\gamma\partial_yw_1+\Lambda_{DN}\gamma w_1)+\partial_y(-\triangle_D)^{-1}J_1|_{y=0},
\end{align*}
thus
\ben\label{e:boundary condition of vorticity 2}
-\varepsilon^2(\gamma\partial_yw_1+\Lambda_{DN}\gamma w_1)-\varepsilon^2\partial_{x_1}\Lambda_{ND}\gamma (\nabla_x \cdot w_h)=\partial_{x_1}(-\triangle_N)^{-1}J_3|_{y=0}
-\partial_y(-\triangle_D)^{-1}J_1|_{y=0}.\nonumber
\een
Using the fact that
\beno
\Lambda_{DN}f=|D_x|f,
\eeno
we deduce the boundary condition of the vorticity.

\subsection{Decomposition of the vorticity}
Motivated by \cite{M},  we decompose the vorticity into two parts: Euler part $w_e$ and Prandtl part $w_p$, which are respectively defined by the following system
\begin{eqnarray}\label{e:decompose vorticity equation-1}
\left \{
\begin {array}{ll}
\partial_tw_e-\varepsilon^2
\Delta  w_e+\widetilde{U}_{a}\cdot\nabla w_e+\widetilde{U}\cdot\nabla w_{a,e}
+\widetilde{U}\cdot\nabla w_e-w_{a,e}\cdot\nabla U-w_{e}\cdot\nabla U_a-w_{e}\cdot\nabla U\\[3pt]
\quad={\rm curl}(R_{e,h},R_{e,v})-M_e,\\[3pt]
w_{e}(0,x,y)=0,\quad w_{e,3}(t,x,0)=0, \\[3pt]
-\varepsilon^2(\partial_y+|D_x|)w_{e,h}(t,x,0)-\varepsilon^2\partial_x\Lambda_{ND}(\gamma\nabla_x\cdot w_{e,h})(t,x,0)=0,
\end{array}
\right.
\end{eqnarray}
and
\begin{eqnarray}\label{e:decompose vorticity equation-2}
\left \{
\begin {array}{ll}
 \partial_tw_p-\varepsilon^2
\Delta w_p+\widetilde{U}_{a}\cdot\nabla w_p+\widetilde{U}\cdot\nabla w_{a,p}
+\widetilde{U}\cdot\nabla w_p-w_{a,p}\cdot\nabla U-w_{p}\cdot\nabla U_a-w_{p}\cdot\nabla U\\[3pt]
\quad={\rm curl}(R_{p,h},R_{p,v})-M_p,\\[3pt]
w_{p}(0,x,y)=0,\quad w_{p,3}(t,x,0)=0, \\[3pt]
-\varepsilon^2(\partial_y+|D_x|)w_{p,h}(t,x,0)-\varepsilon^2\partial_x\Lambda_{ND}(\gamma\nabla_x\cdot w_{p,h})(t,x,0)\\
\quad=-\partial_y(-\triangle_D)^{-1}J_h+\partial_x(-\triangle_N)^{-1}J_3.
\end{array}
\right.
\end{eqnarray}
Here we denote
$$w_{a,e}=(w_{a,e,h},w_{a,e,3})={\rm curl}(u_{a,e},v_{a,e}),\quad w_{a,p}=(w_{a,p,h},w_{a,p,3})={\rm curl}(u_{a,p},v_{a,p}),$$
and
\begin{align}
M_e=\left(
\begin{array}{lll}
\varepsilon^2\partial_{x_2}fe^{-y}\partial_yv_{a,e}+\varepsilon^2fe^{-y}\partial_yu_{a,e,2}\\
-\varepsilon^2\partial_{x_1}fe^{-y}\partial_yv_{a,e}-\varepsilon^2fe^{-y}\partial_yu_{a,e,1}\\
\varepsilon^2\partial_{x_1}fe^{-y}\partial_yu_{a,e,2}-\varepsilon^2\partial_{x_2}fe^{-y}\partial_yu_{a,e,1}
\end{array}
\right)\nonumber
\end{align}
and
\begin{align}
M_p=\left(
\begin{array}{lll}
\varepsilon^2\partial_{x_2}fe^{-y}\partial_yv_{a,p}+\varepsilon^2fe^{-y}\partial_yu_{a,p,2}\\
-\varepsilon^2\partial_{x_1}fe^{-y}\partial_yv_{a,p}-\varepsilon^2fe^{-y}\partial_yu_{a,p,1}\\
\varepsilon^2\partial_{x_1}fe^{-y}\partial_yu_{a,p,2}-\varepsilon^2\partial_{x_2}fe^{-y}\partial_yu_{a,p,1}
\end{array}
\right).\nonumber
\end{align}
It is easy to find that $w$ and $w_e+w_p$ satisfy the same equation and initial-boundary conditions. Therefore, we get
\ben
w=w_e+w_p.
\een

\section{Functional framework and product estimates}

In this section, we introduce the functional spaces we are working on and some product estimates, which will be used in the energy estimate.

Throughout this paper, let $\delta>0$ be a small constant and $\lambda>0$ be a large constant, which are determined later. We denote by $C_0$ a constant independent of $\delta$, which may change from line to line.

\subsection{Functional framework}

Let $\varphi$ be a smooth function  such that
\begin{align}
\varphi(y)=\left\{
\begin{array}{ll}
\delta y,\quad \ \ y\leq 1,\\
\frac{\delta y}{1+y}, \quad y\geq2.
\end{array}
\right.
\end{align}
We introduce the conormal operator $Z=\varphi(y)\partial_y$ and  denote
$$Z^k\tre \varphi(y)^k\partial_y^k,\quad \widetilde{Z}^k\tre (\delta z)^k\partial_z^k.$$

Let us introduce the following Sobolev type spaces
\begin{align}
H_{co}^m(\mathbb{R}^3_+)\eqdef&\Big\{u\in L^2(\mathbb{R}^3_+): \|u\|_{H_{co}^m}=\sum_{|i|+j\leq m}\|\partial_x^iZ^ju\|_{L^2(\R^3_+)}<\infty\Big\},\nonumber\\
H_{tan}^m(\mathbb{R}^3_+)\eqdef&\Big\{u\in L^2(\mathbb{R}^3_+): \|u\|_{H_{tan}^m}=\sum_{|i|\leq m}\|\partial_x^iu\|_{L^2(\R^3_+)}<\infty\Big\}.\nonumber
\end{align}
Here and in what follows, $\partial_x^iu$ means
$\partial^{i_1}_{x_1}\partial^{i_2}_{x_2}u$ for $i=(i_1,i_2)\in {\rm N}^2.$
$H_{co}^m(\R^3_+)$ is the so called conormal  Sobolev space.

We also introduce the norms
\begin{align}
&\|u\|^2_{H_{co}^m(0,a)}=\sum_{|i|+j\leq m}\int_0^a\int_{\mathbb{R}^2}\big|\partial_x^iZ^ju(x,y)\big|^2dxdy,\nonumber\\
&\|u\|^2_{H_{tan}^m(0,a)}=\sum_{|i|\leq m}\int_0^a\int_{\mathbb{R}^2}\big|\partial_x^iu(x,y)\big|^2dxdy,\nonumber\\
&\|u\|^2_{H_{co}^{m,\frac{1}{2}}(0,a)}=\sum_{|i|+j\leq m}\int_0^a\int_{\mathbb{R}^2}\big|\langle D_x\rangle^{\frac{1}{2}}\partial_x^iZ^ju(x,y)\big|^2dxdy,\nonumber\\
&\|u\|^2_{H_{tan}^{m,\frac{1}{2}}(0,a)}=\sum_{|i|\leq m}\int_0^a\int_{\mathbb{R}^2}\big|\langle D_x\rangle^{\frac{1}{2}}\partial_x^iu(x,y)\big|^2dxdy\nonumber
\end{align}
for some $a\in \R_+$ and the inner product
\begin{align}
&\big<u,v\big>_{H_{co}^m}=\sum_{|i|+j\leq m}\int_0^{+\infty}\int_{\mathbb{R}^2}\partial_x^iZ^ju\partial_x^iZ^jv(x,y)dxdy,\nonumber\\
&\big<u,v\big>_{H_{tan}^m}=\sum_{|i|\leq m}\int_0^{+\infty}\int_{\mathbb{R}^2}\partial_x^iu\partial_x^iv(x,y)dxdy.\nonumber
\end{align}

For the vorticity, we need to make the estimates in the weighted type spaces.  Let $\theta(y)$ be an increasing  function $\theta(y)$  satisfying \begin{align}\label{e:property of cutoff function}
|\theta'(y)|+|\theta{''}(y)|\leq C_0 \delta,\quad
\theta(0)=0,\quad \theta(\frac12)=\delta, \quad\theta'(0)=0,\quad\theta'(y)=0\quad {\rm for} \quad y\geq \frac12.
\end{align}
We define
\begin{align}
\phi(t,y)\eqdef \delta-\theta(y)-\lambda t.\nonumber
\end{align}
Let $y(t)\in (0,\frac12)$ be such that $\phi(t,y(t))=0$ for small $t\le \f \delta\lambda $. Let $T_0=\f \delta {2\lambda}$. Then for $t\in [0,T_0]$, there exists $c_0>0$ so that
\ben\label{eq:critical}
y(t)\ge c_0>0.
\een
We introduce two weights
\begin{align}
\Psi_{e}(t,y)\tre e^{\frac{1}{\varepsilon^{2}}\phi(t,y)},\quad
\Psi_{p}(t,y)\tre e^{\frac{y^2}{\varepsilon^2}(\delta-\lambda t)},\nonumber
\end{align}
where $\Psi_e$ is the weight for Euler part $w_e$ and $\Psi_p$ is the weight for Prandtl part $w_p$.
We denote
$$\|u\|^2_{H^m_e(\R^3_+)}=\sum\limits_{|i|+j\leq m}\Big\|e^{\Psi_e}\partial_x^iZ^ju\Big\|^2_{L^2(\R^3_+)},\quad \|u\|^2_{H^m_p(\R^3_+)}=\sum\limits_{|i|+j\leq m}\Big\|e^{\Psi_p}\partial_x^iZ^ju\Big\|^2_{L^2(\R^3_+)}.$$
Moreover, we denote
\begin{align}
&\|u\|^2_{H^m_e(0,a)}=\sum\limits_{|i|+j\leq m}\Big\|e^{\Psi_e}\partial_x^iZ^ju\Big\|^2_{L^2(\R^2\times (0,a))},\nonumber\\
&\|u\|^2_{H^m_p(0,a)}=\sum\limits_{|i|+j\leq m}\Big\|e^{\Psi_p}\partial_x^iZ^ju\Big\|^2_{L^2(\R^2\times (0,a))},\nonumber\\
&\|u\|^2_{H^{m,\frac{1}{2}}_e(0,a)}=\sum\limits_{|i|+j\leq m}\Big\|e^{\Psi_e}\langle D_x\rangle^{\frac{1}{2}}\partial_x^iZ^ju\Big\|^2_{L^2(\R^2\times (0,a))},\nonumber\\
&\|u\|^2_{H^{m,\frac{1}{2}}_p(0,a)}=
\sum\limits_{|i|+j\leq m}\Big\|e^{\Psi_p}\langle D_x\rangle^{\frac{1}{2}}\partial_x^iZ^ju\Big\|^2_{L^2(\R^2\times (0,a))}.\nonumber
\end{align}

For function $u$ compactly supported in Fourier space in $x$ variable, we define
\begin{align}
u_\Phi(t,x,y)\eqdef \mathcal{F}^{-1}_{\xi \rightarrow x}\big(e^{\Phi(t,\xi, y)}\mathcal{F}_{x\rightarrow \xi}u\big)(t,x,y),\nonumber
\end{align}
where $\Phi(t,\xi,y)= \phi(t,y)\langle \xi\rangle$.

\subsection{Product estimates}

Let $a\in [0,\f12]$ and $I=[0,a]$. 

\begin{Lemma}\label{e:product estimate}
Let $m\geq 8$ and $\sigma\in [0,1]$. It holds that
\begin{align}
\|\big<D_x\big>^{-\sigma}(uv)_\Phi\|^2_{H_{co}^m(I)}\leq C\Big(&\sum_{|i|+j\leq m}\big\|\partial_x^iZ^j\langle D_x\rangle^{-\sigma}u_\Phi\big\|^2_{L^\infty_y(I;L^2(\R^2))}\|v_\Phi\|^2_{H_{co}^{m-2}(I)}\nonumber\\
\quad&+\sum_{|i|+j\leq m-2}\big\|\partial_x^iZ^ju_\Phi\big\|^2_{L^\infty_y(I;L^2(\R^2))}\|\langle D_x\rangle^{-\sigma}v_\Phi\|^2_{H_{co}^m(I)}\Big),\nonumber\\
\|\big<D_x\big>^{-\sigma}(uv)_\Phi\|^2_{H_{co}^m(I)}\leq C\Big(&\sum_{|i|+j\leq m-2}\big\|\partial_x^iZ^ju_\Phi\big\|^2_{L^\infty_y(I;L^2(\R^2))}\|\langle D_x\rangle^{-\sigma}v_\Phi\|^2_{H_{co}^m(I)}\nonumber\\
\quad&+\sum_{|i|+j\leq m-2}\big\|\partial_x^iZ^jv_\Phi\big\|^2_{L^\infty_y(I;L^2(\R^2))}\|\langle D_x\rangle^{-\sigma}u_\Phi\|^2_{H_{co}^m(I)}\Big).\nonumber
\end{align}
In particular, we have
\begin{align}
\|(uv)_\Phi\|^2_{H_{co}^m(I)}\leq C\Big( &\big(\|u_\Phi\|^2_{H_{co}^{m-1}(I)}+\|(\partial_yu)_\Phi\big\|^2_{H_{co}^{m-2}(I)}\big)\|v_\Phi\|^2_{H_{co}^m(I)}\nonumber\\
&+\big(\|v_\Phi\|^2_{H_{co}^{m-1}(I)}+\|(\partial_yv)_\Phi\big\|^2_{H_{co}^{m-2}(I)}\big)\|u_\Phi\|^2_{H_{co}^m(I)}\Big).\nonumber
\end{align}
Similar estimates also hold in the space $H^m_{tan}(I)$.
\end{Lemma}

\begin{proof}
We deduce from the definition of $H_{co}^m(I)$ that
\begin{align}
\|\langle D_x\rangle^{-\sigma}(uv)_\Phi\|^2_{H_{co}^m(I)}=\sum_{|i|+j\leq m}\int_0^a\int_{\mathbb{R}^2}\big|\langle D_x\rangle^{-\sigma}\partial_x^iZ^j(uv)_\Phi\big|^2(x,y)dxdy.\nonumber
\end{align}

We only show the case of $j=m$, other cases are similar. By Leibniz's rule, we have
\beno
&&\int_0^a\int_{R^2}\big|\langle D_x\rangle^{-\sigma}Z^m(uv)_\Phi\big|^2dxdy\\
&&\leq C\sum_{m_2+m_3=m-m_1}\int_0^a\big\|\langle D_x\rangle^{-\sigma}(Z^{m_2}uZ^{m_3}v)_\Phi\big\|^2_{H^{m_1}(\R^2)}dy\tre I.
\eeno
We split it into two cases according to the value of $m_1$.
\smallskip

{\bf Case 1. $m_1\geq 4$.}

By classical product estimate(\cite{BCD}), we have
\begin{align*}
I\leq& C\Big(\sum_{m_2+m_3=m-m_1}\int_0^a\big\|\langle D_x\rangle^{-\sigma}(Z^{m_2}u)_\Phi\big\|^2_{H^{m_1}(\R^2)}\|(Z^{m_3}v)_\Phi\|^2_{H^2(\R^2)}dy\\
&\quad+\sum_{m_2+m_3=m-m_1}\int_0^a\big\|(Z^{m_2}u)_\Phi\big\|^2_{H^{2}(\R^2)}\|\langle D_x\rangle^{-\sigma}(Z^{m_3}v)_\Phi\|^2_{H^{m_1}(\R^2)}dy\Big)\\
\leq& C\Big(\sum_{|i|+j\leq m}\big\|\partial_x^iZ^j\langle D_x\rangle^{-\sigma}u_\Phi\big\|^2_{L^\infty(I;L^2(\R^2))}\|v_\Phi\|^2_{H_{co}^{m-2}(I)}\nonumber\\
&\qquad+\sum_{|i|+j\leq m-2}\big\|\partial_x^iZ^ju_\Phi\big\|^2_{L^\infty(I;L^2(\R^2))}\|\langle D_x\rangle^{-\sigma}v_\Phi\|^2_{H_{co}^m(I)}\Big).
\end{align*}

{\bf Case 2. $m_1\leq 3$.}

First of all, we have the following classical product estimate: for any $s\ge 0$,
\ben\label{e:esimate in sobolev space}
&&\big\|\langle D_x\rangle^{-\sigma}(uv)\big\|_{H^s(\mathbb{R}^2)}\leq C\big\|\langle D_x\rangle^{-\sigma}u\big\|_{H^{s}(\mathbb{R}^2)}\|v\big\|_{H^{s+2}(\mathbb{R}^2)}.
\een
We decompose $I$ into three parts as following
\begin{align*}
I\leq &C\sum_{m_2\leq\big[\frac{m-m_1}{2}\big]-1}\int_0^a\big\|\langle D_x\rangle^{-\sigma}(Z^{m_2}uZ^{m_3}v)_\Phi\big\|^2_{H^{m_1}(\R^2)}dy\\
&+C\sum_{m_3\leq \big[\frac{m-m_1}{2}\big]-1}\int_0^a\big\|\langle D_x\rangle^{-\sigma}(Z^{m_2}uZ^{m_3}v)_\Phi\big\|^2_{H^{m_1}(\R^2)}dy\\
&+C\int_0^a\big\|\langle D_x\rangle^{-\sigma}(Z^{[\frac{m-m_1}{2}]}uZ^{m-m_1-[\frac{m-m_1}{2}]}v)_\Phi\big\|^2_{H^{m_1}(\R^2)}dy\tre I_1+I_2+I_3.
\end{align*}
It follows from (\ref{e:esimate in sobolev space}) that
\begin{align*}
I_1\leq& C\sum_{m_2\leq\big[\frac{m-m_1}{2}\big]-1}\int_0^a\big\|(Z^{m_2}u)_\Phi\big\|^2_{H^{m_1+2}(\R^2)}\|\langle D_x\rangle^{-\sigma}(Z^{m_3}v)_\Phi\|^2_{H^{m_1}(\R^2)}dy\\
\leq& C\sum_{|i|+j\leq m-2}\big\|\partial_x^iZ^ju_\Phi\big\|^2_{L^\infty(I;L^2(\R^2))}\big\|\langle D_x\rangle^{-\sigma}v_\Phi\big\|^2_{H_{co}^m(I)}.
\end{align*}
Similar argument gives
\beno
I_2\leq C\sum_{|i|+j\leq m}\big\|\partial_x^iZ^j\langle D_x\rangle^{-\sigma}u_\Phi\big\|^2_{L^\infty(I;L^2(\R^2))}\big\|v_\Phi\big\|^2_{H_{co}^{m-2}(I)}.
\eeno
For $I_3$, we have
\begin{align*}
I_3\leq& C\int_0^a\Big\|(Z^{[\frac{m-m_1}{2}]}u)_\Phi\Big\|^2_{L^\infty(\R^2)}\Big\| \langle D_x\rangle^{-\sigma}(Z^{m-m_1-[\frac{m-m_1}{2}]}v)_\Phi\Big\|^2_{H^{m_1}(\R^2)}dy\\
&\quad+ C\int_0^a\Big\|\langle D_x\rangle^{-\sigma}(Z^{[\frac{m-m_1}{2}]}u)_\Phi\Big\|^2_{H^{m_1}(\R^2)}\Big\| (Z^{m-m_1-[\frac{m-m_1}{2}]}v)_\Phi\Big\|^2_{L^\infty(\R^2)}dy\\
\leq& C\Big(\sum_{|i|+j\leq m}\big\|\partial_x^iZ^j\langle D_x\rangle^{-\sigma}u_\Phi\big\|^2_{L^\infty(I;L^2(\R^2))}\|v_\Phi\|^2_{H_{co}^{m-2}(I)}\\
&\qquad+\sum_{|i|+j\leq m-2}\big\|\partial_x^iZ^ju_\Phi\big\|^2_{L^\infty(I;L^2(\R^2))}\|\langle D_x\rangle^{-\sigma}v_\Phi\|^2_{H_{co}^m(I)}\Big).
\end{align*}
The first inequality follows by collecting these estimates. The proof of the second inequality is similar. The third inequality can be deduced from Sobolev embedding and the second one directly.
\end{proof}

\smallskip
To deal with the term like  $u\partial_xv$, we need the following lemma.
\begin{Lemma}\label{e:double linear estimate}
Let $m\geq 8$. It  holds that
\begin{align}
\Big|\big<(u\partial_xv)_\Phi, w_\Phi\big>_{H_{co}^m(I)}\Big|\leq& C\Big(\sum_{|i|+j\leq m}\big\|\partial_x^iZ^j\langle D_x\rangle^{-\frac{1}{2}}u_\Phi\big\|^2_{L^\infty_y(I;L^2(\R^2))}\nonumber\\
&+\sum_{|i|+j\leq m-2}\big\|\partial_x^iZ^ju_\Phi\big\|^2_{L^\infty_y(I;L^2(\R^2))}\Big)\|v_\Phi\|^2_{H_{co}^{m,\frac{1}{2}}(I)}+ C\|w_\Phi\|^2_{H_{co}^{m,\frac{1}{2}}(I)},\nonumber\\
\Big|\big<(u\partial_xv)_\Phi, w_\Phi\big>_{H_{co}^m(I)}\Big|\leq& C\|u_\Phi\|^2_{H_{co}^{m}(I)}\Big(\sum_{|i|+j\leq m}\big\|\partial_x^iZ^j\langle D_x\rangle^{\frac{1}{2}}v_\Phi\big\|^2_{L^\infty_y(I;L^2(\R^2))}\nonumber\\
&+\sum_{|i|+j\leq m-1}\big\|\partial_x^iZ^jv_\Phi\big\|^2_{L^\infty_y(I; L^2(\R^2))}\Big)+ C\|w_\Phi\|^2_{H_{co}^{m,\frac{1}{2}}(I)},\nonumber\\
\Big|\big<(u\partial_xv)_\Phi, w_\Phi\big>_{H_{co}^m(I)}\Big|\leq& C\|u_\Phi\|^2_{H_{co}^{m}(I)}\Big(\|v_\Phi\|^2_{H_{co}^{m}(I)}+\|(\partial_yv)_\Phi\|^2_{H_{co}^{m-1}(I)}\Big)\nonumber\\
&+C\|v_\Phi\|^2_{H_{co}^{m,\frac{1}{2}}(I)}\Big(\|u_\Phi\|^2_{H_{co}^{m-1}(I)}+\|(\partial_yu)_\Phi\|^2_{H_{co}^{m-2}(I)}\Big)
+ C\|w_\Phi\|^2_{H_{co}^{m,\frac{1}{2}}(I)}.\nonumber
\end{align}
Similar estimates also hold in the space $H^m_{tan}(I)$.
\end{Lemma}

\begin{proof}
Using the first inequality of Lemma \ref{e:product estimate}, we deduce that
\begin{align}
\Big|\big<(u\partial_xv)_\Phi, w_\Phi\big>_{H_{co}^m(I)}\Big|=&\Big|\big<\langle D_x\rangle^{-\frac{1}{2}}(u\partial_xv)_\Phi,\langle D_x\rangle^{\frac{1}{2}} w_\Phi\big>_{H_{co}^m(I)}\Big|\nonumber\\
\leq& C\|\langle D_x\rangle^{-\frac{1}{2}}(u\partial_xv)_\Phi\|^2_{H_{co}^m(I)}+C\|w_\Phi\|^2_{H_{co}^{m,\frac{1}{2}}(I)}\nonumber\\
\leq& C\Big(\sum_{|i|+j\leq m}\big\|\partial_x^iZ^j\big<D_x\big>^{-\frac{1}{2}}u_\Phi\big\|^2_{L^\infty_y(I;L^2_x(\R^2))}\nonumber\\
&+\sum_{|i|+j\leq m-2}\big\|\partial_x^iZ^ju_\Phi\big\|^2_{L^\infty_y(I; L^2_x(\R^2))}\Big)\|v_\Phi\|^2_{H_{co}^{m,\frac{1}{2}}(I)}+ C\|w_\Phi\|^2_{H_{co}^{m,\frac{1}{2}}(I)},\nonumber
\end{align}
which shows the first inequality. The second inequality can be proved in a similar way. By the second inequality of Lemma \ref{e:product estimate} and Sobolev embedding, we can deduce the third inequality.
\end{proof}

\medskip

The following two lemmas are the analogous of Lemma \ref{e:product estimate} and Lemma \ref{e:double linear estimate} in the weighted spaces. Since the proof is the almost same, we omit the details.

\begin{Lemma}\label{e:weighted estimate product}
Let $m\geq 8$ and $\sigma\in [0,1]$. It holds that
\begin{align}
\|\big<D_x\big>^{-\sigma}(uv)_\Phi\|^2_{H_e^m(I)}\leq C\Big(&\sum_{|i|+j\leq m}\big\|\partial_x^iZ^j\langle D_x\rangle^{-\sigma}u_\Phi\big\|^2_{L^\infty_y(I; L^2(\R^2))}\|v_\Phi\|^2_{H_e^{m-2}(I)}\nonumber\\
&+\sum_{|i|+j\leq m-2}\big\|\partial_x^iZ^ju_\Phi\big\|^2_{L^\infty_y(I;L^2(\R^2))}\|\langle D_x\rangle^{-\sigma}v_\Phi\|^2_{H_e^m(I)}\Big),\nonumber\\
\langle D_x\rangle^{-\sigma}(uv)_\Phi\|^2_{H_e^m(I)}\leq C\Big(&\sum_{|i|+j\leq m-2}\big\|\partial_x^iZ^ju_\Phi\big\|^2_{L^\infty_y(I; L^2(\R^2))}\|\big<D_x\big>^{-\sigma}v_\Phi\|^2_{H_e^m(I)}\nonumber\\
&+\sum_{|i|+j\leq m-2}\big\|e^{\Psi_e}\partial_x^iZ^jv_\Phi\big\|^2_{L^\infty_y(I; L^2(\R^2))}\|\langle D_x\rangle^{-\sigma}u_\Phi\|^2_{H_{co}^m(I)}\Big).\nonumber
\end{align}
Similar estimates also hold in the weighted space $H^m_p(I)$.
\end{Lemma}

\begin{Lemma}\label{e:weighted double linear estimate}
Let $m\geq 8$. It holds that
\begin{align}
\Big|\big<(u\partial_xv)_\Phi, w_\Phi\big>_{H_e^m(I)}\Big|\leq& C\Big(\sum_{|i|+j\leq m}\big\|\partial_x^iZ^j\langle D_x\rangle^{-\frac{1}{2}}u_\Phi\big\|^2_{L^\infty_y(I;L^2(\R^2))}\nonumber\\
&+\sum_{|i|+j\leq m-2}\big\|\partial_x^iZ^ju_\Phi\big\|^2_{L^\infty_y(I;L^2(\R^2))}\Big)\|v_\Phi\|^2_{H_e^{m,\frac{1}{2}}(I)}
+ C\|w_\Phi\|^2_{H_e^{m,\frac{1}{2}}(I)},\nonumber\\
\Big|\big<(u\partial_xv)_\Phi, w_\Phi\big>_{H_e^m(I)}\Big|\leq& C\|u_\Phi\|^2_{H_e^{m}(I)}\Big(\sum_{|i|+j\leq m}\big\|\partial_x^iZ^j\langle D_x\rangle^{\frac{1}{2}}v_\Phi\big\|^2_{L^\infty_y(I;L^2(\R^2))}\nonumber\\
&+\sum_{|i|+j\leq m-1}\big\|\partial_x^iZ^jv_\Phi\big\|^2_{L^\infty_y(I; L^2(\R^2))}\Big)+ C\|w_\Phi\|^2_{H_e^{m,\frac{1}{2}}(I)}.\nonumber
\end{align}
Similar estimates also hold in the weighted space $H^m_p(I)$.
\end{Lemma}

\section{Energy functional and Proof of Theorem \ref{thm:main}}

\subsection{Construction of energy functional}

To control the error, let us introduce the following energy functional
\begin{align}
&E_v(t)\eqdef \varepsilon^{-2}\Big(\big\|U_\Phi\big\|^2_{H^9_{tan}}+\big\|U\big\|^2_{H^{10}_{tan}}\Big),\nonumber\\
&K_v(t)\eqdef \varepsilon^{-2}\big\|U_\Phi\big\|^2_{H^{9,\frac12}_{tan}(0,y(t))},\nonumber\\
&E_w(t)\eqdef \varepsilon^{-2}\Big(\big\|(\varphi w_{e})_\Phi\big\|^2_{H_e^8}+\big\|(\varphi w_p)_\Phi\big\|^2_{H_p^8}+\big\|(w_{e,3})_\Phi\big\|^2_{H_e^8}+\big\|(w_{p,3})_\Phi\big\|^2_{H_p^8}+\nonumber\\
&\qquad\qquad\quad\quad+\big\|\varphi w_e\big\|^2_{H^9_{co}}+\big\|\varphi w_p\big\|^2_{H_p^9}+\big\|w_{e,3}\big\|^2_{H^9_{co}}+\big\|w_{p,3}\big\|^2_{H_p^9}\Big)\nonumber\\
&\qquad\qquad+\big\|(w_e)_\Phi\big\|^2_{H_e^8}+\big\|(w_p)_\Phi\big\|^2_{H_p^8}+\big\|w_{e}\big\|^2_{H_{co}^9}+\big\|w_{p}\big\|^2_{H_p^9},\nonumber\\
&K_w(t)\eqdef \varepsilon^{-2}\Big(\big\|(\varphi w_{e})_\Phi\big\|^2_{H_e^{8,\frac12}(0,y(t))}+\big\|(\varphi w_p)_\Phi\big\|^2_{H_p^{8,\frac12}(0,y(t))}
+\big\|(w_{e,3})_\Phi\big\|^2_{H_e^{8,\frac12}(0,y(t))}\nonumber\\
&\qquad\qquad\quad+\big\|(w_{p,3})_\Phi\big\|^2_{H_p^{8,\frac12}(0,y(t))}\Big)+\big\|(w_{e})_\Phi\big\|^2_{H_e^{8,\frac12}(0,y(t))}
+\big\|(w_{p})_\Phi\big\|^2_{H_p^{8,\frac12}(0,y(t))}.\nonumber
\end{align}
We denote
\beno
E(t)\tre E_v(t)+E_w(t),\quad K(t)\tre K_v(t)+K_w(t).
\eeno

In the following Section 7-Section 12, we will show that if $E(t)\le C_1\varepsilon^2$, then it holds that
\ben\label{eq:energy}
\f d {dt}E(t)+(\lambda-C)K(t)\le C\varepsilon^2
\een
under the following uniform estimates for the approximate solutions and the remainders $R_{e,h}, R_{e,v}, R_{p,h}, R_{p,v}$. The proof will be presented in appendix.

\begin{Lemma}\label{e:uniform boundness for approximate solution}
There exist $T_a> 0$ such that for any $t\in [0,T_a]$, there holds
\begin{align}
&\big\|(u_e^{(i)},v_e^{(i)})_{\Phi_e}\big\|_{H^{15}(\R^3_+)}\leq C,\nonumber\\
&\sum_{|m|+n\leq 15}\Big(\sum_{l\leq 1}\big\|e^{z^2}\partial_x^m\widetilde{Z}^n\partial_t^l(u_p^{(i)},v_p^{(i+1)})_{\Phi_p}\big\|_{L^2(\R^3_+)}
+\big\|e^{z^2}\partial_x^m\widetilde{Z}^n\partial_z^2(u_p^{(i)},v_p^{(i+1)})_{\Phi_p}\big\|_{L^2(\R^3_+)}\Big)\leq C\nonumber
\end{align}
for $i=0,1$, where $\Phi_e=(1-y)_+\langle\xi\rangle$ and $\Phi_p=(\f12-\lambda_pt)\langle\xi\rangle$.
\end{Lemma}

\begin{Lemma}\label{e:uniform boundness for error}
There exist $T_a> 0$ such that for any $t\in [0,T_a]$, there holds
\begin{align}
&\big\|(R_{e,h},R_{e,v})_{\Phi_e}\big\|_{H^{11}(\R^3_+)}\leq C\varepsilon^2,\nonumber\\
&\sum_{|m|+n\leq 10}\big\|e^{z^2}\partial_x^m\widetilde{Z}^n(R_{p,h},R_{p,v})_{\Phi_p}\big\|_{L^2(\R^3_+)}\leq C\varepsilon^2.\nonumber
\end{align}
\end{Lemma}

\subsection{Proof of Theorem \ref{thm:main}}

Let us  prove Theorem \ref{thm:main}  under the energy inequality \eqref{eq:energy}. \smallskip

For fixed $\varepsilon>0$, the local well-posedness of the Navier-Stokes equations can be easily proved in the energy space like
\beno
\mathcal{E}(t)=\big\|U_\Phi^\varepsilon\big\|^2_{H^{10}_{tan}}+\big\|U^\varepsilon\big\|_{H^{10}_{tan}}+
\big\|(w_{e}^\varepsilon)_\Phi\big\|^2_{H_{co}^9}+\big\|(w_{p}^\varepsilon)_\Phi\big\|^2_{H_p^9}+\big\|w_{e}^\varepsilon\big\|^2_{H_{co}^9}+\big\|w_{p}^\varepsilon\big\|^2_{H_p^9},
\eeno
where $w^\varepsilon=\text{curl} U^\varepsilon=w_e^\varepsilon+w_p^\epsilon$ with
\begin{eqnarray*}
\left \{
\begin {array}{ll}
\partial_tw_e^\varepsilon-\varepsilon^2
\Delta  w_e^\varepsilon+{U^\varepsilon}\cdot\nabla w_e^\varepsilon
-w_{e}^\varepsilon\cdot\nabla U^\varepsilon=0,\\[3pt]
w_{e}^\varepsilon(0,x,y)=w_0,\quad w_{e,3}^\varepsilon(t,x,0)=0, \\[3pt]
-\varepsilon^2(\partial_y+|D_x|)w_{e,h}^\varepsilon(t,x,0)-\varepsilon^2\na_x\Lambda_{ND}(\gamma\nabla_x\cdot w_{e,h}^\varepsilon)(t,x,0)=0,
\end{array}
\right.
\end{eqnarray*}
and
\begin{eqnarray*}\left \{
\begin {array}{ll}
 \partial_tw_p^\varepsilon-\varepsilon^2
\Delta w_p^\varepsilon+{U^\varepsilon}\cdot\nabla w_p^\varepsilon
-w_{p}^\varepsilon\cdot\nabla U^\varepsilon=0,\\[3pt]
w_{p}^\varepsilon(0,x,y)=0,\quad w_{p,3}^\varepsilon(t,x,0)=0, \\[3pt]
-\varepsilon^2(\partial_y+|D_x|)w_{p,h}^\varepsilon(t,x,0)-\varepsilon^2\na_x\Lambda_{ND}(\gamma\nabla_x\cdot w_{p,h}^\varepsilon)(t,x,0)\\
\quad=-\partial_y(-\triangle_D)^{-1}J_h^\varepsilon+\partial_x(-\triangle_N)^{-1}J_3^\varepsilon.
\end{array}
\right.
\end{eqnarray*}
Here $(J_h^\varepsilon, J_3^\varepsilon)=\text{curl}(U^\varepsilon\cdot\na U^\varepsilon)$.

Let $T^*_\varepsilon$ be the maximal existence time of the solution $U^\varepsilon$.
If we take $\lambda\ge C$ and $T_1\le T_a$ so that $T_1C\le \f {C_1} 2$, we deduce from \eqref{eq:energy} that
\beno
E(t)\le \f {C_1} 2\varepsilon^2 \quad\text{for}\quad t\in \big[0,\min(T^*_\varepsilon, T_1)\big],
\eeno
which in turn implies $T^*_\varepsilon=T_1$ by a continuous argument. Therefore, there holds
\begin{align}
\sup_{0\leq t\leq T_0}\Big(\varepsilon^{-2}\big\|U^\varepsilon(t)-U_{a}(t)\big\|^2_{H^{10}_{tan}}+\big\|w_e^\varepsilon(t)-w_{e,a}(t)\big\|^2_{H_{co}^8}+\big\|w_p^\varepsilon(t)-w_{p,a}(t)\big\|^2_{H_p^8}\Big)\leq C\varepsilon^2.\nonumber
\end{align}
Then Sobolev embedding gives
$$
\|U^\varepsilon(t)-U_a(t)\|_{L^\infty(\R^3_+)}\leq C\varepsilon^\f32 \quad\text{for}\quad t\in [0,T_1]
$$
with $C$ independent of $\varepsilon$. The proof of Theorem \ref{thm:main} is completed.
\medskip

Let us conclude this section by the following lemma, which will be used in the energy estimate of the vorticity.

\begin{Lemma}\label{e:one order elliptic estimate}
Let $w$ solve the equation
\begin{eqnarray}
\left \{
\begin {array}{ll}
\partial_yw  +|D_x|w=f,\\ [3pt]
\displaystyle\lim_{y\rightarrow+\infty}w(x,y)=0.
\end{array}
\right.\nonumber
\end{eqnarray}
Then we have
\begin{align}
\||D_x|w\|^2_{L^2_{x,y}}+\|\partial_yw\|^{2}_{L^2_{x,y}}\leq C\|f\|^{2}_{L^2_{x,y}}.\nonumber
\end{align}
\end{Lemma}

\begin{proof}
Taking Fourier transform in $x$ variable, we get
\begin{eqnarray}
\partial_y\widehat{w} +|\xi|\widehat{w}=\widehat{f}.
\end{eqnarray}
Solving this ODE, we get
$$\widehat{w}(\xi,y)=-\int_{y}^{+\infty}e^{-|\xi|(y-y')}\widehat{f}(\xi,y')dy^{'},$$
from which, it follows that
\begin{align}
\||D_x|w\|_{L^{2}_{x,y}}\leq\||\xi|e^{-|\xi|y}\ast_y\widehat{f}(\xi,\cdot)\|_{L^2_xL^2_y}
\leq \sup_{\xi}\Big(|\xi|\|e^{-|\xi|y}\|_{L^{1}_y}\Big{)}\|f\|_{L^2_xL^2_y}
\leq\|f\|_{L^2_xL^2_y}.\nonumber
\end{align}
Using the equation, we can obtain the estimate of $\|\partial_yw\|^{2}_{L^2_{x,y}}$.
\end{proof}

\section{Tangential analytic and Sobolev estimates of the pressure}

\subsection{Elliptic equation of the pressure}

Taking ${\rm div}$ on both sides of the system (\ref{scaled-error-equation-2}), we obtain the following elliptic equation on the pressure $p$ with Neumann boundary condition
\begin{eqnarray}\label{e:equation for pressure}
\left \{
\begin {array}{ll}
-\triangle p=({\rm div}_xF+\partial_yG)-({\rm div}_xR_{h}+\partial_yR_{v}),\\[3pt]
\partial_yp(x,0)=\varepsilon^2\partial_{yy}v(t,x,0)+R_{v}(x,0)+\varepsilon^2\partial_tf(t,x)-\varepsilon^4\Delta_{x}f(t,x),
\end{array}
\right.
\end{eqnarray}
where
\begin{align}\label{q:some new quantity}
F\tre \widetilde{U}_a\cdot\nabla u+\widetilde{U}\cdot\nabla u_a+\widetilde{U}\cdot\nabla u,\quad
G\tre\widetilde{U}_a\cdot\nabla v+\widetilde{U}\cdot\nabla v_a+\widetilde{U}\cdot\nabla v.
\end{align}
It is easy to see that $F(t,x,0)=0$ and $G(t,x,0)=0$.

To proceed, let us present the estimate of $(F,G)$.

\begin{Lemma}
\label{lem:F-analytic}
It holds that
\begin{align*}
\big\|( F,G)_{\Phi}\big\|^2_{H^{8,\frac{1}{2}}_{tan}(0,y(t))}
\leq C\varepsilon^2\big(1+E(t)\big)\big(E(t)+K(t)+\varepsilon^2\big).\end{align*}
\end{Lemma}
\begin{proof}
We first handle $F$ and estimate it term by term. By Lemma \ref{e:product estimate},  Lemma \ref{e:uniform boundness for approximate solution}, $\tilde{v}_a|_{y=0}=0$ and $\partial_yu=(w_2,-w_1)+\na_xv\tre w_h^{\bot}+\na_xv$, we deduce that
\begin{align*}
\big\|(\widetilde{U}_a\cdot\nabla u)_\Phi\big\|^2_{H^{8,\frac{1}{2}}_{tan}(0,y(t))}\leq&\big\|(u_a\partial_x u)_\Phi\big\|^2_{H^{8,\frac{1}{2}}_{tan}(0,y(t))}+
\big\|(\tilde{v}_a\partial_y u)_\Phi\big\|^2_{H^{8,\frac{1}{2}}_{tan}(0,y(t))}\\
\leq&C\big\| u_\Phi\big\|^2_{H^{9,\frac{1}{2}}_{tan}(0,y(t))}+
\Big\|\Big(\frac{\tilde{v}_a}{\varphi}(\varphi w^{\perp}_h+\varphi\na_xv)\Big)_\Phi\Big\|^2_{H^{8,\frac{1}{2}}_{tan}(0,y(t))}\\
\leq&C\Big(\|U_\Phi\|^2_{H_{tan}^{9,\frac{1}{2}}(0,y(t))}+\|(\varphi w)_\Phi\|^2_{H_{co}^{8,\frac{1}{2}}(0,y(t))}+\varepsilon^4\Big)\\
\leq& C\varepsilon^2(K(t)+\varepsilon^2).
\end{align*}
Similarly, using $\tilde{v}|_{y=0}=0$ and $-\pa_yv=\na_x\cdot u$, we get
\begin{align*}
\big\|(\widetilde{U}\cdot\nabla u_a)_\Phi\big\|^2_{H^{8,\frac{1}{2}}_{tan}(0,y(t))}\leq&\big\|(u\partial_x u_a)_\Phi\big\|^2_{H^{8,\frac{1}{2}}_{tan}(0,y(t))}+
\big\|(\tilde{v}\partial_y u_a)_\Phi\big\|^2_{H^{8,\frac{1}{2}}_{tan}(0,y(t))}\\
\leq&C\big\| u_\Phi\big\|^2_{H^{8,\frac{1}{2}}_{tan}(0,y(t))}+
\Big\|\Big(\frac{\tilde{v}}{\varphi}Zu_a\Big)_\Phi\Big\|^2_{H^{8,\frac{1}{2}}_{tan}(0,y(t))}\\
\leq&C\big(\|U_\Phi\|^2_{H_{tan}^{9,\frac{1}{2}}(0,y(t))}+\varepsilon^4\big)\leq C\varepsilon^2\big(K(t)+\varepsilon^2\big).
\end{align*}

To deal with the nonlinear term, we need the following product estimate
\ben\label{e:product estimate classical}
\|\langle D_x\rangle^{\frac{1}{2}}(uv)_\Phi(\cdot,y)\|^2_{L^2(\R^2)}\leq C\|\langle D_x\rangle^{\frac{1}{2}}u_\Phi(\cdot,y)\|^2_{L^2(\R^2)}\|v_\Phi(\cdot,y)\|^2_{H^2(\R^2)} \quad\text{for}\quad y\in (0,y(t)).\nonumber
\een
Then by Sobolev embedding, $\partial_yu=w_h^{\bot}+\na_xv$ and $-\pa_yv=\na_x\cdot u$, we get
\begin{align*}
\big\|(u\partial_x u)_\Phi\big\|^2_{H^{8,\frac{1}{2}}_{tan}(0,y(t))}
\leq&C\sum_{|i|\leq 8, |j|\leq 4,}\int_0^{y(t)}\big\|\langle D_x\rangle^{\frac{1}{2}}(\partial_x^ju\partial_x \partial_x^{i-j}u)_\Phi(\cdot,y)\big\|^2_{L^2(\R^2)}dy\\
&+
C\sum_{|i|\leq 8, |j|> 4}\int_0^{y(t)}\big\|\langle D_x\rangle^{\frac{1}{2}}(\partial_x^ju\partial_x \partial_x^{i-j}u)_\Phi(\cdot,y)\big\|^2_{L^2(\R^2)}dy\\
\leq&C \|u_\Phi\|_{H^6_{tan}}\big(\|(\partial_yu)_\Phi\|_{H^6_{tan}}+\|u_\Phi\|_{H^7_{tan}}\big)\|u_\Phi\|^2_{H_{tan}^{9,\frac{1}{2}}(0,y(t))}\\
\leq&C \big(\|w_\Phi\|^2_{H^8_{co}}+\|U_\Phi\|^2_{H^9_{tan}}\big)\|u_\Phi\|^2_{H_{tan}^{9,\frac{1}{2}}(0,y(t))}\leq C\varepsilon^2E(t)K(t),
\end{align*}
and
\begin{align}
\big\|(\tilde{v}\partial_y u)_\Phi\big\|^2_{H^{8,\frac{1}{2}}_{tan}(0,y(t))}
\leq&C\sum_{|i|\leq 8, |j|\leq 4}\int_0^{y(t)}\big\|\langle D_x\rangle^{\frac{1}{2}}(\partial_x^j\tilde{v}\partial_y \partial_x^{i-j}u)_\Phi(\cdot,y)\big\|^2_{L^2(\R^2)}dy\nonumber\\
&+
C\sum_{|i|\leq 8, |j|> 4}\int_0^{y(t)}\big\|\langle D_x\rangle^{\frac{1}{2}}(\partial_x^j\tilde{v}\partial_y \partial_x^{i-j}u)_\Phi(\cdot,y)\big\|^2_{L^2(\R^2)}dy\nonumber\\
\leq&C \|\tilde{v}_\Phi\|_{H^6_{tan}}\big(\|(\partial_y\tilde{v})_\Phi\|_{H^6_{tan}}+\|\tilde{v}_\Phi\|_{H^7_{tan}}\big)
\Big(\|w_\Phi\|^2_{H_{co}^{8,\frac{1}{2}}(0,y(t))}
+\|v_\Phi\|^2_{H_{tan}^{9,\frac{1}{2}}(0,y(t))}\Big)\nonumber\\
&+C \|w_\Phi+\partial_xv_\Phi\|^2_{H^6_{tan}}
\Big(\|U_\Phi\|^2_{H_{tan}^{9,\f12}(0,y(t))}+\varepsilon^4\Big)\nonumber\\
\leq&C \varepsilon^2E(t)\big(\varepsilon^2+K(t)\big).\nonumber
\end{align}

Summing up, we deduce that
\beno
\big\|F_{\Phi}\big\|^2_{H^{8,\frac{1}{2}}_{tan}(0,y(t))}
\leq C\varepsilon^2\big(1+E(t)\big)\big(E(t)+K(t)+\varepsilon^2\big).
\eeno
The estimate for $G$ is similar.
\end{proof}

\begin{Lemma}\label{lem:F-Sob}
It holds that for $k=7,8,9$,
\beno
\|(F,G)\|^2_{H^k_{tan}}\le C \varepsilon^2\big(\varepsilon^2+E(t)\big)(E(t)+1).
\eeno
\end{Lemma}
\begin{proof}
We only prove the case of $k=9$ for $F$. By Lemma \ref{e:product estimate}  and Lemma \ref{e:uniform boundness for approximate solution}, $\tilde{v}_a|_{y=0}=0$, $\partial_yu=w_h^{\bot}+\na_xv$ and $-\pa_yv=\na_x\cdot u$, we deduce that
\begin{align*}
\big\|\widetilde{U}_a\cdot\nabla u\big\|^2_{H^9_{tan}}\leq&\|u_a\partial_x u\|^2_{H^9_{tan}}+
\|\tilde{v}_a\partial_y u\|^2_{H^9_{tan}}
\leq C\| u_\Phi\|^2_{H^{10}_{tan}}+
\Big\|\frac{\tilde{v}_a}{\varphi}(\varphi w^{\perp}+\varphi\partial_xv)\Big\|^2_{H^9_{tan}}\\
\leq& C\big(\|U\|^2_{H_{tan}^{10}}+\|\varphi w\|^2_{H_{co}^9}+\varepsilon^4\big)
\leq C\varepsilon^2\big(E(t)+\varepsilon^2\big),
\end{align*}
and
\begin{align*}
\big\|\widetilde{U}\cdot\nabla u_a\big\|^2_{H^{9}_{tan}}\leq&\big\|u\partial_x u_a\big\|^2_{H^{9}_{tan}}+
\big\|\tilde{v}\partial_y u_a\big\|^2_{H^{9}_{tan}}
\leq C\big\|u\big\|^2_{H^9_{tan}}+
\Big\|\frac{\tilde{v}}{\varphi}Zu_a\Big\|^2_{H^{9}_{tan}}\\
\leq&C\big(\|U\|^2_{H_{tan}^{10}}+\varepsilon^4\big)\leq C\varepsilon^2(E(t)+\varepsilon^2).
\end{align*}
Similarly, we have
\begin{align*}
\big\|u\partial_x u\big\|^2_{H^{9}_{tan}}
\leq&C\sum_{|i|\leq 9,|j|\leq 4}\int_0^{+\infty}\big\|\partial_x^ju\partial_x \partial_x^{i-j}u(\cdot,y)\big\|^2_{L^2(\R^2)}dy\\
&+
C\sum_{|i|\leq 9, |j|> 4}\int_0^{+\infty}\big\|\partial_x^ju\partial_x \partial_x^{i-j}u(\cdot,y)\big\|^2_{L^2(\R^2)}dy\\
\leq&C \|u\|_{H^7_{tan}}\|\partial_yu\|_{H^7_{tan}}\|u\|^2_{H_{tan}^{10}}
\leq C \big(\|w\|^2_{H^9_{co}}+\|U\|^2_{H^{10}_{tan}}\big)\|u\|^2_{H_{tan}^{10}}\leq C\varepsilon^2E(t)^2,
\end{align*}
and
\begin{align}
\big\|\tilde{v}\partial_y u\big\|^2_{H^{9}_{tan}}
\leq&C\sum_{|i|\leq 9, |j|\leq 4}\int_0^{+\infty}\big\|\partial_x^j\tilde{v}\partial_y \partial_x^{i-j}u(\cdot,y)\big\|^2_{L^2(\R^2)}dy\nonumber\\
&+
C\sum_{|i|\leq 9, |j|> 4}\int_0^{+\infty}\big\|\partial_x^j\tilde{v}\partial_y \partial_x^{i-j}u(\cdot,y)\big\|^2_{L^2(\R^2)}dy\nonumber\\
\leq&C \|\tilde{v}\|_{H^6_{tan}}\|\partial_y\tilde{v}\|_{H^6_{tan}}\big(\|w\|^2_{H_{co}^{9}}
+\|v\|^2_{H_{tan}^{10}}\big)+C \|w+\partial_xv\|^2_{H^6_{tan}}
\big(\|U\|^2_{H_{tan}^{10}}+\varepsilon^4\big)\nonumber\\
\leq&C \varepsilon^2\big(\varepsilon^2+E(t)\big)E(t).\nonumber
\end{align}
Putting these estimates together, we deduce the estimate of $F$.
\end{proof}

\subsection{Tangential analytic estimate of the pressure}

\begin{Lemma}\label{e:analytical pressure estimate}
There exists $\delta_0>0$ such that for any $\delta\in (0, \delta_0)$, there holds
\begin{align}
&\delta\big\|(\nabla p)_{\Phi}\big\|^2_{H^{7}_{tan}}+\big\|\theta'(\nabla p)_{\Phi}\big\|^2_{H^{8,\frac{1}{2}}_{tan}}\nonumber\\
&\leq  C\varepsilon^2\big(E(t)+K(t)+\varepsilon^2\big)\big(1+E(t)\big)\nonumber\\
&\qquad+C_0\delta\varepsilon^4\big\|(\partial_y+|D_x|)w\big\|^2_{H^9_{co}}
+C_0\delta\varepsilon^4\big\|((\partial_y+|D_x|)w)_\Phi\big\|^2_{H^8_{co}}.\nonumber
\end{align}
\end{Lemma}
\begin{proof}
A straightforward computation gives
\begin{align}\label{e:equation for analytical function}
\left \{
\begin {array}{ll}
-\Delta_{x}p_{\Phi}-\partial_y(\partial_yp)_{\Phi}-\langle D_x\rangle\theta'(\partial_yp)_{\Phi}
=\nabla_x\cdot F_\Phi+\partial_yG_\Phi+\theta'\langle D_x\rangle G_\Phi\\[3pt]
\quad\quad\quad\quad\qquad\qquad -\nabla_x\cdot(R_{h})_\Phi-\partial_y(R_v)_\Phi-\langle D_x\rangle\theta' (R_v)_\Phi,\\[3pt]
(\partial_yp)_\Phi(x,0)=-\varepsilon^2(\partial_{y}\na_x\cdot u)_\Phi(x,0)+(R_{v})_\Phi(x,0)+\varepsilon^2(\partial_tf)_\Phi(t,x)-\varepsilon^4\Delta_{x}f_\Phi(t,x).
\end{array}
\right.
\end{align}
Acting $\langle D_x\rangle^{\frac{1}{2}}\partial_x^i$ on both sides of (\ref{e:equation for analytical function}) , then taking $L^2(\mathbb{R}^3_+)$ inner product with $(\theta')^2\langle D_x\rangle^{\frac{1}{2}}\partial_x^ip_\Phi$ and summing over $1\leq |i|\leq 8$, we obtain
\begin{align}\label{e: energy estimate foe Euler part}
&\sum_{1\leq |i|\leq 8}\Big<-\langle D_x\rangle^{\frac12}\partial_x^i\triangle_xp_{\Phi}-
\langle D_x\rangle^{\frac12}\partial_x^i\partial_y(\partial_yp)_{\Phi}-
\langle D_x\rangle\theta'\langle D_x\rangle^{\frac12}\partial_x^i(\partial_yp)_{\Phi}, (\theta')^2\langle D_x\rangle^{\frac12}\partial_x^ip_\Phi\Big>\nonumber\\
&=\sum_{1\leq |i|\leq 8}\Big<\langle D_x\rangle^{\frac12}\partial_x^i\Big[\nabla_x\cdot F_\Phi+\partial_yG_\Phi+\langle D_x\rangle\theta' G_\Phi\Big]
,(\theta')^2\langle D_x\rangle^{\frac12}\partial_x^ip_\Phi\Big>\nonumber\\
&\quad+\sum_{1\leq |i|\leq 8}\Big<\langle D_x\rangle^{\frac12}\partial_x^i\Big[-\nabla_x\cdot(R_{h})_\Phi-\partial_y(R_{v})_\Phi-\langle D_x\rangle\theta' (R_{v})_\Phi
\Big]
,(\theta')^2\langle D_x\rangle^{\frac{1}{2}}\partial_x^ip_\Phi\Big>\nonumber\\
&\tre I_1+I_2.
\end{align}

First, integrating by parts and using (\ref{e:property of cutoff function}), the left hand side of (\ref{e: energy estimate foe Euler part}) is bigger than
\begin{align}
\big(\frac12-C_0\delta\big)\sum_{1\leq |i|\leq 8}\big\|\theta'\langle D_x\rangle^{\frac12}\partial_x^i(\nabla p)_{\Phi}\big\|^2_{L^2}-C_0\delta\big\|(\partial_xp)_{\Phi}\big\|^2_{H^{7,\frac12}_{tan}}.\nonumber
\end{align}

We get by integration by parts and Lemma \ref{e:uniform boundness for error} that
\begin{align}
I_{1}\leq& C\Big(\big\|F_{\Phi}\big\|^2_{H^{8,\frac{1}{2}}_{tan}}+\big\|G_{\Phi}\big\|^2_{H^{8,\frac{1}{2}}_{tan}}\Big)
+\frac{1}{10}\sum_{1\leq|i|\leq 8}\big\|\theta'\langle D_x\rangle^{\frac12}\partial_x^i(\nabla p)_{\Phi}\big\|^2_{L^2},\nonumber\\
I_2\leq& C\varepsilon^4+\frac{1}{10}\sum_{1\leq|i|\leq 8}\big\|\theta'\langle D_x\rangle^{\frac12}\partial_x^i(\nabla p)_{\Phi}\big\|^2_{L^2}.\nonumber
\end{align}
Thus, collecting the above estimates and fixing $\delta$ small, we arrive at
\begin{align}
&\sum_{1\leq |i|\leq 8}\big\|\theta'\langle D_x\rangle^{\frac12}\partial_x^i(\nabla p)_{\Phi}\big\|^2_{L^2}
\leq C\Big(\big\|F_{\Phi}\big\|^2_{H^{8,\frac{1}{2}}_{tan}}+\big\|G_{\Phi}\big\|^2_{H^{8,\frac{1}{2}}_{tan}}\Big)
+C_0\delta\big\|(\partial_x p)_{\Phi}\big\|^2_{H^{7,\frac{1}{2}}_{tan}}+C\varepsilon^4.\nonumber
\end{align}
Moreover, we get by (\ref{e:property of cutoff function}) that
\begin{align}
\big\|\theta'\langle D_x\rangle^{\frac12}(\nabla p)_{\Phi}\big\|^2_{L^2}\leq C_0\delta\big\|\langle D_x\rangle^{\frac12}(\nabla p)_{\Phi}\big\|^2_{L^2(\R^3_+)}\leq C_0\delta\big\|(\nabla p)_{\Phi}\big\|^2_{H^{7,\frac{1}{2}}_{tan}}.\nonumber
\end{align}
Thus, we arrive at
\begin{align}\label{e:pressure estimate step one}
\big\|\theta'(\nabla p)_{\Phi}\big\|^2_{H^{8,\frac{1}{2}}_{tan}}\leq C\big\| (F,G)_{\Phi}\big\|^2_{H^{8,\frac{1}{2}}_{tan}}
+C_0\delta\big\|(\nabla p)_{\Phi}\big\|^2_{H^{7,\frac{1}{2}}_{tan}}+C\varepsilon^4.
\end{align}

Acting $|D_x|^{\frac12}\partial_x^i$ on  both sides of (\ref{e:equation for analytical function}), then taking $L^2$ inner with $|D_x|^{\frac12}\partial_x^ip_\Phi$ and summing over $|i|\leq 7$, we obtain
\begin{align}\label{e: energy estimate low pressure part}
&\sum_{|i|\leq 7}\Big<-|D_x|^{\frac12}\partial_x^i\triangle_xp_{\Phi}-
|D_x|^{\frac12}\partial_x^i\partial_y(\partial_yp)_{\Phi}-
\langle D_x\rangle\theta'|D_x|^{\frac12}\partial_x^i(\partial_yp)_{\Phi}, |D_x|^{\frac12}\partial_x^ip_\Phi\Big>\nonumber\\
&=\sum_{ |i|\leq 7}\Big<|D_x|^{\frac12}\partial_x^i\Big[\nabla_x\cdot F_\Phi+\partial_yG_\Phi+\langle D_x\rangle\theta' G_\Phi\Big]
,|D_x|^{\frac12}\partial_x^ip_\Phi\Big>\nonumber\\
&\quad+\sum_{ |i|\leq 7}\Big<|D_x|^{\frac12}\partial_x^i\Big[-\nabla_x\cdot(R_{h})_\Phi-\partial_y(R_{v})_\Phi-\langle D_x\rangle\theta' (R_{v})_\Phi
\Big]
,| D_x|^{\frac{1}{2}}\partial_x^ip_\Phi\Big>.
\end{align}

Integrating by parts and using $\pa_yv=-\nabla_x\cdot u$, the left hand side of (\ref{e: energy estimate low pressure part}) is bigger than
\begin{align}
&(1-C_0\delta)\big\|(\nabla p)_{\Phi}\big\|^2_{H^{7,\frac{1}{2}}_{tan}}-C_0\|\nabla p\|^2_{L^2(\R^3_+)}+\sum_{ |i|\leq 7}\int_{\R^2}|D_x|^{\frac12}\partial_x^ip_\Phi|D_x|^{\frac12}\partial_x^i(R_{v})_\Phi(t,x,0)dx\nonumber\\
&+\sum_{ |i|\leq 7}\int_{\R^2}|D_x|^{\frac12}\partial_x^ip_\Phi|D_x|^{\frac12}\partial_x^i(-\varepsilon^2(\partial_{y}\na_x\cdot u)_\Phi+\varepsilon^2(\partial_tf)_\Phi-\varepsilon^4\Delta_{x}f_\Phi)(t,x,0)dx.\nonumber
\end{align}
Recalling that $f=\partial_x\int_0^{+\infty}u_p^{(1)}(t,x,y)dy$, by Lemma \ref{e:uniform boundness for approximate solution}, we get
\begin{align}
&\sum_{ |i|\leq 7}\Big|\int_{\R^2}|D_x|^{\frac12}\partial_x^ip_\Phi\langle |D_x|^{\frac12}\partial_x^i(\varepsilon^2(\partial_tf)_\Phi-\varepsilon^4\Delta_{x}f_\Phi)(t,x,0)dx\Big|\leq  \frac14\big\|(\nabla p)_{\Phi}\big\|^2_{H^{7,\frac{1}{2}}_{tan}}+C\varepsilon^4,\nonumber
\end{align}
and using $\partial_yu=w_h^{\bot}+\na_xv$,
\begin{align}
&\sum_{ |i|\leq 7}\Big|\int_{\R^2}|D_x|^{\frac12}\partial_x^ip_\Phi|D_x|^{\frac12}\partial_x^i\varepsilon^2(\partial_{y}\na_x\cdot u)_\Phi)(t,x,0)dx\Big|\\
&=\sum_{ |i|\leq 7}\Big|\int_{\R^3_+}\pa_y\big(|D_x|^{\frac12}\partial_x^ip_\Phi|D_x|^{\frac12}\partial_x^i\varepsilon^2(\partial_{y}\na_x\cdot u)_\Phi)(t,x,y)\big)dxdy\Big|\nonumber\\
&\leq  \frac14\big\|(\nabla p)_{\Phi}\big\|^2_{H^{7,\frac{1}{2}}_{tan}}+C_0\varepsilon^4\big\|((\partial_y+|D_x|)w)_{\Phi}\big\|^2_{H^{8}_{co}}
+C\varepsilon^4\Big(\big\|w_{\Phi}\big\|^2_{H^{8,\f12}_{co}}+\big\|U_{\Phi}\big\|^2_{H^{9,\f12}_{tan}}\Big).\nonumber
\end{align}
Thus, fixing $\delta$ small, the left hand side of (\ref{e: energy estimate low pressure part}) is bigger than
\begin{align}
&\frac12\big\|(\nabla p)_{\Phi}\big\|^2_{H^{7,\frac{1}{2}}_{tan}}-C_0\|\nabla p\|^2_{L^2(\R^3_+)}+\sum_{ |i|\leq 7}\int_{\R^2}|D_x|^{\frac12}\partial_x^ip_\Phi|D_x|^{\frac12}\partial_x^iR_{v}(t,x,0)dx\nonumber\\
&-C_0\varepsilon^4\big\|((\partial_y+|D_x|)w)_{\Phi}\big\|^2_{H^{8}_{co}}
-C\varepsilon^4\Big(\big\|w_{\Phi}\big\|^2_{H^{8,\f12}_{co}}+\big\|U_{\Phi}\big\|^2_{H^{9,\f12}_{tan}}+1\Big).\nonumber
\end{align}
Integrating by parts and using Lemma \ref{e:uniform boundness for error}, the right hand side of (\ref{e: energy estimate low pressure part}) can be bounded by
\begin{align}
\frac14\big\|(\nabla p)_{\Phi}\big\|^2_{H^{7,\frac{1}{2}}_{tan}}+\sum_{ |i|\leq 7}\int_{\R^2}|D_x|^{\frac12}\partial_x^ip_\Phi|D_x|^{\frac12}\partial_x^iR_{v}(t,x,0)dx+C\varepsilon^4
+C\big\|(F,G)_{\Phi}\big\|^2_{H^8_{tan}}.\nonumber
\end{align}
Therefore, we arrive at
\begin{align}\label{e:low pressure estimate}
\big\|(\nabla p)_{\Phi}\big\|^2_{H^{7,\f12}_{tan}}\leq & C\big\|(F,G)_{\Phi}\big\|^2_{H^8_{tan}}+C_0\varepsilon^4\big\|((\partial_y+|D_x|)w)_{\Phi}\big\|^2_{H^{8}_{co}}\nonumber\\
&+C\varepsilon^4\big(\big\|w_{\Phi}\big\|^2_{H^{8,\f12}_{co}}+\big\|U_{\Phi}\big\|^2_{H^{9,\f12}_{tan}}+1\big)+C_0\|\nabla p\|_{L^2(\R^3_+)}^2.
\end{align}
Following the proof of \eqref{e:low pressure estimate}, we have
\begin{align}\label{eq:p-L2}
\|\nabla p\|_{L^2(\R^3_+)}^2\le & C\big\|(F,G)\big\|^2_{L^2}+C_0\varepsilon^4\big\|((\partial_y+|D_x|)w)\big\|^2_{L^2}\nonumber\\
&+C\varepsilon^4\big(\big\|w\big\|^2_{H^{1}_{co}}+\big\|U\big\|^2_{H^{2}_{tan}}+1\big).
\end{align}

Putting (\ref{e:low pressure estimate}) and \eqref{eq:p-L2}  into (\ref{e:pressure estimate step one}), we obtain
\begin{align}
\big\|\theta'(\nabla p)_{\Phi}\big\|^2_{H^{8,\frac{1}{2}}_{tan}}\leq& C\Big(\big\| (F,G)_{\Phi}\big\|^2_{H^{8,\frac{1}{2}}_{tan}(0,y(t))}+
\big\|(F,G)\big\|^2_{H^9_{tan}}\Big)+C_0\delta\varepsilon^4\big\|((\partial_y+|D_x|)w)\big\|^2_{L^2}\nonumber\\
&+C_0\delta\varepsilon^4\big\|((\partial_y+|D_x|)w)_\Phi\big\|^2_{H^8_{tan}}+C\varepsilon^4\big(1+E(t)+K(t)\big),\nonumber
\end{align}
which along with Lemma \ref{lem:F-analytic} and Lemma \ref{lem:F-Sob} gives our result.
\end{proof}

\subsection{Tangential Sobolev estimate of the pressure}

\begin{Lemma}\label{e:Sobolev pressure estimate}
It holds that
\begin{align}
\big\|\nabla p\big\|^2_{H^7_{tan}}\leq  C\varepsilon^2\big(E(t)+\varepsilon^2\big)\big(1+E(t)\big)+C_0\varepsilon^4\|(\partial_y+|D_x|)w\|^2_{H^8_{co}}.\nonumber
\end{align}
\end{Lemma}
\begin{proof}
Acting $\partial_x^i$ on both sides of (\ref{e:equation for pressure}), and then taking $L^2$ inner product with $\partial_x^i p$, summing over all $|i|\leq 7$, we arrive at
\begin{align}\label{e:equality on Euler pressure}
-\sum\limits_{|i|\leq 7}\big<\partial_x^i\triangle p,\partial_x^ip\big>
=\sum\limits_{|i|\leq 7}\big<\partial_x^i(\nabla_x\cdot F+\partial_yG)-\partial_x^i(\nabla_x\cdot R_h+\partial_yR_v),\partial_x^ip\big>.
\end{align}
By integrating by parts, the left hand side of (\ref{e:equality on Euler pressure}) is bigger than
\begin{align}
&\big\|\nabla p\big\|^2_{H^7_{tan}}+\sum\limits_{|i|\leq 7}\int_{\R^2}\partial_x^ip\partial_x^iR_v(t,x,0)dx\nonumber\\
&\quad+\sum\limits_{|i|\leq 7}\int_{\R^2}\partial_x^ip\big(\varepsilon^2\partial_x^i\partial_{yy}v+\varepsilon^2\partial_t\partial_x^if
-\varepsilon^4\Delta_{x}\partial_x^if\big)(t,x,0)dx\nonumber.
\end{align}
Recalling that $f(t,x)=\partial_x\int_0^{+\infty} u_p^{(1)}(t,x,y)dy$ and by Lemma \ref{e:uniform boundness for approximate solution}, we have
\begin{align}
&\sum\limits_{|i|\leq 7}\Big|\int_{\R^2}\partial_x^ip\big(\varepsilon^2\partial_t\partial_x^if
-\varepsilon^4\partial_x^i\Delta_xf\big)(t,x,0)dx\Big|\leq \frac14\big\|\nabla p\big\|^2_{H^7_{tan}}+C\varepsilon^4.\nonumber
\end{align}
Using  $\pa_yv=-\na_x\cdot u$ and $\partial_yu=w_h^{\bot}+\na_xv$, we get
\begin{align}
\sum\limits_{|i|\leq 7}\Big|\int_{\R^2}\partial_x^ip\big(\varepsilon^2\partial_x^i\partial_{yy}\big)(t,x,0)dx\Big|
&=\sum\limits_{|i|\leq 7}\Big|\int_{\R^3_+}\partial_y\big(\partial_x^ip\varepsilon^2\partial_x^i\partial_{xy}u\big)(t,x,y)dxdy\Big|\nonumber\\
&\leq \frac14\big\|\nabla p\big\|^2_{H^7_{tan}}+C_0\varepsilon^4\big(\|\partial_yw\|^2_{H^8_{co}}+\|U\|^2_{H^9_{tan}}\big).\nonumber
\end{align}
So, the left hand side of (\ref{e:equality on Euler pressure}) is bigger than
\begin{align}
&\frac12\big\|\nabla p\big\|^2_{H^7_{tan}}+\sum\limits_{|i|\leq 7}\int_{\R^2}\partial_x^ip\partial_x^iR_v(t,x,0)dx-C_0\varepsilon^4\big(1+\|\partial_yw\|^2_{H^8_{co}}+\|U\|^2_{H^9_{tan}}\big)
\nonumber.
\end{align}
By integrating by parts, the right hand side of (\ref{e:equality on Euler pressure}) can be bounded by
\beno
\frac14\big\|\nabla p\big\|^2_{H^7_{tan}}+C\big\|(F,G)\big\|^2_{H^7_{tan}}+C\varepsilon^4
+\sum\limits_{|i|\leq 7}\int_{\R^2}\partial_x^ip\partial_x^iR_v(t,x,0)dx.
\eeno
Thus, we arrive at
\begin{align}
\big\|\nabla p\big\|^2_{H^7_{tan}}\leq C\big\|(F,G)\big\|^2_{H^7_{tan}}+C\varepsilon^4\big(1+\|w\|^2_{H^9_{co}}+\|U\|^2_{H^9_{tan}}\big)+C_0\ve^4\big\|(\partial_y+|D|)w\big\|^2_{H^8_{co}},\nonumber
\end{align}
which along with Lemma \ref{lem:F-Sob} gives our result.
\end{proof}

\section{Tangential analytic type estimate of the velocity}

In this section, we make tangential analytic type estimates for the velocity. In what follows, we always assume $t\in \big[0,\min{(T_0,T_a)}\big]$.

\begin{Proposition}\label{p:analytical velocity estimate}
There exists $\delta_0>0$ such that for any $\delta\in (0, \delta_0)$, there holds
\begin{align}
&\frac{1}{2}\frac{d}{dt}\big\|U_\Phi\big\|^2_{H^9_{tan}}+\lambda\big\|U_\Phi\big\|^2_{H^{9,\frac{1}{2}}_{tan}(0,y(t))}+\frac{\varepsilon^2}{2}\big\|(\nabla U)_\Phi\big\|^2_{H^9_{tan}}\nonumber\\
&\leq C\varepsilon^2\big(E(t)+K(t)+\varepsilon^2\big)\big(1+E(t)\big)\nonumber\\
&\qquad+\frac{\varepsilon^4}{100}\Big(\big\|((\partial_y+|D_x|)w)_\Phi\big\|^2_{H^8_{co}}+\big\|(\partial_y+|D_x|)w\big\|^2_{H^9_{co}}\Big).\nonumber
\end{align}
\end{Proposition}

\begin{proof}
Acting $\partial_x^ie^\Phi$ on both sides of (\ref{scaled-error-equation-2}), taking $L^2$ inner product with $\partial_x^i U_\Phi$, and then summing over all $|i|\leq 9$, we arrive at
\begin{align}
&\sum_{|i|\leq 9}\big<\partial_x^i(\partial_tU)_\Phi,\partial_x^iU_\Phi\big>-\varepsilon^2\sum_{|i|\leq 9}\big<\partial_x^i(\Delta U)_\Phi,\partial_x^iU_\Phi\big>\tre I_0\nonumber\\
&\leq\Big|\sum\limits_{|i|\leq 9}\big<\partial_x^i(\widetilde{U}_{a}\cdot\nabla U)_\Phi,\partial_x^iU_\Phi\big>\Big|
+\Big|\sum\limits_{|i|\leq 9}\big<\partial_x^i(\widetilde{U}\cdot\nabla U_{a})_\Phi,\partial_x^iU_\Phi\big>\Big|\nonumber\\
&\quad+\Big|\sum\limits_{|i|\leq 9}\big<\partial_x^i(\widetilde{U}\cdot\nabla U)_\Phi,\partial_x^iU_\Phi\big>\Big|
+\Big|\sum\limits_{|i|\leq 9}\big<\partial_x^i(\nabla p)_\Phi,\partial_x^iU_\Phi\big>\Big|\nonumber\\
&\quad+\Big|\sum\limits_{|i|\leq 9}\big<\partial_x^i(R_h,R_v)_\Phi,\partial_x^iU_\Phi\big>\Big|
\tre\sum_{i=1}^5I_i.\nonumber
\end{align}

Let us now handle them term by term.\smallskip

\no{\bf Step 1. Estimate of $I_0$}

We  get by integration by parts that
\begin{align}
\sum_{|i|\leq 9}\big<\partial_x^i(\partial_tU)_\Phi,\partial_x^iU_\Phi\big>=&\sum_{|i|\leq 9}\big<\partial_x^i(\partial_tU_\Phi)-\partial_x^i\partial_t\Phi U_\Phi,\partial_x^iU_\Phi\big>\nonumber\\
\geq&\frac{1}{2}\frac{d}{dt}\big\|U_{\Phi}\big\|^2_{H^9_{tan}}+
\lambda\big\|U_{\Phi}\big\|^2_{H^{9,\frac{1}{2}}_{tan}(0, y(t))}.\nonumber
\end{align}
Using $u(t,x,0)=0, \partial_yv|_{y=0}=-\na_x\cdot u|_{y=0}=0$ and $|\theta'(y)|\leq C_0\delta$, we deduce that
\begin{align}
-\varepsilon^2\sum_{|i|\leq 9}\big<\partial_x^i(\Delta U)_\Phi,\partial_x^iU_\Phi\big>
=&-2\varepsilon^2\big<\theta'(\partial_yU)_\Phi,\langle D_x\rangle U_{\Phi}\big>_{H^9_{tan}}
+\varepsilon^2\|(\nabla U)_\Phi\|^2_{H^9_{tan}}\nonumber\\
\geq&\frac{\varepsilon^2}{2}\big\|(\nabla U)_{\Phi}\big\|^2_{H^9_{tan}}-C\varepsilon^2\big\|U_{\Phi}\big\|^2_{H^9_{tan}}\nonumber
\end{align}
for small $\delta$. Thus, we obtain
\begin{align}
I_0\geq\frac{1}{2}\frac{d}{dt}\big\|U_{\Phi}\big\|^2_{H^9_{tan}}+
\lambda\big\|U_{\Phi}\big\|^2_{H^{9,\frac{1}{2}}_{tan}(0, y(t))}+\frac{\varepsilon^2}{2}\big\|(\nabla U)_{\Phi}\big\|^2_{H^9_{tan}}-C\varepsilon^2\big\|U_{\Phi}\big\|^2_{H^9_{tan}}.\nonumber
\end{align}

\no{\bf Step 2. Estimate of $I_1$.}

First of all, we deal with $\big|\big<\widetilde{U}_{a}\cdot\nabla u)_{\Phi},u_{\Phi}\big>_{H^9_{tan}}\big|$, which can be controlled by
\begin{align}
&\sum_{|i|\leq 9}\Big|\int_{0}^{y(t)}\big<\partial_x^i(\widetilde{U}_{a}\cdot\nabla u)_{\Phi}),\partial_x^iu_{\Phi}\big>_{L^2_x}dy\Big|
+\sum_{|i|\leq 9}\Big|\int_{y(t)}^{+\infty}\big<\partial_x^i(\widetilde{U}_{a}\cdot\nabla u)_\Phi,\partial_x^iu_\Phi\big>_{L^2_x}dy\Big|
\tre I_{11}+I_{12}.\nonumber
\end{align}
It follows from the first inequality of Lemma \ref{e:double linear estimate} and  Lemma \ref{e:uniform boundness for approximate solution}  that
\begin{align}
\sum_{|i|\leq9}\Big|\int_{0}^{y(t)}\big<\partial_x^i(u_{a}\partial_{x}u)_{\Phi},\partial_x^iu_{\Phi}\big>_{L^2_x}dy\Big|\leq& C\big\|u_\Phi\big\|^2_{H^{9,\frac{1}{2}}_{tan}(0,y(t))}.\nonumber
\end{align}
Using $\partial_yu=w_h^{\bot}+\na_xv$
and $\big(v_{a}-\varepsilon^2f(t,x)e^{-y}\big)|_{y=0}=0$, the same argument as above gives
\begin{align}
&\sum_{|i|\leq9}\Big|\int_{0}^{y(t)}\big<
\partial_x^i((v_{a}-\varepsilon^2f(t,x)e^{-y})\partial_yu)_{\Phi},\partial_x^iu_{\Phi}\big>_{L^2_x}dy\Big|\nonumber\\
&=\sum_{|i|\leq9}\Big|\int_{0}^{y(t)}\big<
\partial_x^i\Big(\frac{(v_{a}-\varepsilon^2f(t,x)e^{-y})}{\varphi}(\varphi w_h^{\bot}+ \varphi \partial_xv)\Big)_{\Phi},\partial_x^iu_{\Phi}\big>_{L^2_x}dy\Big|\nonumber\\
&\leq C\Big(\big\|U_\Phi\big\|^2_{H^{9,\frac{1}{2}}_{tan}(0,y(t))}+\big\|\varphi w_\Phi\big\|^2_{H^{8,\frac12}_{co}(0,y(t))}\Big).\nonumber
\end{align}
So, we get
\beno
I_{11}\le C\Big(\big\|U_\Phi\big\|^2_{H^{9,\frac{1}{2}}_{tan}(0,y(t))}+\big\|\varphi w_\Phi\big\|^2_{H^{8,\frac12}_{co}(0,y(t))}\Big).
\eeno

Thanks to $\phi(t,y)\leq 0$ for $y\ge y(t)$, $I_{12}$ can be controlled by
\begin{align}
 &C\sum_{|i|\leq 9}\big\|\partial_x^i(u_{a}\partial_{x}u+(v_{a}-\varepsilon^2f(t,x)e^{-y})\partial_yu)\big\|^2_{L^2(\R^3_+)}+C\|u\|^2_{H^9_{tan}}\nonumber\\
 &\le C\Big(\big\|U\big\|^2_{H^{10}_{tan}}+\big\|\varphi w\big\|^2_{H^9_{co}}\Big).\nonumber
\end{align}
Here we used $\partial_yu=w_h^{\bot}+\na_xv$ again.

Collecting the above estimates, we obtain
\begin{align}
&\big|\big<\widetilde{U}_{a}\cdot\nabla u)_{\Phi},u_{\Phi}\big>_{H^9_{tan}}\big|\nonumber\\
&\leq C\Big(\big\|U_\Phi\big\|^2_{H^{9,\frac{1}{2}}_{tan}(0,y(t))}+\big\|\varphi w_\Phi\big\|^2_{H^{8,\frac12}_{co}(0,y(t))}
+\big\|U\big\|^2_{H^{10}_{tan}}+\big\|\varphi w\big\|^2_{H^9_{co}}\Big).\nonumber
\end{align}
Using the first inequality of Lemma \ref{e:double linear estimate}, $\pa_yv=-\nabla_x\cdot u$ and Lemma \ref{e:uniform boundness for approximate solution},  similar argument as above gives
\begin{align}
\big|\big<\widetilde{U}_{a}\cdot\nabla v)_{\Phi},v_{\Phi}\big>_{H^9_{tan}}\big|\leq C\Big(\big\|U_\Phi\big\|^2_{H^{9,\frac{1}{2}}_{tan}(0,y(t))}+\big\|U\big\|^2_{H^{10}_{tan}}\Big).\nonumber
\end{align}
This shows that
\begin{align}
I_1
\leq C\Big(\big\|U_\Phi\big\|^2_{H^{9,\frac{1}{2}}_{tan}(0,y(t))}+\big\|\varphi w_\Phi\big\|^2_{H^{8,\frac12}_{co}(0,y(t))}
+ \big\|U\big\|^2_{H^{10}_{tan}}+\big\|\varphi w\big\|^2_{H^9_{co}}\Big).\nonumber
\end{align}

\no{\bf Step 3. Estimate of $I_2$.}

Using Lemma \ref{e:product estimate} and Lemma \ref{e:uniform boundness for approximate solution}, it is easy to deduce that
\begin{align}
\big|\big<(u\partial_{x}u_{a})_\Phi,u_{\Phi}\big>_{H^9_{tan}}\big|
\leq C\big(\|u_\Phi\|^2_{H^9_{tan}(0,y(t))}+\|u\|^2_{H^9_{tan}}).\nonumber
\end{align}
On the other hand, we have
\begin{align}
&\sum_{|i|\leq9}\big|\big<\partial_x^i((v+\varepsilon^2f(t,x)e^{-y})\partial_yu_{a})_{\Phi},\partial_x^iu_{\Phi}\big>\big|\nonumber\\
&\leq\sum_{|i|\leq9}\Big|\int_{0}^{y(t)}\big<\partial_x^i((v+\varepsilon^2f(t,x)e^{-y})\partial_yu_a)_{\Phi}
,\partial_x^iu_{\Phi}\big>_{L^2_x}dy\Big|\nonumber\\
&\quad+\sum_{|i|\leq 9}\Big|\int_{y(t)}^{+\infty}\big<\partial_x^i((v+\varepsilon^2f(t,x)e^{-y})\partial_yu_a)_\Phi,\partial_x^iu_\Phi\big>_{L^2_x}dy\Big|
\tre I_{21}+I_{22}.\nonumber
\end{align}
Using the fact that $\|\partial_x^i\partial_yu_a\|_{L^\infty(\mathbb{R}^2\times (y(t),\infty))}\leq C$ due to $y(t)\ge c_0$, it is easy to get
\begin{align}
I_{22}\leq C\big(\|U\|^2_{H^9_{tan}}+\varepsilon^4\big).\nonumber
\end{align}
Thanks to $(v+\varepsilon^2f(t,x)e^{-y})|_{y=0}=0$, $\pa_yv=-\nabla_x\cdot u$, Hardy's inequality and Lemma \ref{e:uniform boundness for approximate solution}, we arrive at
\begin{align}
I_{21}\leq &
 \sum_{|i|\leq9}\Big|\int_{0}^{y(t)}\Big<\partial_x^i\Big(\frac{1}{y}\int_0^y\big(-{\rm div}_xu(x,y')
 -\varepsilon^2f(t,x)e^{-y'}\big)dy'z\partial_zu_{a,p}\Big)_{\Phi}
,\partial_x^iu_{\Phi}\Big>_{L^2_x}(y)dy\Big|\nonumber\\
&+C\big(\|U_\Phi\|^2_{H^9_{tan}(0,y(t))}+\varepsilon^4\big)
 \leq C\Big(\big\|u_\Phi\big\|^2_{H^{9,\frac{1}{2}}_{tan}(0,y(t))}+\varepsilon^4\Big).\nonumber
\end{align}
This shows that
\begin{align}
\Big|\big<\big(\widetilde{U}_{a}\cdot\nabla u_a\big)_\Phi,u_{\Phi}\big>_{H^9_{tan}}\Big|
\leq C\Big(\|U\|^2_{H^9_{tan}}+\big\|U_{\Phi}\big\|^2_{H^{9,\frac{1}{2}}_{tan}(0,y(t))}+\varepsilon^4\Big).\nonumber
\end{align}
Using the fact that $\|\partial_x^i\partial_yv_a\|_{L^\infty}\leq C$, it is easy to deduce that
\begin{align}
\Big|\big<(\widetilde{U}_{a}\cdot\nabla v_a)_{\Phi},v_{\Phi}\big>_{H^9_{tan}}\Big|
\leq C\Big(\|U_{\Phi}\|^2_{H^9_{tan}}+\|U\|^2_{H^9_{tan}}+\varepsilon^4\Big).\nonumber
\end{align}
Thus, we obtain
\begin{align}
I_2\leq C\Big(\|U\|^2_{H^9_{tan}}+\big\|U_{\Phi}\big\|^2_{H^{9,\frac{1}{2}}_{tan}(0,y(t))}+\varepsilon^4\Big).\nonumber
\end{align}

\no {\bf Step 4. Estimate of $I_3$.}

We first consider $\big|\big<(\widetilde{U}\cdot\nabla u)_\Phi, u_\Phi\big>_{H^9_{tan}}\big|$, which can be bounded by
\begin{align}
&\sum\limits_{|i|\leq 9}\Big|\int_0^{y(t)}\big<\partial_x^i(\widetilde{U}\cdot\nabla u)_\Phi, \partial_x^iu_\Phi\big>_{L^2_x}dy\Big|
+\sum\limits_{|i|\leq 9}\Big|\int_{y(t)}^{+\infty}\Big|\big<\partial_x^i(\widetilde{U}\cdot\nabla u)_\Phi, \partial_x^iu_\Phi\big>_{L^2_x}dy\Big|\tre I_{31}+I_{32}.\nonumber
\end{align}
Using the interpolation inequality $\|g\|_{L^\infty_y}\le C\|g\|_{L^2_y}^\f12\|\pa_yg\|_{L^2_y}^\f12$ and $\partial_yu=w_h^{\bot}+\na_xv$, we deduce that
\begin{align}
&\sum_{|i|\leq 9}\int_{y(t)}^{+\infty}\big|\big<(\partial_x^{i}(u\cdot\na_xu)_\Phi,\partial_x^i u_\Phi\big>_{L^2_x}\big|dy\nonumber\\
&\leq C\|u\|^2_{H^{10}_{tan}}\|u\|_{L^\infty_yH^8_x}\leq C\|u\|^2_{H^{10}_{tan}}\|u\|^{\frac12}_{H^8_{tan}}\|w_h^{\bot}+\na_xv\|^{\frac12}_{H^8_{tan}}\nonumber\\
&\leq C\|U\|^2_{H^{10}_{tan}}\big(\|U\|_{H^{10}_{tan}}+\|U\|^{\frac12}_{H^{10}_{tan}}\|w\|^{\frac12}_{H^9_{co}}\big),\nonumber
\end{align}
and
\begin{align}
&\sum_{|i|\leq 9}\int_{y(t)}^{+\infty}\big|\big<(\partial_x^{i}((v+\varepsilon^2f(t,x)e^{-y})\partial_yu)_\Phi,\partial_x^i u_\Phi\big>_{L^2_x}\big|dy\nonumber\\
&\leq C\|u\|_{H^9_{tan}}(\|U\|_{H^{10}_{tan}}+\varepsilon^2)\|\partial_yu\|_{H^7_{tan}}
+C\|u\|_{H^9_{tan}}\|\partial_yu\|_{H^9_{tan}}\|v+\varepsilon^2f(t,x)e^{-y}\|_{L^\infty_yH^7_x}\nonumber\\
&\leq C\big(\|U\|^2_{H^{10}_{tan}}+\varepsilon^4\big)\big(\|U\|_{H^{10}_{tan}}+\|w\|_{H^9_{co}}\big),\nonumber
\end{align}
from which, we infer that
\begin{align}
I_{32}\leq & C\big(\|U\|^2_{H^{10}_{tan}}+\varepsilon^4\big)\big(\|U\|_{H^{10}_{tan}}+\|w\|_{H^9_{co}}\big).\nonumber
\end{align}

For $I_{31}$, using the third inequality of  Lemma \ref{e:double linear estimate} and $\partial_yu=w_h^{\bot}+\na_xv$ , we obtain
\begin{align}
&\sum\limits_{|i|\leq 9}\Big|\int_0^{y(t)}\big<\partial_x^i(u\partial_xu)_\Phi, \partial_x^iu_\Phi\big>_{L^2_x}dy\Big|\leq C\|u_\Phi\|^2_{H_{tan}^{9,\frac{1}{2}}(0,y(t))}\Big(1+\|U_\Phi\|^2_{H_{tan}^9}
+\|w_\Phi\|^2_{H_{co}^8}\Big),\nonumber
\end{align}
and by Lemma \ref{e:product estimate} and $\partial_yu=w_h^{\bot}+\na_xv$ and $\pa_yv=-\nabla_x\cdot u$,
\begin{align}
&\sum\limits_{|i|\leq 9}\Big|\int_0^{y(t)}\big<\partial_x^i((v+\varepsilon^2f(t,x)e^{-y})\partial_yu)_\Phi, \partial_x^iu_\Phi\big>_{L^2_x}(y)dy\Big|\nonumber\\
&\leq C\|\langle D_x\rangle^{-\frac12}((v+\varepsilon^2f(t,x)e^{-y})\partial_yu)_\Phi\|^2_{H^{9}_{tan}(0,y(t))}+C\|u_\Phi\|^2_{H^{9,\frac12}_{tan}(0,y(t))}\nonumber\\
&\leq C\Big(\|U_\Phi\|^2_{H^{9}_{tan}}+\varepsilon^4\Big)\Big(\|w_\Phi\|^2_{H^{8,\frac12}_{co}(0,y(t))}+\|v_\Phi\|^2_{H^{9,\frac12}_{tan}(0,y(t))}\Big)
\nonumber\\
&\quad+ C\Big(\|U_\Phi\|^2_{H^{9,\f12}_{tan}(0,y(t))}+\varepsilon^4\Big)\Big(\|w_\Phi\|^2_{H^{8}_{co}}+\|v_\Phi\|^2_{H^{9}_{tan}}\Big)
+C\|u_\Phi\|^2_{H^{9,\frac12}_{tan}(0,y(t))}
.\nonumber
\end{align}

Similarly, $\big|\big<(\widetilde{U}\cdot\nabla v)_\Phi, v_\Phi\big>_{H^9_{tan}}\big|$ can be controlled by
\begin{align}
\sum_{|i|\leq 9}\int_0^{y(t)}\big<\partial_x^i(\widetilde{U}\cdot\nabla v)_\Phi,\partial_x^iv_\Phi \big>_{L^2_x}dy
+\sum_{|i|\leq 9}\int_{y(t)}^{+\infty}\big<\partial_x^i(\widetilde{U}\cdot \nabla v)_\Phi,\partial_x^iv_\Phi \big>dy\tre I_{33}+I_{34}.\nonumber
\end{align}
Similar arguments as $I_{32}$ give
\begin{align}
I_{34}\leq &C\big(\|U\|^2_{H^{10}_{tan}}+\varepsilon^4\big)\big(\|U\|_{H^{10}_{tan}}+\|w\|_{H^9_{co}}\big).\nonumber
\end{align}
Using $\pa_yv=-\nabla_x\cdot u$ and the third inequality of Lemma \ref{e:double linear estimate}, we obtain
\begin{align}
I_{33}\leq C\|U_\Phi\|^2_{H_{tan}^{9,\frac{1}{2}}(0,y(t))}
\Big(1+\|U_\Phi\|^2_{H_{tan}^9}+\|w_\Phi\|^2_{H_{co}^8}\Big).\nonumber
\end{align}
Summing up the estimates of $I_{31}-I_{34}$, we deduce that
\begin{align}
I_3\leq C\ve^2\big(1+E(t)\big)\big(\ve^2+E(t)+K(t)\big).\nonumber
\end{align}

\no {\bf Step 5. Estimate of $I_4$.}

Using $\pa_yv=-\nabla_x\cdot u$ and  integration by parts, we arrive at
\begin{align}
I_4
\leq &\sum\limits_{1\leq |i|\leq 9}\Big|\int_{\mathbb{R}^3_+}\langle D_x\rangle^{\frac{1}{2}}\partial_x^ip_\Phi \theta'\langle D_x\rangle^{\frac{1}{2}}\partial_x^iv_\Phi dxdy\Big|\nonumber\\
&+\sum\limits_{1\leq |i|\leq 9}\Big|\int_{\mathbb{R}^2}\partial_x^ip_\Phi \partial_x^iv_\Phi(x,0) dx\Big|+C_0\delta\big\|(\nabla p)_\Phi\big\|^2_{L^2}+C\big\|U_\Phi\big\|^2_{L^2}\nonumber\\
\leq&C_0\Big(\big\|\theta'(\partial_xp)_\Phi\big\|^2_{H_{tan}^{8,\frac{1}{2}}}
+\big\|U_\Phi\big\|^2_{H_{tan}^{9,\frac{1}{2}}(0,y(t))}+\big\|U\big\|^2_{H_{tan}^{10}}\Big)\nonumber\\
&+C\varepsilon^4
+C_0\delta\big\|(\nabla p)_\Phi\big\|^2_{H_{tan}^{7}}+C\big\|U_\Phi\big\|^2_{H_{tan}^9}.\nonumber
\end{align}
Here we used $v(t,x,0)=-\varepsilon^2f(t,x).$\smallskip

\no{\bf Step 6. Estimate of $I_5$.}

It follows from Lemma \ref{e:uniform boundness for error}  that
\begin{align}
I_5 \leq C\|U_\Phi\|^2_{H^9_{tan}}+C\varepsilon^4.\nonumber
\end{align}

Putting the estimates of $I_1-I_5$ together, we deduce that
\begin{align}
&\frac{1}{2}\frac{d}{dt}\big\|U_\Phi\big\|^2_{H^9_{tan}}+\lambda\big\|U_\Phi\big\|^2_{H^{9,\frac{1}{2}}_{tan}(0,y(t))}+\frac{\varepsilon^2}{2}\big\|(\nabla U)_\Phi\big\|^2_{H^9_{tan}}\nonumber\\
&\leq C\varepsilon^2\big(E(t)+K(t)+\varepsilon^2\big)\big(1+E(t)\big)\nonumber\\
&\qquad
+C_0\delta\big\|(\nabla p)_\Phi\big\|^2_{H_{tan}^7}+C_0\big\|\theta'(\partial_xp)_\Phi\big\|^2_{H_{tan}^{8,\frac{1}{2}}},\nonumber
\end{align}
which along with  Lemma \ref{e:analytical pressure estimate}
gives our result.
\end{proof}

\section{Tangential Sobolev estimates of velocity}

In this section, we make the energy estimate of the velocity in tangential Sobolev space, which will be used to control the regularity of the velocity away from the boundary.

\begin{Proposition}\label{p:Sobolev velocity estimate}
There exists $\delta_0>0$ such that for any $\delta\in (0, \delta_0)$, there holds
\begin{align}
&\frac{1}{2}\frac{d}{dt}\big\|U\big\|^2_{H^{10}_{tan}}+\frac{\varepsilon^2}{2}\big\|\nabla U\big\|^2_{H^{10}_{tan}}
\leq C\varepsilon^{\frac43}\big(E(t)+\varepsilon^2\big)^{\frac{5}{3}}\nonumber\\
&\qquad+C\varepsilon^2\big(E(t)+1\big)\big(E(t)+\varepsilon^2\big)+\f {\varepsilon^4} {100}\|(\partial_y+|D_x|)w\|^2_{H^9_{co}}.\nonumber
\end{align}
\end{Proposition}

\begin{proof}
Acting $\partial_x^i$ on  both sides of (\ref{scaled-error-equation-2}), and taking $L^2$ inner product with $\partial_x^i U$, then summing over $|i|\leq 10$, we arrive at
\begin{align}
\frac{1}{2}\frac{d}{dt}\big\|U\big\|^2_{H^{10}_{tan}}
+\varepsilon^2\big\|\nabla U\big\|^2_{H^{10}_{tan}}
\leq&\Big|\sum\limits_{|i|\leq 10}\big<\partial_x^i(\widetilde{U}_{a}\cdot\nabla U),\partial_x^iU\big>\Big|
+\Big|\sum\limits_{|i|\leq 10}\big<\partial_x^i(\widetilde{U}\cdot\nabla U_{a}),\partial_x^iU\big>\Big|\nonumber\\
&+\Big|\sum\limits_{|i|\leq 10}\big<\partial_x^i(\widetilde{U}\cdot\nabla U),\partial_x^iU\big>\Big|
+\Big|\sum\limits_{|i|\leq 10}\big<\partial_x^i(\nabla p),\partial_x^iU\big>\Big|\nonumber\\
&+\Big|\sum\limits_{|i|\leq 10}\big<\partial_x^i(R_h,R_v),\partial_x^iU\big>\Big|
\tre\sum_{i=1}^5I_i.\nonumber
\end{align}
Let us now handle them term by term.\smallskip

\no{\bf Step 1. Estimate of $I_1$.}

The term $I_1$ is bounded by
\begin{align}
\Big|\sum\limits_{|i|\leq 10}\big<\widetilde{U}_{a}\cdot\nabla\partial_x^iU,\partial_x^iU\big>\Big|
+\Big|\sum\limits_{|i|\leq 10, j< i}\big<\partial_x^{i-j}\widetilde{U}_{a}\cdot\nabla\partial_x^jU,\partial_x^iU\big>\Big|.\nonumber
\end{align}
Due to $\widetilde{U}_a|_{y=0}=0$, integrating by parts and using Lemma \ref{e:uniform boundness for approximate solution}, the first term can be controlled by
$C\|U\|^2_{H^{10}_{tan}}.$ Using $\partial_yu=w_h^{\bot}+\na_xv$, $\pa_yv=-\na_x\cdot u$ and $(v_{a}-\varepsilon^2f(t,x)e^{-y})|_{y=0}=0$ and Lemma \ref{e:uniform boundness for approximate solution}, the second term can be bounded by
\begin{align}
C\big(\|U\|^2_{H^{10}_{tan}}+\|\varphi w\|^2_{H^9_{co}}\big),\nonumber
\end{align}
where we used the fact that
\begin{align}
&\Big|\sum\limits_{|i|\leq 10, j< i}\big<\partial_x^{i-j}(v_{a}-\varepsilon^2f(t,x)e^{-y})\partial_x^jw,\partial_x^iU\big>\Big|\nonumber\\
&=\Big|\sum\limits_{|i|\leq 10, j< i}\Big<\frac{\partial_x^{i-j}(v_{a}-\varepsilon^2f(t,x)e^{-y})}{\varphi}\partial_x^j(\varphi w),\partial_x^iU\big>\Big|
\leq C\|\varphi w\|^2_{H^9_{co}}.\nonumber
\end{align}
This shows that
$$I_1\leq C\big(\|U\|^2_{H^{10}_{tan}}+\|\varphi w\|^2_{H^9_{co}}\big).$$

\no{\bf Step 2. Estimate of $I_2$.}

It is easy to deduce from Lemma \ref{e:uniform boundness for approximate solution} that
\begin{align}
\Big|\sum\limits_{|i|\leq 10}\big<\partial_x^i(u\partial_xU_{a}+(v+\varepsilon^2f(t,x)e^{-y})\partial_y(U_{a,e}+\ve v_{a,p}e_3),\partial_x^iU\big>\Big|\leq  C\|U\|^2_{H^{10}_{tan}},\nonumber
\end{align}
where $e_3=(0,0,1)$. On the other hand, we have
\begin{align}
&\Big|\sum\limits_{|i|\leq 10}\big<\partial_x^i((v+\varepsilon^2f(t,x)e^{-y})\partial_yu_{a,p}),\partial_x^iu\big>\Big|\nonumber\\
=&\Big|\sum\limits_{|i|\leq 10}\int_{\frac{y(t)} {2}}^{+\infty}\big<\partial_x^i((v+\varepsilon^2f(t,x)e^{-y})\partial_yu_{a,p}),\partial_x^iu\big>\Big|\nonumber\\
&+\Big|\sum\limits_{|i|\leq 10}\int_0^{\frac{y(t)} {2}}\big<\partial_x^i((v+\varepsilon^2f(t,x)e^{-y})\partial_yu_{a,p}),\partial_x^iu\big>\Big|.\nonumber
\end{align}
Using the fact that $\|\partial_x^i\partial_yu_{a,p}\|_{L^\infty(\mathbb{R}^2\times (\frac{y(t)}{2},\infty))}\leq C$ due to $y(t)\ge c_0$,  the first term can be bounded  by
\begin{align}
 C\big(\|U\|^2_{H^{10}_{tan}}+\varepsilon^4\big).\nonumber
\end{align}
 Thanks to $(v+\varepsilon^2f(t,x)e^{-y})|_{y=0}=0$, $\pa_yv=-\nabla_x\cdot u,$ Hardy's inequality and Lemma \ref{e:uniform boundness for approximate solution}, the second term can be bounded by
\begin{align}
&\sum_{|i|\leq10}\Big|\int_{0}^{\frac{y(t)}{2}}\Big<\partial_x^i\Big(\frac{1}{y}\int_0^y\big(-{\rm div}_xu(x,y')
 -\varepsilon^2f(t,x)e^{-y'}\big)dy'(z\partial_zu_{a,p})\Big)
,\partial_x^iU\Big>_{L^2_x}(y)dy\Big|\nonumber\\
&\leq C\Big(\big\|U_\Phi\big\|^2_{H^{9}_{tan}}+\varepsilon^4\Big),\nonumber
\end{align}
where we used the fact that
\begin{align}
\sum_{|i|\leq 11}\int_{0}^{\frac{y(t)}{2}}\int_{\R^2}|\partial_x^{i}u|^2dxdy\leq \sum_{|i|\leq 9}\int_{0}^{\frac{y(t)}{2}}\int_{\R^2}|\partial_x^{i}u_\Phi|^2dxdy\leq C\big\|u_\Phi\big\|^2_{H^{9}_{tan}},\nonumber
\end{align}
due to $\phi(t,y)\geq c\delta> 0$ for $y\leq \frac{y(t)}{2}$.

Summing up, we arrive at
\begin{align}
I_2\leq C\Big(\big\|U_\Phi\big\|^2_{H^{9}_{tan}}+\|U\|^2_{H^{10}_{tan}}+\varepsilon^4\Big).\nonumber
\end{align}

\no{\bf Step 3. Estimate of $I_3$.}

By Sobolev embedding and  $\partial_yu=w_h^{\bot}+\na_xv$, we deduce that
\begin{align}
&\sum_{|i|\leq 10,j< i}\int_{0}^{+\infty}\big|\big<\partial_x^{i-j}u\partial_x\partial_x^ju,\partial_x^i u\big>_{L^2_x}\big|dy\nonumber\\
&\leq C\|u\|^2_{H^{10}_{tan}}\|u\|_{L^\infty_yH^8_x(\R^2)}\leq C\|u\|^2_{H^{10}_x}\|u\|^{\frac12}_{H^8_{tan}}\|w_h^{\bot}+\partial_xv\|^{\frac12}_{H^8_{tan}}\nonumber\\
&\leq C\|u\|^2_{H^{10}_{tan}}\big(\|U\|_{H^{10}_{tan}}+\|u\|^{\frac12}_{H^{10}_{tan}}\|w\|^{\frac12}_{H^9_{co}}\big),\nonumber
\end{align}
and
\begin{align}
&\sum_{|i|\leq 10, j< i}\int_{0}^{+\infty}\big|\big<(\partial_x^{i-j}(v+\varepsilon^2f(t,x)e^{-y})\partial_y\partial_x^ju,\partial_x^i u\big>_{L^2_x}\big|dy\nonumber\\
&\leq C\|u\|_{H^{10}_{tan}}(\|v\|_{H^{10}_{tan}}+\varepsilon^2)\|\partial_yu\|_{L^\infty_yH^7_x}
+C\|u\|_{H^9_{tan}}\|\partial_yu\|_{H^9_{tan}}\|v+\varepsilon^2f(t,x)e^{-y}\|_{L^\infty_yH^7_x}\nonumber\\
&\leq C\big(\|U\|^2_{H^{10}_{tan}}+\varepsilon^4\big)\big(\|U\|_{H^{10}_{tan}}+\|w\|_{H^9_{co}}
+\|\partial_yw\|^{\frac12}_{H^7_{co}}\|w\|^{\frac12}_{H^7_{co}}\big).\nonumber
\end{align}
On the other hand, we get by integration by parts and Lemma \ref{e:uniform boundness for approximate solution}  that
\begin{align}
\Big|\sum\limits_{|i|\leq 10}\big<u\partial_x\partial_x^iU+(v+\varepsilon^2f(t,x)e^{-y})\partial_y\partial_x^iU,\partial_x^iU\big>\Big|\leq C\big(\|U\|_{H^{10}_{tan}}^2+\ve^4\big).\nonumber
\end{align}
Thus, we obtain
\begin{align}
I_3\leq C\big(\|U\|^2_{H^{10}_{tan}}+\varepsilon^4\big)\big(1+\|U\|^2_{H^{10}_{tan}}+\|w\|^2_{H^9_{co}}
+\|\partial_yw\|^{\frac12}_{H^7_{co}}\|w\|^{\frac12}_{H^7_{co}}\big).\nonumber
\end{align}

\no{\bf Step 4. Estimate of $I_4$ and $I_5$.}

Using the divergence free condition and Lemma \ref{e:uniform boundness for approximate solution}, integrating by parts, we obtain
\begin{align}
I_4\leq
\sum\limits_{ |i|\leq 10}\Big|\int_{\mathbb{R}^2}\partial_x^ip\partial_x^iv(x,0) dx\Big|
\leq C\varepsilon^4
+\delta\big\|\nabla p\big\|^2_{H_{tan}^{7}}.\nonumber
\end{align}
Here we used $v(t,x,0)=-\varepsilon^2f(t,x).$

It follows from Lemma \ref{e:uniform boundness for error} that
\begin{align}
I_5\leq C\big(\|U\|^2_{H^{10}_{tan}}+\varepsilon^4\big).\nonumber
\end{align}

Summing up the estimates of $I_1-I_5$, we deduce that
\begin{align}
&\frac{1}{2}\frac{d}{dt}\big\|U\big\|^2_{H^{10}_{tan}}+\frac{\varepsilon^2}{2}\big\|\nabla U\big\|^2_{H^{10}_{tan}}
\leq C\varepsilon^{\frac43}\big(E_v(t)+\varepsilon^2\big)^{\frac{4}{3}}E_w(t)^{\frac13}\nonumber\\
&\quad+C\varepsilon^2\big(E(t)+1\big)\big(E(t)+\varepsilon^2\big)
+\delta\varepsilon^4\|(\partial_y+|D_x|)w\|^2_{H^9_{co}}+\delta\|\nabla p\|^2_{H^7_{tan}},\nonumber
\end{align}
which along with  Lemma \ref{e:Sobolev pressure estimate} gives our result.
\end{proof}

\section{Tangential analytic estimate of the vorticity: Euler part}

In this section, we make tangential analytic  estimates for the Euler part $w_{e}$ of the vorticity.

\subsection{Tangential analytic estimate of $w_e$}

Using  (\ref{e:decompose vorticity equation-1}), we first observe that $(w_{e})_\Phi$ satisfies
\begin{align}\label{e:vorticity equation on Euler part}
&\partial_t(w_e)_\Phi+\lambda\langle D_x \rangle (w_{e})_\Phi-\varepsilon^2
(\Delta  w_e)_\Phi+\big(\widetilde{U}_{a}\cdot\nabla w_e\big)_\Phi
+\big(\widetilde{U}\cdot\nabla w_{a,e}\big)_\Phi
+\big(\widetilde{U}\cdot\nabla w_e\big)_\Phi\nonumber\\[3pt]
&\qquad-(w_{a,e}\cdot\nabla U)_\Phi-(w_{e}\cdot\nabla U_a)_\Phi-(w_{e}\cdot\nabla U)_\Phi
=({\rm curl}(R_{e,h},R_{e,v})-M_e)_\Phi
\end{align}
together with the following initial-boundary conditions
\begin{align}
\left\{
\begin{array}{ll}
-\varepsilon^2((\partial_y+|D_x|)w_{e,h})_\Phi(t,x,0)-\varepsilon^2\partial_x(\Lambda_{ND}(\gamma\nabla_x\cdot w_{e,h}))_\Phi(t,x,0)=0,\\
(w_{e,3})_\Phi(t,x,0)=0,\\
(w_e)_\Phi(0,x,y)=0.\nonumber
\end{array}
\right.
\end{align}
\begin{Proposition}\label{e:euler vorticity etimate}
There exists $\delta_0>0$ such that for any $\delta\in (0, \delta_0)$, there holds
\begin{align}
&\frac{1}{2}\frac{d}{dt}\big\|(w_{e})_\Phi\big\|^2_{H^8_e}+\lambda\big\|(w_{e})_\Phi\big\|^2_{H^{8,\frac12}_e(0,y(t))}
+(\lambda-C)\Big\|\frac{(w_{e})_\Phi}{\varepsilon}\Big\|^2_{H^8_e(0,y(t))}
+\frac{\varepsilon^2}{10}\big\|((\partial_y+|D_x|)w_e)_\Phi\big\|^2_{H^8_e}\nonumber\\
&\leq C\big(E(t)+K(t)+\varepsilon^2\big)\big(1+E(t)\big)+C\varepsilon^{-\f23}E(t)^\f 53+\frac{\varepsilon^2}{100}\|((\partial_y+|D_x|)w)_\Phi\|^2_{H^8_{co}}.\nonumber
\end{align}
\end{Proposition}
\begin{proof}
Acting $\partial_x^iZ^j$ on both sides of (\ref{e:vorticity equation on Euler part}),  taking $L^2$ inner product with $e^{2\Psi_e}\partial_x^iZ^j (w_{e})_\Phi$, and then summing over all $|i|+j\leq 8$, we arrive at
\begin{align}
&\frac{1}{2}\frac{d}{dt}\big\|(w_{e})_\Phi\big\|^2_{H^8_e}+\lambda\big\|(w_{e})_\Phi\big\|^2_{H^{8,\frac12}_e(0,y(t))}
+\lambda\Big\|\frac{(w_{e})_\Phi}{\varepsilon}\Big\|^2_{H^8_e(0,y(t))}
-\varepsilon^2\sum\limits_{|i|+j\leq 8}\big<\partial_x^iZ^j(\Delta  w_e)_\Phi,e^{2\Psi_e}\partial_x^iZ^j(w_{e})_\Phi\big>\nonumber\\
&\leq\Big|\sum\limits_{|i|+j\leq 8}\big<\partial_x^iZ^j(\widetilde{U}_{a}\cdot\nabla w_e)_\Phi,e^{2\Psi_e}\partial_x^iZ^j(w_{e})_\Phi\big>\Big|
+\Big|\sum\limits_{|i|+j\leq 8}\big<\partial_x^iZ^j(\widetilde{U}\cdot\nabla w_{a,e})_\Phi,e^{2\Psi_e}\partial_x^iZ^j(w_{e})_\Phi\big>\Big|\nonumber\\
&\quad+\Big|\sum\limits_{|i|+j\leq 8}\big<\partial_x^iZ^j(\widetilde{U}\cdot\nabla w_e)_\Phi,e^{2\Psi_e}\partial_x^iZ^j(w_{e})_\Phi\big>\Big|
+\Big|\sum\limits_{|i|+j\leq 8}\big<\partial_x^iZ^j(w_{a,e}\cdot\nabla U)_\Phi,e^{2\Psi_e}\partial_x^iZ^j(w_{e})_\Phi\big>\Big|\nonumber\\
&\quad+\Big|\sum\limits_{|i|+j\leq 8}\big<\partial_x^iZ^j(w_{e}\cdot\nabla U_a)_\Phi,e^{2\Psi_e}\partial_x^iZ^j(w_{e})_\Phi\big>\Big|
+\Big|\sum\limits_{|i|+j\leq 8}\big<\partial_x^iZ^j(w_{e}\cdot\nabla U)_\Phi,e^{2\Psi_e}\partial_x^iZ^j(w_{e})_\Phi\big>\Big|\nonumber\\
&\quad+\Big|\sum\limits_{|i|+j\leq 8}\big<\partial_x^iZ^j({\rm curl}(R_{e,h},R_{e,v})-M_e)_\Phi,e^{2\Psi_e}\partial_x^iZ^j(w_{e})_\Phi\big>\Big|\tre\sum_{i=1}^7I_i.\nonumber
\end{align}

\no{\bf Step 1. Estimate of $I_1$.}

For any $|i|+j\leq 8$, there holds
\begin{align}
\big<\partial_x^iZ^j(\widetilde{U}_{a}\cdot\nabla w_e)_\Phi,e^{2\Psi_e}\partial_x^iZ^j(w_{e})_\Phi\big>=&\int_0^{y(t)}\int_{\R^2}\partial_x^iZ^j(\widetilde{U}_{a}\cdot\nabla w_e)_\Phi e^{2\Psi_e}\partial_x^iZ^j(w_{e})_\Phi dxdy\nonumber\\
&+\int_{y(t)}^{+\infty}\int_{\R^2}\partial_x^iZ^j(\widetilde{U}_{a}\cdot\nabla w_e)_\Phi e^{2\Psi_e}\partial_x^iZ^j(w_{e})_\Phi dxdy
\tre I_{11}+I_{12}.\nonumber
\end{align}

Thanks to $\phi(t,y)\leq 0$ and $\varphi(y)\ge c\delta$ for $y\geq y(t)$,  we get by  Lemma \ref{e:uniform boundness for approximate solution}  that
\begin{align}
I_{12} \leq C\|w_e\|^2_{H^9_{co}}.\nonumber
\end{align}

For $I_{11}$, by Lemma \ref{e:uniform boundness for approximate solution} and the first inequality of Lemma \ref{e:weighted double linear estimate}, we have
\begin{align}
\Big|\int_0^{y(t)}\int_{\R^2}\partial_x^iZ^j\big(u_{a}\partial_xw_e-(v_a-\varepsilon^2f(t,x)e^{-y}\big)|D_x|w_e)_\Phi e^{2\Psi_e}\partial_x^iZ^j(w_e)_\Phi dxdy\Big|\leq C\big\|(w_e)_\Phi\big\|^2_{H^{8,\frac{1}{2}}_e(0,y(t))},\nonumber
\end{align}
and  by Lemma \ref{e:uniform boundness for approximate solution} and the first inequality of  Lemma \ref{e:weighted estimate product}, we get
\begin{align}
&\Big|\int_0^{y(t)}\int_R\partial_x^iZ^j(\tilde{v}_a(\partial_y+|D_x|)w_e)_\Phi e^{2\Psi_e}\partial_x^iZ^j(w_e)_\Phi dxdy\Big|\nonumber\\
&\quad\leq  C\Big\|\frac{(w_e)_\Phi}{\varepsilon}\Big\|^2_{H^8_e(0,y(t))}+\frac{\varepsilon^2}{100}\big\|((\partial_y+|D_x|)w_e)_\Phi\big\|^2_{H^8_e(0,y(t))}.\nonumber
\end{align}
Hence, we obtain
\begin{align}
I_{11}
\leq C\big\|(w_e)_\Phi\big\|^2_{H^{8,\frac{1}{2}}_e(0,y(t))}
+C\Big\|\frac{(w_e)_\Phi}{\varepsilon}\Big\|^2_{H^8_e(0,y(t))}+\frac{\varepsilon^2}{100}\big\|(\partial_y+|D_x|)w_{e})_\Phi\big\|^2_{H^8_e}.
\nonumber
\end{align}
This shows that
\begin{align}
I_1\leq C\Big(\|w_e\|^2_{H^9_{co}}+\big\|(w_e)_\Phi\big\|^2_{H^{8,\frac{1}{2}}_e(0,y(t))}
+\Big\|\frac{(w_e)_\Phi}{\varepsilon}\Big\|^2_{H^8_e(0,y(t))}\Big)
+\frac{\varepsilon^2}{100}\big\|((\partial_y+|D_x|)w_e)_\Phi\big\|^2_{H^8_e}.\nonumber
\end{align}

\no{\bf Step 2. Estimate of $I_2$ and $I_7$.}

Using the fact that $w_{a,e}=0$ for $y\leq y(t)$, it is easy to deduce that
\begin{align}
I_2=&\Big|\sum\limits_{|i|+j\leq 8}\int_{y(t)}^{+\infty}\int_{\R^2}\partial_x^iZ^j(u\cdot\na_xw_{a,e}+(v+\varepsilon^2f(t,x)e^{-y})\partial_yw_{a,e})_\Phi e^{2\Psi_e}\partial_x^iZ^j(w_e)_\Phi dxdy\Big|\nonumber\\
\leq& C\Big(\|w_{e}\|^2_{H^8_e}+\big\|U\big\|^2_{H^9_{tan}}
+\|\varphi w\|^2_{H^8_{co}}+\varepsilon^4\Big).\nonumber
\end{align}
Here we used
\ben\label{eq:uw}
\|U\|_{H^9_{co}}\le C\big(\|U\|_{H^9_{tan}}+\|\varphi w\|_{H^8_{co}}\big),
\een
which follows from $Zu=\varphi w_h^{\bot}+\varphi\na_xv$ and
$-\pa_yv=\nabla_x\cdot u$.


It follows from the fact that $w_{a,e}=0$ for $y\leq y(t)$ and Lemma \ref{e:uniform boundness for error}  that
 \begin{align}
 I_7\leq C\big(\|w_{e}\|^2_{H^8_{co}}+\varepsilon^4\big).\nonumber
\end{align}

\no{\bf Step 3. Estimate of $I_3$ and $I_6$.}

Firstly, $I_3$ can be controlled by
\begin{align}
&\Big|\sum\limits_{|i|+j\leq 8}\int_0^{y(t)}\int_{\R^2}\partial_x^iZ^j(\widetilde{U}\cdot\nabla w_e)_\Phi e^{2\Psi_e}\partial_x^iZ^j(w_e)_\Phi dxdy\Big|\nonumber\\
&+\Big|\sum\limits_{|i|+j\leq 8}\int_{y(t)}^{+\infty}\int_{\R^2}\partial_x^iZ^j(\widetilde{U}\cdot\nabla w_e)_\Phi e^{2\Psi_e}\partial_x^iZ^j(w_e)_\Phi dxdy\Big|\tre I_{31}+I_{32}.\nonumber
\end{align}
As $\phi(t,y)\leq 0$ and $\varphi(y)\ge c\delta$ for $y\geq y(t)$, by \eqref{eq:uw}, we get
 \begin{align}
I_{32}\leq& C\|w_e\|^2_{H^9_{co}}\Big(\varepsilon^2+\|U\|^2_{H^{10}_{tan}}+\|\varphi w\|^2_{H^9_{co}}\Big).\nonumber
\end{align}
Using  the first inequality of  Lemma \ref{e:weighted double linear estimate}, we obtain
\begin{align}
&\Big|\sum\limits_{|i|+j\leq 8}\int_0^{y(t)}\int_{\R^2}\partial_x^iZ^j(u\cdot\na_xw_e-(v+\varepsilon^2f(t,x)e^{-y})|D_x|w_e)_\Phi e^{2\Psi_e}\partial_x^iZ^j(w_e)_\Phi dxdy\Big|\nonumber\\
&\leq C\Big(\varepsilon^2+\|U_\Phi\big\|^2_{H^9_{tan}}+\| w_\Phi\big\|^2_{H^8_{co}}\Big)
\big\|(w_e)_\Phi \big\|^2_{H^{8,\frac{1}{2}}_e(0,y(t))}.\nonumber
\end{align}
Here we used
\ben\label{eq:uw-2}
\|\pa_yU_\Phi\|_{H^8_{co}}\le C\big(\|w_\Phi\|_{H^8_{co}}+\|U_\Phi\|_{H^9_{tan}}\big),
\een
which can be deduced from $Zu=\varphi w_h^{\bot}+\varphi\na_xv$ and  $-\pa_yv=\nabla_x\cdot u$.

On the other hand, we have
\begin{align}
&\Big|\sum\limits_{|i|+j\leq 8}\int_0^{y(t)}\int_{\R^2}\partial_x^iZ^j((v+\varepsilon^2f(t,x)e^{-y})(\partial_y+|D_x|)w_e)_\Phi e^{2\Psi_e}\partial_x^iZ^j(w_e)_\Phi dxdy\Big|\nonumber\\
&\leq \big\|(w_e)_\Phi\big\|_{H^8_e}\|((v+\varepsilon^2f(t,x)e^{-y})(\partial_y+|D_x|)w_e)_\Phi \|_{H^8_e}\nonumber\\
&\leq \frac{\varepsilon^2}{100}\|((\partial_y+|D_x|)w_e)_\Phi\|^2_{H^8_e}+
C\varepsilon^{-2}\Big(\varepsilon^4+\|U_\Phi\big\|^2_{H^9_{tan}}+\|(\varphi w)_\Phi\big\|^2_{H^8_{co}}\Big)\|(w_e)_\Phi\big\|^2_{H^8_e}.\nonumber
\end{align}
Summing up, we obtain
\begin{align}
I_3\leq& \frac{\varepsilon^2}{100}\|((\partial_y+|D_x|)w_e)_\Phi\|^2_{H^8_e}
+C\varepsilon^{-2}\|(w_e)_\Phi\|^2_{H^8_e}\big(\|U_\Phi\|^2_{H^9_{tan}}+\|(\varphi w)_\Phi\big\|^2_{H^8_{co}}
+\varepsilon^4\big)\nonumber\\
&+C\Big(\epsilon^2+\|U_\Phi\big\|^2_{H^9_{tan}}+\|w_\Phi\big\|^2_{H^8_{co}}\Big)
\big\|(w_e)_\Phi\big\|^2_{H^{8,\frac{1}{2}}_e(0,y(t))}\nonumber\\
&+C\|w_e\|^2_{H^9_{co}}\Big(\epsilon^2+\|U\|^2_{H^{10}_{tan}}+\|\varphi w\|^2_{H^9_{co}}\Big).\nonumber
\end{align}

Similarly, $I_6$ can be controlled by
\begin{align}
&\Big|\sum\limits_{|i|+j\leq 8}\int_0^{y(t)}\int_{\R^2}\partial_x^iZ^j(w_{e,h}\cdot\na_xU+w_{e,3}\partial_yU)_\Phi e^{2\Psi_e}\partial_x^iZ^j(w_e)_\Phi dxdy\Big|\nonumber\\
&+\Big|\sum\limits_{|i|+j\leq 8}\int_{y(t)}^{+\infty}\int_{\R^2}\partial_x^iZ^j(w_{e,h}\cdot\na_xU+w_{e,3}\partial_yU)_\Phi e^{2\Psi_e}\partial_x^iZ^j(w_e)_\Phi dxdy\Big|\tre I_{61}+I_{62}.\nonumber
\end{align}
Similar to $I_{32}$, we have
\beno
I_{62}\le \big(\|U\|_{H^{10}_{tan}}+\|w\big\|_{H^9_{co}}\big)\|w_e\|_{H^8_{co}}^2.
\eeno
By the second inequality of Lemma \ref{e:weighted estimate product}, Sobolev embedding and \eqref{eq:uw-2}, we obtain
\begin{align}
I_{61}\leq &\|(w_{e,h}\cdot\na_xU+w_{e,3}\partial_yU)_\Phi\|_{H^8_e(0,y(t))}\|(w_e)_\Phi\|_{H^8_e}\nonumber\\
\le& C\Big(\|(w_e)_\Phi\|_{H^8_e}^\f12+\big\|\f {(w_e)_\Phi} {\varepsilon}\big\|_{H^8_e(0,y(t))}^\f12\varepsilon^{-\f12}+\|((\pa_y+|D_x|)w_e)_\Phi\|_{H^8_e}^\f12\Big)\big(\|U_\Phi\|_{H^9_{tan}}+\|w_\Phi\big\|_{H^8_{co}}\big)\|(w_e)_\Phi\|_{H^8_e}^\f32\nonumber\\
&+\Big(\|U_\Phi\|_{H^9_{tan}}+\|w_\Phi\big\|_{H^8_{co}}+\|w_\Phi\|_{H^8_{co}}^\f12\|((\pa_y+|D_x|)w)_\Phi\|_{H^8_{co}}^\f12\Big)\|(w_e)_\Phi\|_{H^8_e}^2\nonumber\\
\leq& \frac{\varepsilon^2}{100}\|((\partial_y+|D_x|)w_e)_\Phi\|^2_{H^8_e}
+\big\|\f {(w_e)_\Phi} {\varepsilon}\big\|_{H^8_e(0,y(t))}^2+C\varepsilon^{-\f23}\|(w_e)_\Phi\|^2_{H^8_e}\Big(\|U_\Phi\|_{H^9_{tan}}+\|w_\Phi\big\|_{H^8_{co}}+\|(w_e)_\Phi\|_{H^8_e}\Big)^\f43\nonumber\\
&+C\big(\|U_\Phi\|_{H^9_{tan}}+\|w_\Phi\big\|_{H^8_{co}}\big)\|(w_e)_\Phi\|_{H^8_e}^2+\frac{\varepsilon^2}{100}\|((\partial_y+|D_x|)w)_\Phi\|^2_{H^8_{co}}.\nonumber
\end{align}

\no{\bf Step 4. Estimate of $I_4$ and $I_5$.}

By Lemma \ref{e:uniform boundness for approximate solution},  \eqref{eq:uw} and \eqref{eq:uw-2}, we get
\begin{align}
I_4\leq& \Big|\sum\limits_{|i|+j\leq 8}\big<\partial_x^iZ^j(w_{a,e,h}\cdot\na_xU+w_{a,e,3}\partial_yU)_\Phi,e^{2\Psi_e}\partial_x^iZ^j(w_e)_\Phi\big>\Big|\nonumber\\
\leq &C\Big(\|U_\Phi\big\|^2_{H^9_{tan}}+\|w_\Phi\big\|^2_{H^8_{co}}+\|(w_e)_\Phi\big\|^2_{H^8_e}
+\|U\big\|^2_{H^9_{tan}}+\|w\big\|^2_{H^8_{co}}\Big).\nonumber
\end{align}
And $I_5$ can be controlled by
\begin{align}
&\Big|\sum\limits_{|i|+j\leq 8}\big<\partial_x^iZ^j(w_{e,3}\partial_yU_a)_\Phi,e^{2\Psi_e}\partial_x^iZ^j(w_e)_\Phi\big>\Big|+
\Big|\sum\limits_{|i|+j\leq 8}\big<\partial_x^iZ^j(w_{e,h}\cdot\na_xU_a)_\Phi,e^{2\Psi_e}\partial_x^iZ^j(w_e)_\Phi\big>\Big|\nonumber\\
&\tre I_{51}+I_{52}.\nonumber
\end{align}
By Lemma \ref{e:uniform boundness for approximate solution}, it is obvious that
$$
I_{52}\leq C\big(\|(w_{e})_\Phi\|^2_{H^8_e}+\|w_{e}\|^2_{H^8_{co}}\big).
$$
We decompose $I_{51}$ into two parts
\begin{align}
I_{51}\leq &\Big|\sum\limits_{|i|+j\leq 8}\int_0^{y(t)}\int_{\R^2}\partial_x^iZ^j(w_{e,3}\partial_yU_a)_\Phi e^{2\Psi_e}\partial_x^iZ^j(w_e)_\Phi dxdy\Big|\nonumber\\
&+\Big|\sum\limits_{|i|+j\leq 8}\int_{y(t)}^{+\infty}\int_{\R^2}\partial_x^iZ^j(w_{e,3}\partial_yU_a)_\Phi e^{2\Psi_e} \partial_x^iZ^j(w_e)_\Phi dxdy\Big|\tre I_{51}^1+I_{51}^2.\nonumber
\end{align}
Using the fact $\|\partial_x^iZ^j\partial_yU_a\|_{L^\infty(\R^2\times (y(t),+\infty))}\leq C$, we deduce that
$$I_{51}^2\leq C\|w_e\|^2_{H^8_{e}}.$$
And we get by Lemma \ref{e:uniform boundness for approximate solution} that
\begin{align}
I_{51}^1\leq C\Big\|\frac{(w_e)_\Phi}{\varepsilon}\Big\|^2_{H^8_e(0,y(t))}.\nonumber
\end{align}
Thus, we arrive at
$$I_{5}\leq C\|w_{e}\|^2_{H^8_{e}}+C\Big\|\frac{(w_e)_\Phi}{\varepsilon}\Big\|^2_{H^8_e(0,y(t))}.$$

\no{\bf Step 5. Estimate of dissipative term.}

Let $A\tre \partial_y+|D_x|.$ Using  the boundary conditions of $w_e$, we get by integration by parts that
 \begin{align}\label{e:diffusion estimate}
&-\varepsilon^2\sum\limits_{|i|+j\leq 8}\big<\partial_x^iZ^j(\Delta  w_e)_\Phi,e^{2\Psi_e}\partial_x^iZ^j(w_e)_\Phi\big>\nonumber\\
&=-\varepsilon^2\sum\limits_{|i|+j\leq 8}\big<\partial_x^iZ^j\big[\partial_y(A  w_e)_\Phi-|D_x|(A  w_e)_\Phi+\theta'\langle D_x\rangle(A  w_e)_\Phi\big],e^{2\Psi_e}\partial_x^iZ^j(w_e)_\Phi\big>\nonumber\\
&=\varepsilon^2\|(Aw_e)_\Phi\|^2_{H^8_e}-\varepsilon^2\sum\limits_{|i|+j\leq 8,j\geq 1}\big<[\partial_x^iZ^j,\partial_y](A  w_e)_\Phi,e^{2\Psi_e}\partial_x^iZ^j(w_e)_\Phi\big>\nonumber\\
&\quad-\varepsilon^2\sum\limits_{|i|\leq 8}\int_{\R^2}\partial_x^i\na_x(\Lambda_{ND}(\gamma\nabla_x\cdot w_{e,h}))_\Phi\cdot e^{2\Psi_e}\partial_x^i(w_{e,h})_\Phi(t,x,0)dx\nonumber\\
&\quad+\varepsilon^2\sum\limits_{|i|+j\leq 8,j\geq 1}\big<\partial_x^iZ^j  (Aw_e)_\Phi,e^{2\Psi_e}[\partial_y,\partial_x^iZ^j](w_e)_\Phi\big>
-2\sum\limits_{|i|+j\leq 8}\big<\partial_x^iZ^j  (Aw_e)_\Phi,e^{2\Psi_e}\theta'\partial_x^iZ^j(w_e)_\Phi\big>\nonumber\\
&\quad-2\varepsilon^2\sum\limits_{|i|+j\leq 8}\big<\partial_x^iZ^j  (Aw_e)_\Phi,e^{2\Psi_e}\theta'\langle D_x\rangle\partial_x^iZ^j(w_e)_\Phi\big>\nonumber\\
&\geq\frac{\varepsilon^2}{2}\big\|(Aw_e)_\Phi\big\|^2_{H^8_e}
-C_0\delta^{2}\varepsilon^{2}\big\|\langle D_x\rangle (w_e)_\Phi\big\|^2_{H^8_e}
-C_0\delta^{2}\Big\|\frac{(w_e)_\Phi}{\varepsilon}\Big\|_{H^8_e}^{2}\nonumber\\
&\quad-\varepsilon^2\sum\limits_{|i|\leq 8}\int_{\R^2}\partial_x^i\na_x(\Lambda_{ND}(\gamma{\rm div}_xw_{e,h}))_\Phi\cdot e^{2\Psi_e}\partial_x^i(w_e)_\Phi(t,x,0)dx.
\end{align}
By Lemma \ref{e:one order elliptic estimate} and (\ref{e:property of cutoff function}), we have
\begin{align}
\varepsilon\big\||D_x|(w_e)_\Phi\big\|_{H^8_e}=&\varepsilon\sum_{|i|+j\leq 8}\Big\||D_x|\big(e^{\Psi_e}\partial_x^iZ^j(w_e)_\Phi\big)\Big\|_{L^2(\R^3_+)}\nonumber\\
\leq& C_0\varepsilon\sum_{|i|+j\leq 8}\Big\|(\partial_y+|D_x|)\big(e^{\Psi_e}\partial_x^iZ^j(w_e)_\Phi\big)\Big\|_{L^2(\R^3_+)}\nonumber\\
\leq& C_0\delta\varepsilon\|(w_e)_\Phi\|_{H^8_e}+C_0\varepsilon\big\|((\partial_y+|D_x|)w_e)_\Phi\big\|_{H^8_e}+C_0\delta\varepsilon\||D_x|(w_e)_\Phi\|_{H^8_e}+
C\Big\|\frac{(w_e)_\Phi}{\varepsilon}\Big\|_{H^8_e},\nonumber
\end{align}
which implies that
\begin{align}\label{e:elliptic estimate one}
\varepsilon\big\||D_x|(w_e)_\Phi\big\|_{H^8_e}\leq C_0\delta\varepsilon\|(w_e)_\Phi\|_{H^8_e}+{C_0\varepsilon}\big\|((\partial_y+|D_x|)w_e)_\Phi\big\|_{H^8_e}
+C\Big\|\frac{(w_e)_\Phi}{\varepsilon}\Big\|_{H^8_e}.
\end{align}
Using the fact that $\mathcal{F}(\Lambda_{ND}f)=\frac{\hat{f}(\xi)}{|\xi|},$
we find that
\begin{align}\label{e:part diffusion estimate positive}
&-\sum\limits_{|i|\leq 8}\int_{\R^2}\partial_x^i\na_x(\Lambda_{ND}(\gamma{\rm div}_xw_{e,h}))_\Phi\cdot e^{2\Psi_e}\partial_x^i(w_{e,h})_\Phi(t,x,0)dx\nonumber\\
&=\sum\limits_{|i|\leq 8}\int_{\R^2}\partial_x^i(\Lambda_{ND}(\gamma{\rm div}_xw_{e,h}))_\Phi e^{2\Psi_e}\partial_x^i({\rm div}_xw_{e,h})_\Phi(t,x,0)\geq 0.
\end{align}
Putting (\ref{e:elliptic estimate one})-(\ref{e:part diffusion estimate positive}) into (\ref{e:diffusion estimate}) and fixing $\delta$ small, we arrive at
\begin{align}
&-\varepsilon^2\sum\limits_{i+j\leq 8}\big<\partial_x^iZ^j(\Delta  w_e)_\Phi,e^{2\Psi_e}\partial_x^iZ^j(w_{e})_\Phi\big>\nonumber\\
&\geq \frac{\varepsilon^2}{4}\big\|((\partial_y+|D_x|)w_e)_\Phi\big\|^2_{H^8_e}
-C\|(w_e)_\Phi\|^2_{H^8_e}-C\Big\|\frac{(w_e)_\Phi}{\varepsilon}\Big\|_{H^8_e}^{2}.\nonumber
\end{align}

Summing up the estimates in Step 1-Step 5, we conclude the proposition.\end{proof}

\subsection{Improved tangential analytic estimate of $w_{e,3}$}

Recall that $(w_{e,3})_\Phi$ satisfies
\begin{align}
&\partial_t(w_{e,3})_\Phi+\lambda\big<D_x\big>(w_{e,3})_\Phi-\varepsilon^2
(\Delta w_{e,3})_\Phi+\big(\widetilde{U}_{a}\cdot\nabla w_{e,3}\big)_\Phi
+\big(\widetilde{U}\cdot\nabla w_{a,e,3}
\big)_\Phi+\big(\widetilde{U}\cdot\nabla w_{e,3}\big)_\Phi\nonumber\\[3pt]
&\quad-(w_{a,e}\cdot\nabla v)_\Phi-(w_{e}\cdot\nabla v_a)_\Phi
-(w_{e}\cdot\nabla v)_\Phi
=({\rm curl}(R_{e,h},R_{e,v})_3-M_{e,3})_\Phi.\nonumber
\end{align}

\begin{Proposition}\label{e:Euler vertical vorticity etimate}
There exists $\delta_0>0$ such that for any $\delta\in (0, \delta_0)$, there holds
\begin{align}
&\frac{1}{2}\frac{d}{dt}\big\|(w_{e,3})_\Phi\big\|^2_{H^8_e}+{\lambda}\big\|(w_{e,3})_\Phi\big\|^2_{H^{8,\frac{1}{2}}_e(0,y(t))}
+({\lambda}-C)\Big\|\frac{(w_{e,3})_\Phi}{\varepsilon}\Big\|^2_{H^8_e(0,y(t))}
+\frac{\varepsilon^2}{2}\big\|(\nabla w_{e,3})_\Phi\big\|^2_{H^8_e}\nonumber\\
&\leq C\varepsilon^2\big(E(t)+1\big)\big(\varepsilon^2+E(t)+K(t)\big)+C\varepsilon^\f 43E(t)^\f53\nonumber\\
&\qquad+\frac{\varepsilon^2}{100}\Big(\varepsilon^2\big\|((\partial_y+|D_x|)w_e)_\Phi\big\|^2_{H^8_e}+\Big\|\frac{(w_{e})_\Phi}{\varepsilon}\Big\|^2_{H^8_e((0,y(t))}+\varepsilon^2\big\|((\partial_y+|D_x|)w)_\Phi\big\|^2_{H^8_{co}}\Big).\nonumber
\end{align}
\end{Proposition}

\begin{proof}
Compared with the case in Proposition \ref{e:euler vorticity etimate},
we have more decay in $\varepsilon$ for the third component $w_{e,3}$ of $w_e$. The key reason is that the following terms
\beno
(w_{a,e}\cdot\nabla v)_\Phi+(w_{e}\cdot\nabla v_a)_\Phi
+(w_{e}\cdot\nabla v)_\Phi
\eeno
behave better than the corresponding terms in the equation of $w_e$ due to $\pa_yv=-\nabla_x\cdot u$. Let us just show the following estimates:
\begin{align}
&\big|\big<-(w_{e,h}\cdot\na_xv_a)_\Phi-(w_{e,h}\cdot\na_xv)_\Phi, (w_{e,3})_\Phi\big>_{H^8_e}\big|\nonumber\\
&\leq \big|\big<(w_{e,h}\cdot\na_x(v_a-\varepsilon^2f(t,x)e^{-y}))_\Phi, (w_{e,3})_\Phi\big>_{H^8_e}\big|+\big|(w_{e,h}\cdot\na_x(v+\varepsilon^2f(t,x)e^{-y}))_\Phi, (w_{e,3})_\Phi\big>_{H^8_e}\big|.\nonumber
\end{align}
Using $(v_a-\varepsilon^2f(t,x)e^{-y})(t,x,0)=0$ and $\|\partial_x^iZ^j\partial_yv_a\|_{L^\infty}\leq C$, the first term can be bounded by
\begin{align}
&\Big|\Big<-\Big(\varphi w_{e,h}\cdot\frac{\na_x(v_a-\varepsilon^2f(t,x)e^{-y})}{\varphi}\Big)_\Phi, (w_{e,3})_\Phi\Big>_{H^8_e}\Big|\nonumber\\
&\leq C\Big(\|(w_{e,3})_\Phi\|^2_{H^8_e}+\|w_{e,3}\|^2_{H^8_{co}}+\|(\varphi w_e)_\Phi\|^2_{H^8_e}+\|(\varphi w_e)\|^2_{H^8_{co}}\Big).\nonumber
\end{align}
Using the second inequality of Lemma \ref{e:weighted estimate product} and Sobolev emdedding,  the second term in $(0,y(t))$ can be controlled by
\begin{align}
 & C\big\|(w_{e,3})_\Phi\big\|_{H^8_e}\big\|(w_{e,h}\partial_x(v+\varepsilon^2f(t,x)e^{-y}))_\Phi\big\|_{H^8_e(0,y(t))}\nonumber\\
&\leq  C\big\|(w_{e,3})_\Phi\big\|_{H^8_e}\big\|(w_{e})_\Phi\big\|_{H^8_e}\Big(\|U_\Phi\|_{H^9_{tan}}+\varepsilon^2+\|(\varphi w)_\Phi\|_{H^8_{co}}\Big)\nonumber\\
&\quad+\varepsilon^{-\f43}\|(w_{e,3})_\Phi\|^\f 43_{H^8_e}\big(\|U_\Phi\|_{H^9_{tan}}+\|(\varphi w)_\Phi\|_{H^8_{co}}\big)^\f 43\|(w_e)_\Phi\|^\f23_{H^8_e}+\frac{\varepsilon^2}{100}\Big\|\frac{(w_e)_\Phi}{\varepsilon}\Big\|_{H^8_e(0,y(t))}^{2}\nonumber
\end{align}
and using Sobolev embedding, the second term in $(y(t), \infty)$ can be controlled by
\beno
\frac{\varepsilon^4}{100}\|((\partial_y+|D_x|)w_e)_\Phi\big\|^2_{H^8_e}
+C\big\|w_{e,3}\big\|_{H^8_{co}}\big\|w_{co}\big\|_{H^8_e}\Big(\|U\|_{H^{10}_{tan}}+\varepsilon^2+\|w\|_{H^9_{co}}\Big).
\eeno
The estimates of the other terms can follow the proof of Proposition \ref{e:euler vorticity etimate}  line by line.
\end{proof}

\subsection{Improved tangential analytic estimate of $\varphi w_e$}
It is easy to verify that $\varphi w_e$ satisfies
\begin{align*}
&\partial_t(\varphi w_e)-\varepsilon^2
\Delta  (\varphi w_e)+\varphi\widetilde{U}_{a}\cdot\nabla w_e+\varphi\widetilde{U}\cdot\nabla w_{a,e}
+\varphi\widetilde{U}\cdot\nabla w_e-\varphi w_{a,e}\cdot\nabla U\\[4pt]
&\quad-(\varphi w_{e})\cdot\nabla U_a-(\varphi w_{e})\cdot\nabla U
=\varphi({\rm curl}(R_{e,h},R_{e,v})-M_e)
-\varepsilon^2\varphi''w_e-2\varepsilon^2\varphi'\partial_yw_e.
\end{align*}

\begin{Proposition}\label{e:analytical weighted Euler vorticity estimate}
There exists $\delta_0>0$ such that for any $\delta\in (0, \delta_0)$, there holds
\begin{align}
&\frac{1}{2}\frac{d}{dt}\|(\varphi w_e)_\Phi\|^2_{H^8_e}+{\lambda}\big\|(\varphi w_e)_\Phi\big\|^2_{H^{8,\frac12}_e(0,y(t))}
+({\lambda}-C)\Big\|\frac{(\varphi w_e)_\Phi}{\varepsilon}\Big\|^2_{H^8_e(0,y(t))}+\frac{\varepsilon^2}{2}\big\|(\nabla(\varphi w_e))_\Phi\big\|^2_{H^8_e}\nonumber\\
&\leq C\varepsilon^2\big(1+E(t)\big)\big(\varepsilon^2+E(t)+
K(t)\big)+C\varepsilon^{\f 43}E(t)^\f 53+\frac{\varepsilon^4}{100}\big\|((\partial_y+|D_x|)w)_\Phi\big\|^2_{H^8_{co}}+{\varepsilon^2}\Big\|\frac{(w_e)_\Phi}{\varepsilon}\Big\|_{H^8_e}^{2}.
\nonumber
\end{align}
\end{Proposition}
\begin{proof}
The reason why $\varphi w_e$ has more decay is very similar to the case $w_{3,e}$. Now the weight $\varphi$ cancels one singularity from the derivative $\pa_y$ due to $\varphi\pa_y=Z$. Let us just deal with
the following terms.  First, it is easy to get
\begin{align}
\big<\varepsilon^2\varphi''(w_e)_\Phi, (\varphi w_e)_\Phi\big>_{H^8_e}\leq C\|(\varphi w_e)_\Phi\|^2_{H^8_e}+C\varepsilon^4\|(w_e)_\Phi\|^2_{H^8_e}.\nonumber
\end{align}
 We get by integration by parts that
\begin{align}
&\big|\big<\varepsilon^2\varphi'(\partial_yw_e)_\Phi, (\varphi w_e)_\Phi\big>_{H^8_e}\big|\nonumber\\
&\leq  \f {\varepsilon^2}{100}\big\|(\pa_y(\varphi w_e)_\Phi\big\|_{H^8_e}^2+C\varepsilon^2\|(w_e)_\Phi\big\|^2_{H^{8,\f12}_e}
+{\varepsilon^2}\Big\|\frac{(w_e)_\Phi}{\varepsilon}\Big\|_{H^8_e}^{2}+C\big\|(\varphi w_e)_\Phi\big\|_{H^8_e}^2.\nonumber
\end{align}
The estimates of the other terms can follow the proof of Proposition \ref{e:euler vorticity etimate}  line by line.
\end{proof}

\section{Tangential analytic estimate of the vorticity: Prandtl part}

Let us observe that $(w_p)_\Phi$ satisfies
\begin{align}\label{e:vorticity equation on prandtlr part}
&\partial_t(w_p)_\Phi+\lambda\big<D_x\big>(w_p)_\Phi-\varepsilon^2
(\Delta w_p)_\Phi+\big(\widetilde{U}_{a}\cdot\nabla w_p\big)_\Phi
+\big(\widetilde{U}\cdot\nabla w_{a,p}\big)_\Phi+\big(\widetilde{U}\cdot\nabla w_p\big)_\Phi\nonumber\\[3pt]
&-(w_{a,p}\cdot\nabla U)_\Phi-(w_{p}\cdot\nabla U_a)_\Phi
-(w_{p}\cdot\nabla U)_\Phi
=({\rm curl}(R_{p,h},R_{p,v})-M_p)_\Phi
\end{align}
together with the following initial-boundary conditions
\begin{align}
\left\{
\begin{array}{ll}
-\varepsilon^2((\partial_y+|D_x|)w_{p,h})_\Phi(t,x,0)-\varepsilon^2\partial_x(\Lambda_{ND}(\gamma\nabla_x\cdot w_{p,h}))_\Phi(t,x,0)\nonumber\\[3pt]
\quad \quad =-(\partial_y(-\triangle_D)^{-1}J_h)_\Phi|_{y=0}+(\partial_x(-\triangle_N)^{-1}J_3)_\Phi|_{y=0},\nonumber\\[3pt]
 (w_{p,3})_\Phi(t,x,0)=0,\nonumber\\[3pt]
(w_p)_\Phi(0,x,y)=0.\nonumber
\end{array}
\right.
\end{align}
\begin{Proposition}\label{e:prandtl vorticity etimate}
There exists $\delta_0>0$ such that for any $\delta\in (0, \delta_0)$, there holds
\begin{align}
&\frac{1}{2}\frac{d}{dt}\big\|(w_p)_\Phi\big\|^2_{H^8_p}+{\lambda}\big\|(w_p)_\Phi\big\|^2_{H^{8,\frac{1}{2}}_p((0,y(t))}
+({\lambda}-C)\Big\|\frac{y(w_p)_\Phi}{\varepsilon}\Big\|^2_{H^8_p(0,y(t))}
\nonumber\\
&\qquad+\frac{\varepsilon^2}{10}\big\|((\partial_y+|D|)w_p)_\Phi\big\|^2_{H^8_p}\le C\big(1+E(t)\big)(K(t)+E(t)+\varepsilon^2)+C\varepsilon^{-2}E(t)^2\nonumber\\
&\qquad\qquad \qquad \qquad \qquad \qquad \qquad +\frac{\varepsilon^2}{100}\|((\partial_y+|D_x|)w)_\Phi\|^2_{H^8_{co}}.\nonumber
\end{align}
\end{Proposition}
\begin{proof}
Acting $\partial_x^iZ^j$ on both sides of (\ref{e:vorticity equation on prandtlr part}) and then taking $L^2$ inner product with $e^{2\Psi_p}\partial_x^iZ^j (w_p)_\Phi$, summing over all $|i|+j\leq 8$, we arrive at
\begin{align}
&\frac{1}{2}\frac{d}{dt}\big\|(w_p)_\Phi\big\|^2_{H^8_p}+\frac{\lambda}{2}\big\|(w_p)_\Phi\big\|^2_{H^{8,\frac12}_p(0,y(t))}
+\frac{\lambda}{2}\Big\|\frac{y(w_p)_\Phi}{\varepsilon}\Big\|^2_{H^8_p}-\varepsilon^2\sum\limits_{|i|+j\leq 8}\big<\partial_x^iZ^j(\Delta  w_p)_\Phi,e^{2\Psi_p}\partial_x^iZ^j(w_p)_\Phi\big>\nonumber\\
&\leq\Big|\sum\limits_{|i|+j\leq 8}\big<\partial_x^iZ^j(\widetilde{U}_{a}\cdot\nabla w_p)_\Phi,e^{2\Psi_p}\partial_x^iZ^j(w_p)_\Phi\big>\Big|
+\Big|\sum\limits_{|i|+j\leq 8}\big<\partial_x^iZ^j(\widetilde{U}\cdot\nabla w_{a,p})_\Phi,e^{2\Psi_p}\partial_x^iZ^j(w_p)_\Phi\big>\Big|\nonumber\\
&\quad+\Big|\sum\limits_{|i|+j\leq 8}\big<\partial_x^iZ^j(\widetilde{U}\cdot\nabla w_p)_\Phi,e^{2\Psi_p}\partial_x^iZ^j(w_p)_\Phi\big>\Big|
+\Big|\sum\limits_{|i|+j\leq 8}\big<\partial_x^iZ^j(w_{a,p}\cdot\nabla U)_\Phi,e^{2\Psi_p}\partial_x^iZ^j(w_p)_\Phi\big>\Big|\nonumber\\
&\quad+\Big|\sum\limits_{|i|+j\leq 8}\big<\partial_x^iZ^j(w_{p}\cdot\nabla U_a)_\Phi,e^{2\Psi_p}\partial_x^iZ^j(w_p)_\Phi\big>\Big|
+\Big|\sum\limits_{|i|+j\leq 8}\big<\partial_x^iZ^j(w_{p}\cdot\nabla U)_\Phi,e^{2\Psi_p}\partial_x^iZ^j(w_p)_\Phi\big>\Big|\nonumber\\
&+\Big|\sum\limits_{|i|+j\leq 8}\big<\partial_x^iZ^j({\rm curl}(R_{p,h},R_{p,v})-M_p)_\Phi,e^{2\Psi_p}\partial_x^iZ^j(w_p)_\Phi\big>\Big|
\tre\sum_{i=1}^7I_i.\nonumber
\end{align}

\no{\bf Step 1. Estimate of $I_1$.}

Following the argument of $I_1$ in Proposition \ref{e:euler vorticity etimate} , we can deduce that
\begin{align}
I_1\leq C\Big(\|w_p\|^2_{H^9_p}+\big\|(w_p)_\Phi\big\|^2_{H^{8,\frac{1}{2}}_p(0,y(t))}
+\Big\|\frac{y(w_p)_\Phi}{\varepsilon}\Big\|^2_{H^8_p(0,y(t))}\Big)
+\frac{\varepsilon^2}{100}\big\|((\partial_y+|D_x|)w_p)_\Phi\big\|^2_{H^8_p}.\nonumber
\end{align}

\no{\bf Step 2. Estimate of $I_2$.}

First, $I_2$ can be controlled by
\begin{align}
 &\Big|\sum\limits_{|i|+j\leq 8}\int_{y(t)}^{+\infty}\int_{\R^2}\partial_x^iZ^j(u\partial_xw_{a,p}+(v+\varepsilon^2f(t,x)e^{-y})\partial_yw_{a,p})_\Phi e^{2\Psi_p}\partial_x^iZ^j(w_p)_\Phi dxdy\Big|\nonumber\\
&+\Big|\sum\limits_{|i|+j\leq 8}\int_0^{y(t)}\int_{\R^2}\partial_x^iZ^j(u\partial_xw_{a,p}+(v+\varepsilon^2f(t,x)e^{-y})\partial_yw_{a,p})_\Phi e^{2\Psi_p}\partial_x^iZ^j(w_p)_\Phi dxdy\Big|\tre I_{21}+I_{22}.\nonumber
\end{align}
Using the fact that $\|e^{\Psi_p}\partial_x^iZ^j\partial_{x,y}w_{a,p}\|_{L^\infty(\R^2\times (y(t),\infty))}\leq C$, it is easy to get by \eqref{eq:uw} that
\begin{align}
I_{21}\leq C\big(\|w_p\|^2_{H^8_p}+\|U\|^2_{H^9_{tan}}+\|\varphi w\|^2_{H^8_{co}}+\ve^4\big).\nonumber
\end{align}
As $w_{a,p}=\partial_y(u_{a,p})-\partial_xv_{a,p}$,  we get by \eqref{eq:uw} that
\begin{align}
&\Big|\sum\limits_{|i|+j\leq 8}\int_0^{y(t)}\int_{\R^2}\partial_x^iZ^j(u\partial_xw_{a,p})_\Phi e^{2\Psi_p}\partial_x^iZ^j(w_p)_\Phi dxdy\Big|\nonumber\\
&\leq C\Big(\|(w_p)_\Phi\|^2_{H^8_p}+\varepsilon^{-2}\|U_\Phi\|^2_{H^9_{tan}}+\varepsilon^{-2}\|(\varphi w)_\Phi\|^2_{H^8_{co}}\Big).\nonumber
\end{align}
Using $\pa_yv=-\nabla_x\cdot u$ and $(v+\varepsilon^2f(t,x)e^{-y})|_{y=0}=0$ and Lemma \ref{e:uniform boundness for approximate solution}, we deduce that
\begin{align}
&\Big|\sum\limits_{|i|+j\leq 8}\int_0^{y(t)}\int_{\R^2}\partial_x^iZ^j((v+\varepsilon^2f(t,x)e^{-y})\partial_yw_{a,p})_\Phi e^{2\Psi_p}\partial_x^iZ^j(w_p)_\Phi dxdy\Big|\nonumber\\
&\leq \Big|\sum\limits_{|i|+j\leq 8}\int_0^{y(t)}\int_{\R^2}\partial_x^iZ^j\Big(\frac{1}{\varepsilon y}\int_0^y\Big(\partial_xu(x,y')+\varepsilon^2f(t,x)e^{-y'}\Big)dy'z(\varepsilon \partial_{xz}v_p^a+\partial_{zz}u_p^a)\Big)_\Phi e^{2\Psi_p}\partial_x^iZ^j(w_p)_\Phi dxdy\Big|\nonumber\\
&\leq C\|(w_p)_\Phi\|^2_{H^{8}_p}+C\varepsilon^{-2}\Big(\|(\varphi w)_\Phi\|^2_{H^{8}_{co}}
+\|U_\Phi\|^2_{H^{9}_{tan}}+\varepsilon^4\Big).\nonumber
\end{align}
Thus, we obtain
\begin{align}
I_{22}\leq C\|(w_p)_\Phi\|^2_{H^{8}_p}+C\varepsilon^{-2}\Big(\|(\varphi w)_\Phi\|^2_{H^{8}_{co}}
+\|U_\Phi\|^2_{H^{9}_{tan}}+\varepsilon^4\Big).\nonumber
\end{align}
This shows that
\begin{align}
I_2 \leq&  C\Big(\|w_p\|^2_{H^8_p}+\|U\|^2_{H^9_{tan}}+\|\varphi w\|^2_{H^8_{co}}+\|(w_p)_\Phi\|^2_{H^{8}_p}\Big)\nonumber\\
&+C\varepsilon^{-2}\Big(\|(\varphi w)_\Phi\|^2_{H^{8}_{co}}
+\|U_\Phi\|^2_{H^{9}_{tan}}+\varepsilon^4\Big).\nonumber
\end{align}

\no{\bf Step 3. Estimate of $I_3$ and $I_6$.}

Following the arguments of $I_3$ and $I_6$ in Proposition \ref{e:euler vorticity etimate}, we obtain
\begin{align}
I_3\leq& \frac{\varepsilon^2}{100}\|((\partial_y+|D_x|)w_p)_\Phi\|^2_{H^8_p}
+C\varepsilon^{-2}\|(w_p)_\Phi\|^2_{H^8_p}\big(\|U_\Phi\|^2_{H^9_{tan}}+\|(\varphi w)_\Phi\big\|^2_{H^8_{co}}
+\varepsilon^4\big)\nonumber\\
&+C\Big(\epsilon^2+\|U_\Phi\big\|^2_{H^9_{tan}}+\|w_\Phi\big\|^2_{H^8_{co}}\Big)
\big\|(w_p)_\Phi\big\|^2_{H^{8,\frac{1}{2}}_p(0,y(t))}\nonumber\\
&+C\|w_p\|^2_{H^9_{p}}\Big(\epsilon^2+\|U\|^2_{H^{10}_{tan}}+\|\varphi w\|^2_{H^9_{co}}\Big),\nonumber
\end{align}
and
\begin{align}
I_{6}\leq& \frac{\varepsilon^2}{100}\|((\partial_y+|D_x|)w_p)_\Phi\|^2_{H^8_p}
+\big\|\f {y(w_p)_\Phi} {\varepsilon}\big\|_{H^8_p(0,y(t))}^2+C\varepsilon^{-\f23}\|(w_p)_\Phi\|^2_{H^8_p}\big(\|U_\Phi\|_{H^9_{tan}}+\|w_\Phi\big\|_{H^8_{co}}+\|(w_p)_\Phi\|_{H^8_p}\big)^\f43\nonumber\\
&+C\big(\|U_\Phi\|_{H^9_{tan}}+\|w_\Phi\big\|_{H^8_{co}}\big)\|(w_p)_\Phi\|_{H^8_p}^2+\frac{\varepsilon^2}{100}\|((\partial_y+|D_x|)w)_\Phi\|^2_{H^8_{co}}+\big(\|U\|_{H^{10}_{tan}}+\|w\big\|_{H^9_{co}}\big)\|w_p\|_{H^8_{p}}^2.\nonumber
\end{align}

\no{\bf Step 4. Estimate of $I_4$.}

We split $I_4$ into two parts
\begin{align}
I_4\leq& \Big|\sum\limits_{|i|+j\leq 8}\big<\partial_x^iZ^j(w_{a,p,h}\partial_xU)_\Phi,e^{2\Psi_p}\partial_x^iZ^j(w_p)_\Phi\big>\Big|\nonumber\\
&+\Big|\sum\limits_{|i|+j\leq 8}\big<\partial_x^iZ^j(w_{a,p,3}\partial_yU)_\Phi,e^{2\Psi_p}\partial_x^iZ^j(w_p)_\Phi\big>\Big|\tre I_{41}+I_{42}.\nonumber
\end{align}
Using  $\|e^{\Psi_p}\partial_x^iZ^jw_{a,p,3}\|_{L^\infty}\leq C$,
we get by \eqref{eq:uw} and \eqref{eq:uw-2} that
\begin{align}
I_{42}\leq C\Big(\|U_\Phi\|^2_{H^9_{tan}}+\|w_\Phi\|^2_{H^8_{co}}+\|(w_p)_\Phi\|^2_{H^8_p}
+\|U\|^2_{H^9_{tan}}+\|w\|^2_{H^8_{co}}+\|w_{p}\|^2_{H^8_p}\Big),\nonumber
\end{align}
and using $\|e^{\Psi_p}\partial_x^iZ^jw_{a,p,h}\|_{L^\infty(\R^2\times (y(t),+\infty))}\leq C$, we deduce that
\begin{align}
I_{41}\leq& C\big(\|U\|^2_{H^9_{tan}}+\|\varphi w\|^2_{H^8_{co}}+\|w_{p}\|^2_{H^8_p}\big)
+\Big|\sum\limits_{|i|+j\leq 8}\int_0^{y(t)}\int_{\R^2}\partial_x^iZ^j(w_{a,p,h}\partial_xU)_\Phi e^{2\Psi_p}\partial_x^iZ^j(w_p)_\Phi dxdy\Big|\nonumber\\
\leq &C\Big(\|U\|^2_{H^9_{tan}}+\|\varphi w\|^2_{H^8_{co}}+\|w_{p}\|^2_{H^8_p}+\|(w_p)_\Phi\|^2_{H^8_p}\Big)\nonumber\\
&\qquad+C\varepsilon^{-2}\Big(\|(\varphi w)_\Phi\|^2_{H^{8}_{co}(0,y(t))}
+\|U_\Phi\|^2_{H^{9}_{tan}(0,y(t))}\Big).\nonumber
\end{align}
 Thus, we obtain
\begin{align}
I_{4}\leq& C\Big(\|U_\Phi\|^2_{H^9_{tan}}+\|w_\Phi\|^2_{H^8_{co}}+\|(w_p)_\Phi\|^2_{H^8_p}
+\|U\|^2_{H^9_{tan}}+\|w\|^2_{H^8_{co}}+\|w_{p}\|^2_{H^8_p}\Big)\nonumber\\
&\quad+C\varepsilon^{-2}\Big(\|(\varphi w)_\Phi\|^2_{H^{8}_{co}}
+\|U_\Phi\|^2_{H^{9}_{tan}}\Big).\nonumber
\end{align}

\no{\bf Step 5. Estimate of $I_5$.}

We split $I_5$ into two parts
\begin{align}
I_5\leq &\Big|\sum\limits_{|i|+j\leq 8}\big<\partial_x^iZ^j(w_{p,h}\partial_xU_a)_\Phi,e^{2\Psi_p}\partial_x^iZ^j(w_p)_\Phi \big>\Big|\nonumber\\
&+\Big|\sum\limits_{|i|+j\leq 8}\big<\partial_x^iZ^j(w_{p,3}\partial_yU_a)_\Phi,e^{2\Psi_p}\partial_x^iZ^j(w_p)_\Phi\big>\Big|\tre I_{51}+I_{52}.\nonumber
\end{align}
By Lemma \ref{e:uniform boundness for approximate solution}, we have
$$I_{51}\leq C\big(\|(w_p)_\Phi\|^2_{H^8_p}+\|w_{p}\|^2_{H^8_p}\big),$$
and
\begin{align}
I_{52}\leq C\big(\|w_p\|^2_{H^8_p}+\|(w_p)_\Phi\|^2_{H^8_p}\big)+C\varepsilon^{-2}\big(\|w_{p,3}\|^2_{H^8_p}
+\|(w_{p,3})_\Phi\|^2_{H^8_p}\big).\nonumber
\end{align}
This shows that
\begin{align}
I_5\leq C\big(\|w_p\|^2_{H^8_p}+\|(w_p)_\Phi\|^2_{H^8_p}\big)+C\varepsilon^{-2}\big(\|w_{p,3}\|^2_{H^8_p}
+\|(w_{p,3})_\Phi\|^2_{H^8_p}\big).\nonumber
\end{align}

\no{\bf Step 6. Estimate of $I_7$.}

Using Lemma \ref{e:uniform boundness for error}, it is easy to get
\begin{align}
I_7\leq C\|(w_p)_\Phi\|^2_{H^8_p}+C\varepsilon^2.\nonumber
\end{align}

\no{\bf Step 7. Estimate of dissipative term.}

Following the arguments  in Proposition \ref{e:euler vorticity etimate} , we can deduce that
 \begin{align}\label{e:diffusion on prandtl part}
&-\varepsilon^2\sum\limits_{|i|+j\leq 8}\big<\partial_x^iZ^j(\Delta  w_p)_\Phi,e^{2\Psi_p}\partial_x^iZ^jw_{p,\Phi}\big>\nonumber\\
&\geq \frac{\varepsilon^2}{2}\|(Aw_p)_\Phi\|^2_{H^8_p}
-C_0\delta^{2}\varepsilon^{2}\Big\|\langle D_x\rangle (w_p)_\Phi\Big\|^2_{H^8_p}
-C_0\delta^{2}\Big\|\frac{y(w_p)_\Phi}{\varepsilon}\Big\|_{H^8_p}^{2}\nonumber\\
&\quad+\varepsilon^2\sum_{|i|\leq 8}\int_{\R^2}\partial_x^i(Aw_p)_\Phi\partial_x^i(w_p)_\Phi(x,0)dx,
\end{align}
Following the argument of  (\ref{e:elliptic estimate one}), we have
\begin{align}
\varepsilon\big\||D_x|(w_p)_\Phi\big\|_{H^8_p}\leq C\delta\ve\|(w_p)_\Phi\|_{H^8_p}+{C\varepsilon}\Big\|(Aw_p)_\Phi\Big\|_{H^8_p}
+C\Big\|\frac{y(w_p)_\Phi}{\varepsilon}\Big\|_{H^8_p}.\nonumber
\end{align}
 Putting this estimate into (\ref{e:diffusion on prandtl part}) and fixing $\delta$ small, we arrive at
\begin{align}\label{e:final estimate on prandtl diffuaion}
&-\varepsilon^2\sum\limits_{|i|+j\leq 8}\big<\partial_x^iZ^j(\Delta w_p)_\Phi,e^{2\Psi_p}\partial_x^iZ^j(w_p)_\Phi\big>\nonumber\\
&\geq\frac{\varepsilon^2}{4}\|(Aw_p)_\Phi\|^2_{H^8_p}
-C\Big\|\frac{y(w_p)_\Phi}{\varepsilon}\Big\|_{H^8_p}^{2}-C\|(w_p)_\Phi\|^2_{H^8_p}\nonumber\\
&\qquad+\varepsilon^2\sum_{|i|\leq 8}\int_{\R^2}\partial_x^i((\partial_y+|D_x|)w_p)_\Phi\partial_x^i(w_p)_\Phi(t,x,0)dx.
\end{align}

Now, using the boundary condition of $w_p$, we arrive at
\begin{align}
&\varepsilon^2\sum_{|i|\leq 8}\int_{\R^2}\partial_x^i((\partial_y+|D_x|)w_p)_\Phi\partial_x^i(w_p)_\Phi(t,x,0)dx\nonumber\\
&=-\sum_{|i|\leq 8}\varepsilon^2\int_{\R^2}\partial_x^i\na_x(\Lambda_{ND}(\gamma\nabla_x\cdot w_{p,h}))_\Phi \cdot\partial_x^i(w_{p,h})_\Phi(t,x,0)dx\nonumber\\
&\qquad+\sum_{|i|\leq 8}\int_{\R^2}e^{2(\delta-\lambda t)|D_x|}\partial_x^iw_{p,h}\partial_x^i(-\partial_y(-\triangle_D)^{-1}J_h+\partial_x(-\triangle_N)^{-1}J_3)(t,x,0)dx.\nonumber
\end{align}
The same argument as (\ref{e:part diffusion estimate positive}) gives
$$-\varepsilon^2\int_{\R^2}\partial_x^i\na_x(\Lambda_{ND}(\gamma\nabla_x\cdot w_{p,h}))_\Phi\cdot\partial_x^i(w_{p,h})_\Phi(t,x,0)dx\geq 0.$$
To deal with another term, we introduce a smooth cut-off function $\zeta(t,y)$ satisfies $\zeta(t,0)=1$ and $\zeta(t,y)=0$ for $y>y(t)$.
Then we have
\begin{align}
&\int_{\R^2}e^{2(\delta-\lambda t)\langle D_x\rangle}\partial_x^iw_{p,h}(x,0)\partial_x^i(-\partial_y(-\triangle_D)^{-1}J_h+\partial_x(-\triangle_N)^{-1}J_3)(t,x,0)dx\nonumber\\
&=\int_{\R^2}\int_0^{y(t)}\partial_y\Big\{\zeta(t,y)e^{\Phi}\partial_x^iw_{p,h}\partial_x^i(-\partial_y(-\triangle_D)^{-1}J_{h,\Phi}
+\partial_x(-\triangle_N)^{-1}J_{3,\Phi})\Big\}dydx\nonumber\\
&=\int_{\R^2}\int_0^{y(t)}\zeta'(t,y)e^{\Phi}\partial_x^iw_{p,h}\partial_x^i(-\partial_y(-\triangle_D)^{-1}J_{h,\Phi}
+\partial_x(-\triangle_N)^{-1}J_{3,\Phi})dydx\nonumber\\
&\quad-\int_{\R^2}\int_0^{y(t)}\zeta(t,y)e^{\Phi}\partial_x^iw_{p,h}\theta'\langle D_x\rangle \partial_x^i(-\partial_y(-\triangle_D)^{-1}J_{h,\Phi}
+\partial_x(-\triangle_N)^{-1}J_{3,\Phi})dydx\nonumber\\
&\quad+\int_{\R^2}\int_0^{y(t)}\zeta(t,y)e^{\Phi}\partial_x^i\partial_yw_{p,h}\partial_x^i(-\partial_y(-\triangle_D)^{-1}J_{h,\Phi}
+\partial_x(-\triangle_N)^{-1}J_{3,\Phi})dydx\nonumber\\
&\quad+\int_{\R^2}\int_0^{y(t)}\zeta(t,y)e^{\Phi}\partial_x^iw_{p,h}\partial_x^i(-\partial_{yy}(-\triangle_D)^{-1}J_{h,\Phi})dydx\nonumber\\
&\quad+\int_{\R^2}\int_0^{y(t)}\zeta(t,y)e^{\Phi}\partial_x^iw_{p,h}\partial_x^i\partial_{xy}(-\triangle_N)^{-1}J_{3,\Phi}dydx\tre\sum_{k=1}^5I^b_k.\nonumber
\end{align}
Recall that $J=(J_h,J_3)=-\text{curl} (F, G)+\text{curl} R$, where
\beno
(F,G)=\widetilde{U}_a\cdot\nabla U+\widetilde{U}\cdot\nabla U_a+\widetilde{U}\cdot\nabla U.
\eeno
Using $L^2$ boundness of operators $\partial_y(-\triangle_D)^{-1}\partial_x$, $\partial_y(-\triangle_D)^{-1}\partial_y$, $\partial_{yy}(-\triangle_D)^{-1}$, $\nabla(-\triangle_N)^{-1}\partial_x$ and $\partial_{xy}(-\triangle_N)^{-1}$,  we can deduce from Lemma \ref{lem:F-analytic} that
\begin{align}
|I^b_1|\leq& C\|(w_p)_\Phi\|^2_{H^8_p}+C\sum_{|i|\leq 8}\big\|\partial_x^i(F,G)_\Phi\big\|^2_{L^2(\R^2\times (0,y(t)))}+C\varepsilon^4\nonumber\\
\leq &C\|(w_p)_\Phi\|^2_{H^8_p}+C\varepsilon^2(1+E(t))(E(t)+K(t)+\varepsilon^2),\nonumber
\end{align}
and
\begin{align}
|I^b_2|\leq& C\|(w_p)_\Phi\|^2_{H^{8,\frac12}_p}+C\sum_{|i|\leq 8}\big\|\langle D_x\rangle^{\frac12}\partial_x^i(F,G)_\Phi\big\|^2_{L^2(\R^2\times (0,y(t)))}+C\varepsilon^4\nonumber\\
\leq & C\big(1+E(t)\big)\big(\varepsilon^2+E(t)+K(t)\big).\nonumber
\end{align}
Similarly, we have
\begin{align}
|I^b_5|\leq& C\|(w_p)_\Phi\|^2_{H^{8,\frac12}_p}+C\sum_{|i|\leq 8}\big\|\langle D_x\rangle^{-\frac12}\partial_x^i(J_3)_\Phi\big\|^2_{L^2(\R^2\times (0,y(t)))}\nonumber\\
\leq & C\big(1+E(t)\big)\big(\varepsilon^2+E(t)+K(t)\big).\nonumber
\end{align}
For $I_3^b$, we have
\begin{align*}
I_3^b\le C\big(1+E(t)\big)\big(\varepsilon^2+E(t)+K(t)\big)+\f {\varepsilon^2} {100}\|(Aw_p)_\Phi\|_{H^8_p}^2+C\varepsilon^{-2}E(t)^2.\end{align*}
We write
\begin{align*}
I^b_4=\int_{\R^2}\int_0^{y(t)}\zeta(t,y)e^{\Phi}\partial_x^iw_{p,h}\partial_x^i(J_{h,\Phi})dydx
+\int_{\R^2}\int_0^{y(t)}\zeta(t,y)e^{\Phi}\partial_x^iw_{p,h}\partial_x^i(-\partial_{xx}(-\triangle_D)^{-1}J_{h,\Phi})dydx.
\end{align*}
Now, the second term can be handled as $I_5^b$, while the first term can be handed  as $I_1^b, I_2^b, I_3^b$ by integrating by parts and $(F, G)(t,x,0)=0$.

Thus, we arrive at
\begin{align}
&-\varepsilon^2\sum\limits_{|i|+j\leq 8}\big<\partial_x^iZ^j(\Delta w_p)_\Phi,e^{2\Psi_p}\partial_x^iZ^jw_{p,\Phi}\big>_2\nonumber\\
&\geq\frac{\varepsilon^2}{8}\|(Aw_p)_\Phi\|^2_{H^8_p}
-C\Big\|\frac{yw_p}{\varepsilon}\Big\|_{H^8_p}^{2}-
C\big(E(t)+1\big)\big(K(t)+E(t)+\varepsilon^2\big)+\varepsilon^{-2}E(t)^2.\nonumber
\end{align}

The proposition follows by combing the estimates in Step 1-Step 7.
\end{proof}
\medskip

As in $w_{e,3}, \varphi w_e$, $w_{p,3}$ and $\varphi w_p$ have
more decay in $\varepsilon$.  Let us just state the following results
without proof.

\begin{Proposition}\label{e:prandtl vertical vorticity etimate}
There exists $\delta_0>0$ such that for any $\delta\in (0, \delta_0)$, there holds
\begin{align}
&\frac{1}{2}\frac{d}{dt}\big\|(w_{p,3})_\Phi\big\|^2_{H^8_p}+{\lambda}\big\|(w_{p,3})_\Phi\big\|^2_{H^{8,\frac{1}{2}}_p(0,y(t))}
+({\lambda}-C)\Big\|\frac{y(w_{p,3})_\Phi}{\varepsilon}\Big\|^2_{H^8_p}
+\frac{\varepsilon^2}{2}\big\|(\nabla w_{p,3})_\Phi\big\|^2_{H^8_p}\nonumber\\
&\leq C\varepsilon^2\big(1+E(t)\big)\big(\ve^2+E(t)+K(t)\big)+C\varepsilon^\f 43E(t)^\f53\nonumber\\
&\qquad+\frac{\varepsilon^2}{100}\Big(\varepsilon^2\big\|((\partial_y+|D_x|)w_p)_\Phi\big\|^2_{H^8_p}
+\Big\|\frac{y(w_{p})_\Phi}{\varepsilon}\Big\|^2_{H^8_p((0,y(t))}+\varepsilon^2\big\|((\partial_y+|D_x|)w)_\Phi\big\|^2_{H^8_{co}}\Big).\nonumber
\end{align}
\end{Proposition}

\begin{Proposition}\label{e:analytical weighted Prandtl vorticity estimate}
There exists $\delta_0>0$ such that for any $\delta\in (0, \delta_0)$, there holds
\begin{align}
&\frac{1}{2}\frac{d}{dt}\|(\varphi w_p)_\Phi\|^2_{H^8_p}+{\lambda}\big\|(\varphi w_p)_\Phi\big\|^2_{H^{8,\frac12}_p(0,y(t))}
+({\lambda}-C)\Big\|\frac{y(\varphi w_p)_\Phi}{\varepsilon}\Big\|^2_{H^8_p}+\frac{\varepsilon^2}{2}\big\|(\nabla(\varphi w_p))_\Phi\big\|^2_{H^8_p}\nonumber\\
&\leq C\varepsilon^2\big(1+E(t)\big)\big(\varepsilon^2+E(t)+
K(t)\big)+C\varepsilon^{\f 43}E(t)^\f 53+\frac{\varepsilon^4}{100}\big\|((\partial_y+|D_x|)w)_\Phi\big\|^2_{H^8_{co}}
+\varepsilon^2\Big\|\frac{y(w_{p})_\Phi}{\varepsilon}\Big\|^2_{H^8_p}.
\nonumber
\end{align}
\end{Proposition}

\section{Conormal Sobolev estimate of the vorticity: Euler part}

\begin{Proposition}\label{e:Sobolev euler vorticity etimate}
There exists $\delta_0> 0$ such that for any $\delta\in (0,\delta_0)$,
there holds
\begin{align}
&\frac{1}{2}\frac{d}{dt}\big\|w_e\big\|^2_{H^9_{co}}+\frac{\varepsilon^2}{2}\big\|(\partial_y+|D_x|)w_e\big\|^2_{H^9_{co}}\nonumber\\
&\leq
C\big(E(t)+\varepsilon^2\big)+\frac{\varepsilon^2}{100}\|(\partial_y+|D_x|)w\|^2_{H^9_{co}}+C\ve^{-2}E(t)^2.\nonumber
\end{align}
\end{Proposition}
\begin{proof}
First, acting $\partial_x^iZ^j$ on both sides of  (\ref{e:vorticity equation on Euler part}), and then taking $L^2$ inner product with $\partial_x^iZ^j w_e$, summing over all $|i|+j\leq 9$, we arrive at
\begin{align}
&\frac{1}{2}\frac{d}{dt}\big\|w_e\big\|^2_{H^9_{co}}-\varepsilon^2\sum\limits_{|i|+j\leq 9}\big<\partial_x^iZ^j(\Delta  w_e),\partial_x^iZ^jw_{e}\big>\nonumber\\
&\leq\Big|\sum\limits_{|i|+j\leq 9}\big<\partial_x^iZ^j(\widetilde{U}_{a}\cdot\nabla w_e),\partial_x^iZ^jw_e\big>\Big|
+\Big|\sum\limits_{|i|+j\leq 9}\big<\partial_x^iZ^j(\widetilde{U}\cdot\nabla w_{a,e}),\partial_x^iZ^jw_e\big>\Big|\nonumber\\
&\quad+\Big|\sum\limits_{|i|+j\leq 9}\big<\partial_x^iZ^j(\widetilde{U}\cdot\nabla w_{e}),\partial_x^iZ^jw_e\big>\Big|
+\Big|\sum\limits_{|i|+j\leq 9}\big<\partial_x^iZ^j(w_{a,e}\cdot\nabla U),\partial_x^iZ^jw_e\big>\Big|\nonumber\\
&\quad+\Big|\sum\limits_{|i|+j\leq 9}\big<\partial_x^iZ^j(w_{e}\cdot\nabla U_a),\partial_x^iZ^jw_e\big>\Big|
+\Big|\sum\limits_{|i|+j\leq 9}\big<\partial_x^iZ^j(w_{e}\cdot\nabla U
),\partial_x^iZ^jw_e\big>\Big|\nonumber\\
&\quad+\Big|\sum\limits_{|i|+j\leq 9}\big<\partial_x^iZ^j({\rm curl}(R_{e,h},R_{e,v})-M_e),\partial_x^iZ^jw_e\big>\Big|
\tre\sum_{i=1}^7I_i.\nonumber
\end{align}

\no{\bf Step 1. Estimate of $I_1$.}

We bound $I_1$ as
\begin{align}
I_1\leq C
&\sum\limits_{|i|+j\leq 9}\sum\limits_{(m,n)<(i,j)}\Big|\int_{\R^3_+}\big(\partial_x^{i-m}Z^{j-n}u_{a}\partial_x\partial_x^{m}Z^{n}w_e\nonumber\\
&+\partial_x^{i-m}Z^{j-n}(v_{a}-\varepsilon^2f(t,x)e^{-y})\partial_x^{m}Z^{n}\partial_yw_e\big)\partial_x^iZ^jw_edxdy\Big|\nonumber\\
&+\sum\limits_{|i|+j\leq 9}\Big|\int_{\R^3_+}\big(u_{a}\partial_x\partial_x^iZ^jw_e+(v_{a}-\varepsilon^2f(t,x)e^{-y})\partial_x^iZ^j\partial_yw_e\big)\partial_x^iZ^jw_edxdy\Big|.\nonumber
\end{align}
By Lemma \ref{e:uniform boundness for approximate solution}, the first term can be controlled by
$C\|w_e\|^2_{H^9_{co}}$,
where we used $(v_{a}-\varepsilon^2f(t,x)e^{-y})|_{y=0}=0$ and the fact
\begin{align}\label{e:property of va}
\Big\|\frac{\partial_x^{i-m}Z^{j-n}(v_{a}-\varepsilon^2f(t,x)e^{-y})}{\varphi}\Big\|_{L^\infty}\leq C.
\end{align}
Notice that $[\partial_y, \partial_x^{i}Z^{j}]=j\varphi'\partial_x^{i}Z^{j-1}\partial_y$. By Lemma \ref{e:uniform boundness for approximate solution} and integrating by parts, we deduce that the second term can be bounded by
\begin{align}
\sum\limits_{|i|+j\leq 9, j\geq 1}\Big|\int_{\R^3_+}(v_a-\varepsilon^2f(t,x)e^{-y})\varphi'\partial_x^{i}Z^{j-1}\partial_yw_e\partial_x^iZ^jw_edxdy\Big|+ C\|w_e\|^2_{H^9_{co}}.\nonumber
\end{align}
Thus, by (\ref{e:property of va}), we deduce that the second term is bounded by
$
C\|w_e\|^2_{H^9_{co}}.
$
Thus, we arrive at
\begin{align}
I_1\leq C\|w_e\|^2_{H^9_{co}}.\nonumber
\end{align}

\no{\bf Step 2. Estimate of $I_2$ and $I_7$.}

By Lemma \ref{e:uniform boundness for approximate solution}, $Zu=\varphi w_h^{\bot}+\varphi\na_xv$ and $\pa_yv=-\nabla_x\cdot u$, it is easy to obtain
\begin{align}
I_2
\leq C\Big(\|w_{e}\|^2_{H^9_{co}}+\|\varphi w\|^2_{H^9_{co}}+\big\|U\big\|^2_{H^9_{tan}}
+\varepsilon^4\Big).\nonumber
\end{align}
We infer from Lemma \ref{e:uniform boundness for error} that
\begin{align}
 I_7\leq C\big(\|w_{e}\|^2_{H^8_{co}}+\varepsilon^4\big).\nonumber
\end{align}

\no{\bf Step 3. Estimate of $I_3$ and $I_6$.}

First, $I_3$ can be controlled by
\begin{align}
\sum\limits_{|i|+j\leq 9}\sum\limits_{(m,n)<(i,j)}\Big|\big<\partial_x^{i-m}Z^{j-n}\widetilde{U}\cdot\partial_x^mZ^n\nabla w_e,\partial_x^iZ^jw_e\big>\Big|
+\Big|\sum\limits_{|i|+j\leq 9}\big<\widetilde{U}\cdot\partial_x^iZ^j\nabla w_e,\partial_x^iZ^jw_e\big>\Big|.\nonumber
\end{align}
Integrating by parts, the second term can be bounded by
\begin{align}
 C\varepsilon^2\|w_e\|^2_{H^9_{co}}+\sum\limits_{|i|+j\leq 9,j\geq 1}\Big|\big<(v+\varepsilon^2f(t,x)e^{-y})\varphi'\partial_x^iZ^{j-1}\partial_yw_e,\partial_x^iZ^jw_e\big>\Big|.\nonumber
\end{align}
Using $(v+\varepsilon^2f(t,x)e^{-y})|_{y=0}=0,$  we infer that
\begin{align}
&\sum\limits_{|i|+j\leq 9,j\geq 1}\Big|\big<(v+\varepsilon^2f(t,x)e^{-y})\varphi'\partial_x^iZ^{j-1}\partial_yw_e,\partial_x^iZ^jw_e\big>\Big|\nonumber\\
&=\sum\limits_{|i|+j\leq 9,j\geq 1}\Big|\Big<\frac{(v+\varepsilon^2f(t,x)e^{-y})}{\varphi}\varphi'\partial_x^iZ^{j}w_e,\partial_x^iZ^jw_e\Big>\Big|\nonumber\\
&\leq
C\big(\|U\|_{H^{10}_{tan}}+\varepsilon^2\big)\|w_e\|^2_{H^9_{co}}.\nonumber
\end{align}
Therefore, the second term can be bounded by
$$C\big(\varepsilon^2+\|U\|_{H^{10}_{tan}}\big)\|w_e\|^2_{H^9_{co}}.$$
Using $Zu=\varphi w_h^{\bot}+\varphi\na_xv$ and $\pa_yv=-\nabla_x\cdot u$, the first term is bounded by
\begin{align}
&\frac{\varepsilon^2}{100}\|(\partial_y+|D_x|) w_e\|^2_{H^9_{co}}+
C\varepsilon^{-2}\Big(\varepsilon^4+\|U\big\|^2_{H^{10}_{tan}}+\|w\big\|^2_{H^9_{co}}\Big)\|w_{e}\big\|^2_{H^9_e}\nonumber\\
&\quad+C\big(\varepsilon^2+\|U\|_{H^{10}_{tan}}+\|w\|_{H^9_{co}}\big)\|w_e\|^2_{H^9_{co}}.\nonumber
\end{align}
Thus, we obtain
\begin{align}
I_3\leq& \frac{\varepsilon^2}{100}\|(\partial_y+|D_x|) w_e\|^2_{H^9_{co}}+
C\varepsilon^{-2}\Big(\varepsilon^4+\|U\big\|^2_{H^{10}_{tan}}+\|w\big\|^2_{H^9_{co}}\Big)\|w_{e}\big\|^2_{H^9_e}\nonumber\\
&+C\big(\varepsilon^2+\|U\|_{H^{10}_{tan}}+\|w\|_{H^9_{co}}\big)\|w_e\|^2_{H^9_{co}}.\nonumber
\end{align}

Now, $I_6$ can be controlled by
\begin{align}
&\Big|\sum\limits_{|i|+j\leq 9,|m|+n\leq 5}\int_{\R^3_+}\partial_x^{i-m}Z^{j-n}w_{e}\cdot\partial_x^{m}Z^{n}\nabla U ,\partial_x^iZ^jw_{e}dxdy\Big|\nonumber\\
&+\Big|\sum\limits_{|i|+j\leq 9,|m|+n\geq 6}\int_{\R^3_+}\partial_x^{i-m}Z^{j-n}w_{e}\cdot\partial_x^{m}Z^{n}\nabla U ,\partial_x^iZ^jw_{e}dxdy\Big|\tre I_{61}+I_{62}.\nonumber
\end{align}
By Sobolev inequality and $\partial_yu= w_h^{\bot}+\na_xv$, we deduce that
\begin{align}
&I_{61}
\leq C\|w_{e}\|^2_{H^9_{co}}\Big(\|w\|^{\frac12}_{H^7_{co}}\|\partial_yw\|^{\frac12}_{H^7_{co}}+\|w\|_{H^9_{co}}+\|U\|_{H^{10}_{tan}}\Big),\nonumber\\
&I_{62}\leq C\|w_{e}\|_{H^9_{co}}\big(\|w\|_{H^9_{co}}+\|U\|_{H^{10}_{tan}}\big)\|w_e\|_{H^7_{co}}^\f12\|\nabla w_e\|_{H^7_{co}}^{\frac12}.\nonumber
\end{align}
This shows that
\begin{align}
I_6\leq& C\varepsilon^{-\f23}E(t)^\f 53+CE(t)^\f32+\frac{\varepsilon^2}{100}\|(\partial_y+|D_x|)w\|^2_{H^9_{co}}.\nonumber
\end{align}

\no{\bf Step 4. Estimate of $I_4$ and $I_5$.}

By Lemma \ref{e:uniform boundness for approximate solution} and $\partial_yu=w_h^{\bot}+\partial_xv$, we deduce that
\begin{align}
I_4\leq \Big|\sum\limits_{|i|+j\leq 9}\big<\partial_x^iZ^j(w_{a,e,h}\cdot\na_xU+w_{a,e,3}\partial_yU),\partial_x^iZ^jw_e\big>\Big|
\leq C\big(\|U\big\|^2_{H^{10}_{tan}}+\|w\big\|^2_{H^9_{co}}\Big).\nonumber
\end{align}
And $I_5$ can be controlled by
\begin{align}
&\Big|\sum\limits_{|i|+j\leq 9}\big<\partial_x^iZ^j(w_{e,3}\partial_yu_{a,p}),\partial_x^iZ^jw_{e,h}\big>\Big|\nonumber\\
&+\sum\limits_{|i|+j\leq 9}\Big|\big<\partial_x^iZ^j(w_{e,3}\partial_yv_{a,p}),\partial_x^iZ^jw_{e,3}\big>+\big<\partial_x^iZ^j(w_{e,3}\partial_yU_{a,e}+w_{e,h}\partial_xU_a),\partial_x^iZ^jw_e\big>\Big|.\nonumber
\end{align}
By Lemma \ref{e:uniform boundness for approximate solution}, the second term is bounded by $ C\|w_{e}\|^2_{H^9_{co}}$, and
the first term can be controlled by
$
C\Big(\|w_{e}\|^2_{H^9_{co}}+\Big\|\frac{w_{e,3}}{\varepsilon}\Big\|^2_{H^9_{co}}\Big).$ Hence,
\begin{align}
I_5\leq C\Big(\|w\big\|^2_{H^9_{co}}+\Big\|\frac{w_{e,3}}{\varepsilon}\Big\|^2_{H^9_{co}}\Big).\nonumber
\end{align}

\no{\bf Step 5. Estimate of dissipative term.}

We get, by integrating by parts, that
 \begin{align}
-\varepsilon^2\sum\limits_{|i|+j\leq 9}\big<\partial_x^iZ^j(\Delta  w_e),\partial_x^iZ^jw_{e}\big>
=&\varepsilon^2\|Aw_e\|^2_{H^9_{co}}-\varepsilon^2\sum\limits_{|i|+j\leq 9,j\geq 1}\big<[\partial_x^iZ^j,\partial_y](A  w_e),\partial_x^iZ^jw_{e}\big>\nonumber\\
&-\varepsilon^2\sum\limits_{|i|\leq 9}\int_{\R^2}\partial_x^i\na_x(\Lambda_{ND}(\gamma\nabla_x\cdot w_{e,h})) \cdot\partial_x^iw_{e,h}(t,x,0)dx\nonumber\\
&+\varepsilon^2\sum\limits_{|i|+j\leq 9,j\geq 1}\big<\partial_x^iZ^j  (Aw_e),[\partial_y,\partial_x^iZ^j]w_{e}\big>\nonumber\\
\geq&\frac{\varepsilon^2}{2}\big\|Aw_e\big\|^2_{H^9_{co}}
-C_0\delta^{2}\varepsilon^{2}\big\| w_{e}\big\|^2_{H^9_{co}},\nonumber
\end{align}
where we used
\begin{align}
-\varepsilon^2\sum\limits_{|i|\leq 9}\int_{R^2}\partial_x^i\na_x(\Lambda_{ND}(\gamma\nabla_x\cdot w_{e,h}))\cdot \partial_x^iw_{e,h}(t,x,0)dx\ge 0.\nonumber
\end{align}

Combining the estimates in Step 1-Step 5,  we complete the proof of the proposition.
\end{proof}
\smallskip

Similarly, we can prove the following improved decay estimate
in $\varepsilon$ for $w_{e,3}$ and $\varphi w_e$. Here we omit the details.

\begin{Proposition}\label{e:Sobolev Euler vertical vorticity etimate}
There exists $\delta_0> 0$ such that for any $\delta\in (0,\delta_0)$,
there holds
\begin{align}
\frac{1}{2}\frac{d}{dt}\big\|w_{e,3}\big\|^2_{H^9_{co}}
+\frac{\varepsilon^2}{2}\big\|\nabla w_{e,3}\big\|^2_{H^9_{co}}
\leq&C\varepsilon^2\big(E(t)+\varepsilon^2\big)\big(1+E(t)\big)\nonumber\\
&+\frac{\varepsilon^4}{100}\|(\partial_y+|D_x|)w\|^2_{H^9_{co}}+C\varepsilon^{\frac43}E(t)^{\frac53}.\nonumber
\end{align}
\end{Proposition}

\begin{Proposition}\label{e:Sobolev estimate for weihgted we}
There exists $\delta_0> 0$ such that for any $\delta\in (0,\delta_0)$,
there holds
\begin{align}
\frac{1}{2}\frac{d}{dt}\|\varphi w_e\|^2_{H^9_{co}}+\frac{\varepsilon^2}{2}\big\|\nabla(\varphi w_e)\big\|^2_{H^9_{co}}
\leq&C\varepsilon^2\big(E(t)+\varepsilon^2\big)\big(1+E(t)\big)\nonumber\\
&+\frac{\varepsilon^4}{100}\|(\partial_y+|D_x|)w\|^2_{H^9_{co}}+C\varepsilon^{\frac43}E(t)^{\frac53}.\nonumber
\end{align}
\end{Proposition}

\section{Conormal Sobolev estimates of the vorticity:Prandtl part}

\begin{Proposition}\label{e:prandtl vorticity Sobolev etimate}
There exists $\delta_0> 0$ such that for any $\delta\in (0,\delta_0)$,
there holds
\begin{align}
&\frac{1}{2}\frac{d}{dt}\big\|w_{p}\big\|^2_{H^9_p}
+({\lambda}-C)\Big\|\frac{yw_{p}}{\varepsilon}\Big\|^2_{H^9_{p}}
+\frac{\varepsilon^2}{10}\big\|(\partial_y+|D_x|)w_p\big\|^2_{H^9_p}\nonumber\\
&\leq C\big(E(t)+\varepsilon^2\big)+\frac{\varepsilon^2}{100}\|(\partial_y+|D_x|)w\|^2_{H^9_{co}}+C\ve^{-2}E(t)^2.\nonumber
\end{align}
\end{Proposition}
\begin{proof}
First, acting $\partial_x^iZ^j$ on  both sides of (\ref{e:decompose vorticity equation-2}), and then taking $L^2$ inner product with $e^{2\Psi_p}\partial_x^iZ^j w_{p}$, summing over all $|i|+j\leq 9$, we arrive at
\begin{align}\label{e:energy estimate for prandtl part}
&\frac{1}{2}\frac{d}{dt}\big\|w_{p}\big\|^2_{H^9_p}
+\lambda\Big\|\frac{yw_{p}}{\varepsilon}\Big\|^2_{H^9_p}
-\varepsilon^2\sum\limits_{|i|+j\leq 9}\big<\partial_x^iZ^j(\Delta  w_p),e^{2\Psi_p}\partial_x^iZ^jw_{p}\big>\nonumber\\
&\leq\Big|\sum\limits_{|i|+j\leq 9}\big<\partial_x^iZ^j(\widetilde{U}_{a}\cdot\nabla w_p),e^{2\Psi_p}\partial_x^iZ^jw_{p}\big>\Big|
+\Big|\sum\limits_{|i|+j\leq 9}\big<\partial_x^iZ^j(\widetilde{U}\cdot\nabla w_{a,p}),e^{2\Psi_p}\partial_x^iZ^jw_{p}\big>\Big|\nonumber\\
&\quad+\Big|\sum\limits_{|i|+j\leq 9}\big<\partial_x^iZ^j(\widetilde{U}\cdot\nabla w_p),e^{2\Psi_p}\partial_x^iZ^jw_{p}\big>\Big|
+\Big|\sum\limits_{|i|+j\leq 9}\big<\partial_x^iZ^j(w_{a,p}\cdot\nabla U),e^{2\Psi_p}\partial_x^iZ^jw_{p}\big>\Big|\nonumber\\
&\quad+\Big|\sum\limits_{|i|+j\leq 9}\big<\partial_x^iZ^j(w_{p}\cdot\nabla U_a),e^{2\Psi_p}\partial_x^iZ^jw_{p}\big>\Big|
+\Big|\sum\limits_{|i|+j\leq 9}\big<\partial_x^iZ^j(w_{p}\cdot\nabla U),e^{2\Psi_p}\partial_x^iZ^jw_{p}\big>\Big|\nonumber\\
&\quad+\Big|\sum\limits_{|i|+j\leq 9}\big<\partial_x^iZ^j({\rm curl}(R_{p,h},R_{p,v})-M_p),e^{2\Psi_p}\partial_x^iZ^jw_{p}\big>\Big|
\tre\sum_{i=1}^7I_i.\nonumber
\end{align}

\no{\bf Step 1. Estimate of $I_1$.}

We bound $I_1$ as follows
\begin{align}
I_1\leq
&\sum\limits_{|i|+j\leq 9}\sum\limits_{(m,n)<(i,j)}\Big|\int_{\R^3_+}\big[\partial_x^{i-m}Z^{j-n}u_{a}\partial_x\partial_x^{m}Z^{n}w_p\nonumber\\
&\quad+\partial_x^{i-m}Z^{j-n}(v_{a}-\varepsilon^2f(t,x)e^{-y})\partial_x^{m}Z^{n}\partial_yw_p\big]e^{2\Psi_p}\partial_x^iZ^jw_pdxdy\Big|\nonumber\\
&\qquad+\sum\limits_{|i|+j\leq 9}\Big|\int_{\R^3_+}\big[u_{a}\partial_x\partial_x^iZ^jw_p+(v_{a}-\varepsilon^2f(t,x)e^{-y})\partial_x^iZ^j\partial_yw_p\big]e^{2\Psi_p} \partial_x^iZ^jw_pdxdy\Big|.\nonumber
\end{align}
By Lemma \ref{e:uniform boundness for approximate solution}, the first term can be bounded by $C\|w_p\|^2_{H^9_p}$, where we used $(v_{a}-\varepsilon^2f(t,x)e^{-y})|_{y=0}=0$ and (\ref{e:property of va}).
Integrating by parts, the second term can be bounded by
\begin{align}
&\sum\limits_{|i|+j\leq 9, j\geq 1}\Big|\int_{R^3_+}(v_{a}-\varepsilon^2f(t,x)e^{-y})[\partial_x^iZ^j,\partial_y]w_pe^{2\Psi_p} \partial_x^iZ^jw_pdxdy\Big|+C\|w_p\|^2_{H^9_p}\nonumber\\
&\qquad+C\sum\limits_{|i|+j\leq 9}\Big|\int_{R^3_+}(v_{a}-\varepsilon^2f(t,x)e^{-y})(\partial_x^iZ^jw_p)^2\frac{y}{\varepsilon^2}e^{2\Psi_p} dxdy\Big|.\nonumber
\end{align}
Using $[\partial_x^iZ^j,\partial_y]=\varphi'\partial_x^iZ^{j-1}\partial_y$ and (\ref{e:property of va}), we have
\begin{align}
\sum\limits_{|i|+j\leq 9, j\geq 1}\Big|\int_{R^3_+}(v_{a}-\varepsilon^2f(t,x)e^{-y})[\partial_x^iZ^j,\partial_y]w_pe^{2\Psi_p}\partial_x^iZ^jw_p dxdy\Big|\leq C\|w_p\|^2_{H^9_p},\nonumber
\end{align}
and
\begin{align}
&\sum\limits_{|i|+j\leq 9}\Big|\int_{\R^3_+}(v_{a}-\varepsilon^2f(t,x)e^{-y})(\partial_x^iZ^jw_p)^2\frac{y}{\varepsilon^2}e^{2\Psi_p} dxdy\Big|\nonumber\\
&=\sum\limits_{|i|+j\leq 9}\Big|\int_{\R^3_+}\frac{(v_{a}-\varepsilon^2f(t,x)e^{-y})}{y}(\partial_x^iZ^jw_p)^2\frac{y^2}{\varepsilon^2}e^{2\Psi_p}dxdy\Big|\leq C\Big\|\frac{yw_{p}}{\varepsilon}\Big\|^2_{H^9_p}.\nonumber
\end{align}
Thus, we obtain
\begin{align}
I_1\leq C\Big\|\frac{yw_{p}}{\varepsilon}\Big\|^2_{H^9_p}+ C\|w_p\|^2_{H^9_p}.\nonumber
\end{align}

\no{\bf Step 2. Estimate of $I_2$.}

We bound $I_2$ as follows
\begin{align}
 &\Big|\sum\limits_{|i|+j\leq 9}\int_{y(t)}^{+\infty}\int_{\R^2}\partial_x^iZ^j(u\cdot\na_xw_{a,p}+(v+\varepsilon^2f(t,x)e^{-y})\partial_yw_{a,p}) e^{2\Psi_p}\partial_x^iZ^jw_{p}dxdy\Big|\nonumber\\
&+\Big|\sum\limits_{|i|+j\leq 9}\int_0^{y(t)}\int_{\R^2}(\partial_x^iZ^j(u\cdot\na_xw_{a,p}+(v+\varepsilon^2f(t,x)e^{-y})\partial_yw_{a,p}) e^{2\Psi_p}\partial_x^iZ^jw_{p}dxdy\Big|.\nonumber
\end{align}
Using the fact that $\|e^{\Psi_p}\partial_x^iZ^j\partial_{x,y}w_{a,p}\|_{L^\infty(\R^2\times (y(t),\infty))}\leq C$, $Zu=\varphi w_h^{\bot}+\varphi\na_xv$ and $\pa_yv=-\nabla_x\cdot u$, the first term can be bounded by
\begin{align}
C\big(\|w_p\|^2_{H^9_p}+\|U\|^2_{H^9_{tan}}+\|\varphi w\|^2_{H^8_{co}}\big).\nonumber
\end{align}
Notice that $w_{a,p}=\partial_y(u_{a,p})-\partial_xv_{a,p}$, $Zu=\varphi w_h^{\bot}+\varphi\na_xv$, thus, by Lemma \ref{e:uniform boundness for approximate solution},
we have
\begin{align}
&\Big|\sum\limits_{|i|+j\leq 9}\int_0^{y(t)}\int_{\R^2}\partial_x^iZ^j(u\cdot\na_xw_{a,p}) e^{2\Psi_p}\partial_x^iZ^jw_pdxdy\Big|\nonumber\\
 &\leq C\|w_{p}\|^2_{H^9_p}+C\varepsilon^{-2}\Big(\|U\|^2_{H^{10}_{tan}}+\| (\varphi w)\|^2_{H^9_{co}}\Big).\nonumber
\end{align}
Using $\pa_yv=-\nabla_x\cdot u$, $(v+\varepsilon^2f(t,x)e^{-y})|_{y=0}=0$ and Lemma \ref{e:uniform boundness for approximate solution}, we get
\begin{align}
&\Big|\sum\limits_{|i|+j\leq 9}\int_0^{y(t)}\int_{\R^2}\partial_x^iZ^j((v+\varepsilon^2f(t,x)e^{-y})\partial_yw_{a,p}) e^{2\Psi_p}\partial_x^iZ^jw_{p}dxdy\Big|\nonumber\\
&\leq \Big|\sum\limits_{|i|+j\leq 9}\int_0^{y(t)}\int_{\R^2}\partial_x^iZ^j\Big(\frac{1}{\varepsilon y}\int_0^y\Big(\na_x\cdot u(x,y')+\varepsilon^2f(t,x)e^{-y'}\Big)dy'z(\varepsilon\partial_{xz}v_{a,p}+\partial_{zz}u_{a,p})\Big) e^{2\Psi_p}\partial_x^iZ^jw_{p}dxdy\Big|\nonumber\\
&\leq C\|w_{p}\|^2_{H^9_p}+C\varepsilon^{-2}\Big(\|U\|^2_{H^{10}_{tan}}+\| (\varphi w)\|^2_{H^9_{co}}+\varepsilon^4\Big).\nonumber
\end{align}
Thus, we arrive at
\begin{align}
I_2 \leq C\|w_{p}\|^2_{H^9_p}+C\varepsilon^{-2}\Big(\|U\|^2_{H^{10}_{tan}}+\| (\varphi w)\|^2_{H^9_{co}}+\varepsilon^4\Big).\nonumber
\end{align}

\no{\bf Step 3. Estimate of $I_3$ and $I_6$.}

Following the arguments of Step 3 in Proposition \ref{e:Sobolev euler vorticity etimate}, we obtain
\begin{align}
I_3\leq& \frac{\varepsilon^2}{100}\|(\pa_y+|D_x|)w_p\|^2_{H^9_{p}}+\Big\|\frac{yw_{p}}{\varepsilon}\Big\|^2_{H^9_{p}}+C\varepsilon^{-2}\Big(\varepsilon^4+\|U\big\|^2_{H^{10}_{tan}}+\|w\big\|^2_{H^9_{co}}\Big)\|w_{p}\big\|^2_{H^9_p}\nonumber\\
&\quad+C\big(\varepsilon^2+\|U\|_{H^{10}_{tan}}+\|w\|_{H^9_{co}}\big)\|w_p\|^2_{H^9_{p}},\nonumber
\end{align}
and
\begin{align}
I_6\leq& C\varepsilon^{-\f23}E(t)^\f 53+CE(t)^\f32+\frac{\varepsilon^2}{100}\|(\partial_y+|D_x|)w\|^2_{H^9_{co}}\nonumber\\
&+\frac{\varepsilon^2}{100}\|(\partial_y+|D_x|)w_p\|^2_{H^9_{p}}+\Big\|\frac{yw_{p}}{\varepsilon}\Big\|^2_{H^9_{p}}.\nonumber
\end{align}

\no{\bf Step 4. Estimate of $I_4$.}

We split $I_4$ into two parts
\begin{align}
I_4\leq& \Big|\sum\limits_{|i|+j\leq 9}\big<\partial_x^iZ^j(w_{a,p,h}\cdot\na_xU),e^{2\Psi_p}\partial_x^iZ^jw_{p}\big>\Big|\nonumber\\
&+\Big|\sum\limits_{|i|+j\leq 9}\big<\partial_x^iZ^j(w_{a,p,3}\partial_yU),e^{2\Psi_p}\partial_x^iZ^jw_{p}\big>\Big|\tre I_{41}+I_{42}.\nonumber
\end{align}
Using that $\|e^{\Psi_p}\partial_x^iZ^jw_{a,p,3}\|_{L^\infty}\leq C$, $\partial_yu=w_h^{\bot}+\na_xv$ and $\pa_yv=-\nabla_x\cdot u$, we obtain
\begin{align}
I_{42}\leq C\Big(\big\| w\big\|^2_{H^9_{co}}
+\|U\big\|^2_{H^{10}_{tan}}+\|w_{p}\big\|^2_{H^9_p}\Big),\nonumber
\end{align}
While, using $\big\|e^{\Psi_p}\partial_x^iZ^jw_{a,p,h}\big\|_{L^\infty}\leq \frac{C}{\varepsilon}$, we obtain
\begin{align}
I_{41}\leq C\varepsilon^{-2}\Big(\big\| \varphi w\big\|^2_{H^9_{co}}
+\|U\big\|^2_{H^{10}_{tan}}\Big)+C\|w_{p}\big\|^2_{H^9_p}.\nonumber
\end{align}
This shows that
\begin{align}
I_{4}\leq C\varepsilon^{-2}\Big(\big\| \varphi w\big\|^2_{H^9_{co}}
+\|U\|^2_{H^{10}_{tan}}\Big)+C\big(\|w_{p}\|^2_{H^9_p}+\big\| w\big\|^2_{H^9_{co}}\big).\nonumber
\end{align}

\no{\bf Step 5. Estimate of $I_5$.}

First, we split $I_5$ into two parts
\begin{align}
II^p_5\leq &\Big|\sum\limits_{|i|+j\leq 9}\big<\partial_x^iZ^j(w_{p,h}\partial_xU_a),e^{2\Psi_p}\partial_x^iZ^jw_{p}\big>\Big|\nonumber\\
&+\Big|\sum\limits_{|i|+j\leq 9}\big<\partial_x^iZ^j(w_{p,3}\partial_yU_a),e^{2\Psi_p}\partial_x^iZ^jw_{p}\big>\Big|\tre I_{51}+I_{52}.\nonumber
\end{align}
We get by Lemma \ref{e:uniform boundness for approximate solution} that
$$I_{51}\leq C\|w_{p}\|^2_{H^9_p}.$$
Using  $\|\partial_x^iZ^j\partial_yU_a\|_{L^\infty}\leq \frac{C}{\varepsilon}$, we have
\begin{align}
I_{52}\leq  C\Big(\big\|w_{p}\big\|^2_{H^9_p}+\Big\|\frac{w_{p,3}}{\varepsilon}\Big\|^2_{H^9_p}
\Big).\nonumber
\end{align}
Thus, we get
\begin{align}
I_5\leq C\Big(\big\|w_{p}\big\|^2_{H^9_p}+\Big\|\frac{w_{p,3}}{\varepsilon}\Big\|^2_{H^9_p}\Big).\nonumber
\end{align}

\no{\bf Step 6. Estimate of $I_7$.}

It follows from Lemma \ref{e:uniform boundness for error} that
\begin{align}
I_7\leq C\|w_{p}\|^2_{H^9_p}+\varepsilon^2.\nonumber
\end{align}

\no{\bf Step 7. Estimate of dissipative term.}

 We get, by integrating by parts, that
\begin{align}
&-\varepsilon^2\sum\limits_{|i|+j\leq 9}\big<\partial_x^iZ^j(\Delta  w_p),e^{2\Psi_p}\partial_x^iZ^jw_{p}\big>\nonumber\\
&=\varepsilon^2\|Aw_p\|^2_{H^9_p}-\varepsilon^2\sum\limits_{|i|+j\leq 9,j\geq 1}\big<[\partial_x^iZ^j,\partial_y](A  w_p),e^{2\Psi_p}\partial_x^iZ^jw_{p}\big>_2+\varepsilon^2\sum_{|i|\leq 9}\int_{\R^2}\partial_x^i(Aw_p)\partial_x^iw_{p}(x,0)dx\nonumber\\
&\quad+\varepsilon^2\sum\limits_{|i|+j\leq 9,j\geq 1}\big<\partial_x^iZ^j  (Aw_p),e^{2\Psi_p}[\partial_y,\partial_x^iZ^j]w_p\big>
+4(\delta-\lambda t)\sum\limits_{|i|+j\leq 9}\big<\partial_x^iZ^j  (Aw_p),e^{2\Psi_p}y\partial_x^iZ^jw_{p}\big>\nonumber\\
&\geq\frac{\varepsilon^2}{2}\|Aw_p\|^2_{H^9_p}-C\delta\varepsilon^2\big\|w_{p}\big\|_{H^9_p}
-C\delta^{2}\Big\|\frac{yw_{p}}{\varepsilon}\Big\|_{H^9_p}^{2}+\varepsilon^2\sum_{|i|\leq 9}\int_{\R^2}\partial_x^i(Aw_p)\partial_x^iw_{p}(x,0)dx.
\end{align}
Using the boundary condition of  $(\partial_y+|D_x|)w_p$, we arrive at
\begin{align}
&\varepsilon^2\sum\limits_{|i|\leq 9}\int_{\R^2}\partial_x^i((\partial_y+|D_x|)w_p)\partial_x^iw_{p}(t,x,0)dx\nonumber\\
&=-\varepsilon^2\sum\limits_{|i|\leq 9}\int_{\R^2}\partial_x^i\na_x(\Lambda_{ND}(\gamma\nabla_x\cdot w_{p,h}))\cdot \partial_x^iw_{p,h}(t,x,0)dx\nonumber\\
&\quad+\sum\limits_{|i|\leq 9}\int_{\R^2}\partial_x^iw_{p,h}\partial_x^i(-\partial_y(-\triangle_D)^{-1}J_h+\partial_x(-\triangle_N)^{-1}J_3)(t,x,0)dx\nonumber\\
&\ge \sum\limits_{|i|\leq 9}\int_{\R^2}\partial_x^iw_{p,h}\partial_x^i(-\partial_y(-\triangle_D)^{-1}J_h+\partial_x(-\triangle_N)^{-1}J_3)(t,x,0)dx.\nonumber
\end{align}
Thus, we only need to deal with
$$\sum\limits_{|i|\leq 9}\int_{\R^2}\partial_x^iw_{p,h}\partial_x^i(-\partial_y(-\triangle_D)^{-1}J_h+\partial_x(-\triangle_N)^{-1}J_3)(t,x,0)dx,$$
which can be controlled by
\begin{align}
&\sum\limits_{|i|\leq 9}\Big|\int_{\R^2}\int_0^{\frac{y(t)}{2}}\zeta'(t,y)\partial_x^iw_{p,h}\partial_x^i(-\partial_y(-\triangle_D)^{-1}J_{h}
+\partial_x(-\triangle_N)^{-1}J_{3})dydx\Big|\nonumber\\
&+\sum\limits_{|i|\leq 9}\Big|\int_{\R^2}\int_0^{\frac{y(t)}{2}}\zeta(t,y)\partial_x^i\partial_yw_{p,h}\partial_x^i(-\partial_y(-\triangle_D)^{-1}J_{h}
+\partial_x(-\triangle_N)^{-1}J_{3})dydx\Big|\nonumber\\
&+\sum\limits_{|i|\leq 9}\Big|\int_{\R^2}\int_0^{\frac{y(t)}{2}}\zeta(t,y)\partial_x^iw_{p,h}\partial_x^i\partial_{yy}(-\triangle_D)^{-1}J_{h}dydx\Big|\nonumber\\
&+\sum\limits_{|i|\leq 9}\Big|\int_{\R^2}\int_0^{\frac{y(t)}{2}}\zeta(t,y)\partial_x^iw_{p,h}\partial_x^i\partial_{xy}(-\triangle_N)^{-1}J_{3}) dydx\Big|\tre\sum_{k=1}^4I_b^k,\nonumber
\end{align}
where $\zeta(t,y)$ is a smooth cut-off function which satisfies $\zeta(t,0)=1$ and $\zeta(t,y)=0$ for $y\ge y(t)/2$.

Using $L^2$ boundness of operators $\partial_y(-\triangle_D)^{-1}\partial_x$, $\partial_y(-\triangle_D)^{-1}\partial_y$, $\partial_{yy}(-\triangle_D)^{-1}$, $\nabla(-\triangle_N)^{-1}\partial_x$ and $\partial_{xy}(-\triangle_N)^{-1}$,  we can deduce from Lemma \ref{lem:F-Sob} that
\begin{align}
I_b^1\leq& C\|w_p\|^2_{H^9_p}+C\|\widetilde{U}_a\cdot\nabla U+\widetilde{U}\cdot\nabla U_a+\widetilde{U}\cdot\nabla U+R\|^2_{H^9_{tan}}\nonumber\\
\leq &C\|w_p\|^2_{H^9_p}+C\ve^2\big(\ve^2+E(t)\big)(E(t)+1),\nonumber
\end{align}
and by Lemma \ref{e:one order elliptic estimate},
\begin{align}
I_b^2\leq& \kappa\varepsilon^2\|\partial_yw_p\|^2_{H^9_{tan}}+C\varepsilon^{-2}\|\widetilde{U}_a\cdot\nabla U+\widetilde{U}\cdot\nabla U_a+\widetilde{U}\cdot\nabla U+R\|^2_{H^9_{tan}}\nonumber\\
\leq &\kappa\varepsilon^2\|(\partial_y+|D_x|)w_p\|^2_{H^9_p}+C\big(\ve^2+E(t)\big)(E(t)+1).\nonumber
\end{align}
On the other hand, using the fact $\phi(t,y)\geq c\delta>0$ for $y\leq \frac{y(t)}{2}$, we deduce that
\begin{align*}
I_b^4\leq& C\|(w_p)_\Phi\|^2_{H^{8}_p}+C\|\widetilde{U}_a\cdot\nabla U+\widetilde{U}\cdot\nabla U_a+\widetilde{U}\cdot\nabla U+R\|^2_{H^9_{tan}}\nonumber\\
\leq& C\|(w_p)_\Phi\|^2_{H^8_p}+C\ve^2\big(\ve^2+E(t)\big)(E(t)+1).
\end{align*}
We write
\begin{align*}
I_3^b\le& \sum\limits_{|i|\leq 9}\Big|\int_{\R^2}\int_0^{\frac{y(t)}{2}}\zeta(t,y)\partial_x^iw_{p,h}\partial_x^iJ_{h}dydx\Big|\nonumber\\
&+\sum\limits_{|i|\leq 9}\Big|\int_{\R^2}\int_0^{\frac{y(t)}{2}}\zeta(t,y)\partial_x^iw_{p,h}\partial_x^i\partial_{xx}(-\triangle_D)^{-1}J_{h}dydx\Big|.\nonumber\end{align*}
Noticing that $(F,G)(t,x,0)=0$, the first term can be handed as $I_1^b, I_2^b$ by integrating by parts, while the second term can be handled as $I_4^b$.

Finally, fixing $\kappa$ small, we obtain
\begin{align}
&-\varepsilon^2\sum\limits_{|i|+j\leq 9}\big<\partial_x^iZ^j(\Delta w_p),e^{2\Psi_p}\partial_x^iZ^jw_{p}\big>\nonumber\\
&\geq\frac{\varepsilon^2}{4}\|Aw_p\|^2_{H^9_p}-C\Big\|\frac{yw_p}{\varepsilon}\Big\|_{H^9_p}^{2}-C\big(\varepsilon^2+E(t)\big)\big(1+E(t)\big).\nonumber
\end{align}

The proposition follows by combing the estimates in Step 1-Step 7.
\end{proof}

\smallskip

Similarly, we can prove the following improved decay estimate
in $\varepsilon$ for $w_{p,3}$ and $\varphi w_p$. Again we omit the details.

\begin{Proposition}\label{e:prandtl vertical vorticity Sobolev etimate}
There exists $\delta_0> 0$ such that for any $\delta\in (0,\delta_0)$,
there holds
\begin{align}
&\frac{1}{2}\frac{d}{dt}\|w_{p,3}\|^2_{H^9_p}
+({\lambda}-C)\Big\|\frac{yw_{p,3}}{\varepsilon}\Big\|^2_{H^9_p}
+\frac{\varepsilon^2}{2}\big\|\nabla w_{p,3}\big\|^2_{H^9_p}\nonumber\\
&\leq C\varepsilon^2\big(E(t)+\varepsilon^2\big)\big(1+E(t)\big)+\frac{\varepsilon^4}{100}\|(\partial_y+|D_x|)w\|^2_{H^9_{co}}+C\varepsilon^{\frac43}E(t)^{\frac53}.\nonumber
\end{align}
\end{Proposition}

\begin{Proposition}\label{e:Sobolev estimate for weihgted wp}
There exists $\delta_0> 0$ such that for any $\delta\in (0,\delta_0)$,
there holds
\begin{align}
&\frac{1}{2}\frac{d}{dt}\|\varphi w_p\|^2_{H^9_{p}}+({\lambda}-C)\Big\|\frac{y(\varphi w_{p})}{\varepsilon}\Big\|^2_{H^9_p}+\frac{\varepsilon^2}{2}\big\|\nabla(\varphi w_p)\big\|^2_{H^9_{p}}\nonumber\\
&\leq C\varepsilon^2\big(E(t)+\varepsilon^2\big)\big(1+E(t)\big)+\frac{\varepsilon^4}{100}\|(\partial_y+|D_x|)w\|^2_{H^9_{co}}+C\varepsilon^{\frac43}E(t)^{\frac53}.\nonumber
\end{align}
\end{Proposition}

\section{Appendix}

In this appendix, we prove the well-posedness of the Euler system (\ref{e:Euler equation}) and Prandtl system (\ref{e:prandtl equation}). The proof of well-posedness of the linearized equations (\ref{e:linearized Euler equation}) and (\ref{e:linearized prandtl equation}) is similar, thus we omit the details.

\subsection{Well-posedness of the Euler system}

\begin{Proposition}\label{e:unifrom boundness for Euler approximate}
Let $(u_0,v_0)\in H^{30}(\R^3_+)$ with ${\rm div}(u_0, v_0)=0$ and $(u_0, v_0)(x,0)=0$. Moreover, assume that ${\rm curl}(u_0, v_0)=0$ in the domain $\{(x,y)\in \R^3_+:y\leq 2\}.$ Then there exists $T_e>0$ such that the system (\ref{e:Euler equation}) has a unique solution $U^e=(u^e,v^e)$ on $[0,T_e]$, which satisfies
\begin{align}
\sup_{0\leq t\leq T_e}\|U_{\Phi_e}^e\|^2_{H^{30}}\leq C,\quad \sup_{0\leq t\leq T_e}\|(\partial_tU^e)_{\Phi_e}\|^2_{H^{29}}\leq C,\nonumber
\end{align}
where $\Phi_e=(1-y)_{+}\langle\xi\rangle$ and $H^m$ is the usual Sobolev space.
\end{Proposition}

\begin{proof}
Here we only present a priori estimate of the solution. We consider the vorticity equation of the system (\ref{e:Euler equation})
\begin{align}
\left\{
\begin{aligned}
&\partial_t w^e+ U^e\cdot \nabla w^e-w^e\cdot\nabla U^e=0,\\
&w^e(0,x,y)={\rm curl}(u_0, v_0)\tre w^e_0.\nonumber
\end{aligned}
\right.
\end{align}

First of all, the standard energy estimate ensures that
\beno
\sup_{t\in [0,T_e]}\|w^e(t)\|_{H^{29}}\leq C
\eeno
for some $T_e>0$. Because of  $w^e_0=0$ in $\{(x,y)\in \R^3_+:y\leq 2\}$,  we can deduce that $w^e(t,x,y)=0$ in $\big\{(x,y)\in \R^3_+:y\leq \frac32\big\}$ for any $t\in [0, T_e]$(take $T_e$ smaller if necessary). 

Thanks to
\beno
\triangle v^e=\partial_{x_2}w_1^e-\partial_{x_1}w_2^e,\quad v^e|_{y=0}=0,
\eeno
we deduce that
\ben\label{eq:ve-est}
\|(\nabla v^e)_{\Phi_e}\|^2_{H^{29}}\leq C\big(\|w^e\|^2_{H^{29}}+\| v^e_{\Phi_e}\|^2_2\big).
\een
On the other hand,
$$\partial_yu^e=w_h^{e,\bot}+\partial_xv^e,\quad \nabla_x\cdot u^e=-\partial_yv^e,\quad {\rm curl}_x u^e=w_3^e,
$$
therefore,
\beno
&\triangle_x u^e_1=-\partial_{yx_1}v^e+\partial_{x_2}w^e_3,\quad \partial_yu^e_1=w^e_2+\partial_{x_1}v^e,\\
&\triangle_x u^e_2=-\partial_{yx_2}v^e-\partial_{x_1}w^e_3,\quad \partial_yu^e_2=-w^e_1+\partial_{x_2}v^e,
\eeno
which along with \eqref{eq:ve-est} imply that
$$\|(\nabla u^e)_{\Phi_e}\|^2_{H^{29}}\leq C\big(\|w^e\|^2_{H^{29}}+\| U^e_{\Phi_e}\|^2_2\big).$$
Thus, we arrive at
\beno
\|(\nabla U^e)_{\Phi_e}\|^2_{H^{29}}\leq C\big(\|w^e\|^2_{H^{29}}+\| U^e_{\Phi_e}\|^2_2\big),
\eeno
which implies 
\begin{align}
\|U^e_{\Phi_e}\|^2_{H^{30}}\leq C\big(\|w^e\|^2_{H^{29}}+\| U^e\|^2_2\big)\leq C\nonumber
\end{align}
by using the fact that  for any $\varepsilon\in (0,1)$,
\beno
\| U^e_{\Phi_e}\|^2_2\leq \varepsilon\|(\nabla U^e)_{\Phi_e}\|^2_2+ C(\varepsilon)\| U^e\|^2_2.
\eeno

A similar argument as above gives
\beno
\|(\partial_t U^e)_{\Phi_e}\|^2_{H^{29}}\leq C\big(\|\partial_tw^e\|^2_{H^{28}}+\|\partial_t U^e\|^2_2\big).
\eeno
Now, we deduce
$
\|\partial_tw^e\|_{H^{28}}\leq C
$
from the vorticity equation, and $\|\partial_t U^e\|_2\leq C$ from the velocity equation. Thus, 
\beno
\|(\partial_t U^e)_{\Phi_e}\|^2_{H^{29}}\leq C.
\eeno
This completes the proof.
\end{proof}

\subsection{Well-posedness of the Prandtl system}
To prove the well-posedness of the Prandtl system, we first introduce some weighted norms
\begin{align*}
\|U^p\|^2_{\overline{H}^{m}_p}=&\sum\limits_{|j|+k\leq m}\int_{\R^3_+}\big| e^{\phi_p(t,z)}\partial_x^j\widetilde{Z}^kU^p\big|^2dxdz,\\
\|U^p\|^2_{\overline{H}^{m,\frac12}_p}=&\sum\limits_{|j|+k\leq m}\int_{\R^3_+}\big| e^{\phi_p(t,z)}\langle D_x \rangle^{\frac12}\partial_x^j\widetilde{Z}^kU^p\big|^2dxdz+\sum\limits_{|j|+k\leq m}\int_{\R^3_+}\big| ze^{\phi_p(t,z)} \partial_x^j\widetilde{Z}^kU^p\big|^2dxdz.
\end{align*}
where $\phi_p(t,z)=\rho_p(t)z^2$ with $\rho_p(t)=1-\lambda_pt\geq 1$ and $\lambda_p$ defined later.
\begin{Proposition}\label{e:well posedness of prandtl equation}
Let $(u^e,v^e)$ be given as above proposition. There exists $T_P>0$ such that the system (\ref{e:prandtl equation}) has a unique solution $U^p=(u^p,v^p)$ on $[0,T_P]$, which satisfies
\begin{align}
\sup_{0\leq t\leq T_P}\|U^p_{\Phi_p}\|^2_{\overline{H}^{27}_p}\leq C_0,\quad \sup_{0\leq t\leq T_P}\Big(\|(\partial_tU^p)_{\Phi_p}\|^2_{\overline{H}^{24}_p}+\|(\partial_z^2U^p)_{\Phi_p}\|^2_{\overline{H}^{24}_p}\Big)\leq C,\nonumber
\end{align}
where $\Phi_p=\rho_p(t)\langle\xi\rangle$.
\end{Proposition}

\begin{proof}
As in Proposition \ref{e:unifrom boundness for Euler approximate}, we only give a priori estimate. Recall that $u^p$ satisfies
\begin{align}
\left\{
\begin{aligned}
&\partial_tu^p-\partial_{zz}u^p+u^p\cdot\na_xu^e(t,x,0)+\big(u^e(t,x,0)+u^p\big)\cdot\na_xu^p
\nonumber\\&\qquad\qquad\qquad\qquad+\Big(-\int_0^z\na_x\cdot u^p(t,x,z)dz+ z\partial_yv^e(t,x,0)\Big)\partial_zu^p=0,\\[3pt]
&u^p(0,x,z)=0,\\[3pt]
&\lim\limits_{z\rightarrow \infty}u^p(t,x,z)=0,\quad u_p(t,x,0)=-u^e(t,x,0).
\end{aligned}
\right.
\end{align}
We set
$$\overline{u}^p=u^p+e^{-2\phi_p(t,z)}u^e(t,x,0)\tre u^p+g.$$
Thus, it's easy to verify that $\overline{u}^p$ satisfies
\begin{align}\label{e:mofified prandtl equation}
\left\{
\begin{aligned}
&\partial_t\overline{u}^p-\partial_{zz}\overline{u}^p+F^p=0,\\[3pt]
&\overline{u}^p(0,x,y)=0,\\[3pt]
&\lim\limits_{z\rightarrow \infty}\overline{u}^p(t,x,z)=0,\quad \overline{u}^p(t,x,0)=0,
\end{aligned}
\right.
\end{align}
where
\begin{align}
F^p=&-\partial_tg+\partial_{zz} g+(\overline{u}^p-g)\cdot\na_xu^e(t,x,0)+(u^e(t,x,0)+\overline{u}^p-g)\cdot\na_x(\overline{u}^p-g)\nonumber\\
&+\Big(-\int_0^z\na_x\cdot(\overline{u}^p-g)(t,x,z')dz'+z\partial_yv^e(t,x,0)\Big)\partial_z(\overline{u}^p-g).\nonumber
\end{align}

Acting $\partial_x^j\widetilde{Z}^ke^{\rho_p(t)\delta\langle D_x\rangle}$ on both sides of (\ref{e:mofified prandtl equation}), then taking $L^2$ inner product with $e^{2\phi_p(t,z)}\partial_x^j\widetilde{Z}^k \overline{u}_{\rho_p}^p$, integrating by parts, summing over $|j|+k\leq 27$ and fixing $\delta$ small, we arrive at
\begin{align}
\frac12\frac{d}{dt}\|\overline{u}_{\Phi_p}^p\|^2_{\overline{H}^{27}_p}+\frac{\lambda_p}{2}\big\|\overline{u}^p_{\Phi_p}\big\|^2_{\overline{H}^{27,\frac12}_p}
+\frac{1}{2}\big\|(\partial_z\overline{u}^p)_{\Phi_p}\big\|^2_{\overline{H}^{27}_p}\leq C\sum_{|j|+k\leq 27}\Big|\Big<\partial_x^j\widetilde{Z}^kF_{\Phi_p}^p, e^{2\phi_p(t,z)}\partial_x^j\widetilde{Z}^k \overline{u}_{\Phi_p}^p\Big>\Big|.\nonumber
\end{align}

Notice that $g=e^{-2\phi_p(t,z)}u^e(t,x,0)$. It is easy to get  by Proposition \ref{e:unifrom boundness for Euler approximate} that
\begin{align}
&\sum_{|j|+k\leq 27}\Big|\Big<\partial_x^j\widetilde{Z}^k(\partial_tg+\triangle g+g\partial_xu^e(t,x,0)+(u^e(t,x,0)-g)\partial_xg)_{\Phi_p}, e^{2\phi_p(t,z)}\partial_x^j\widetilde{Z}^k \overline{u}_{\Phi_p}^p\Big>\Big|\nonumber\\
&\quad\leq  C+\|\overline{u}_{\Phi_p}^p\|^2_{\overline{H}^{27}_p},\nonumber
\end{align}
and
\begin{align}
&\sum_{|j|+k\leq 27}\Big|\Big<\partial_x^j\widetilde{Z}^k\Big(\int_0^z\partial_xg(t,x,z')dz'\partial_zg+z\partial_yv^e(t,x,0)\partial_zg\Big)_{\Phi_p}, e^{2\phi_p(t,z)}\partial_x^j\widetilde{Z}^k \overline{u}_{\Phi_p}^p\Big>\Big|\nonumber\\
&\quad\leq  C+\|\overline{u}_{\Phi_p}^p\|^2_{\overline{H}^{27}_p}.\nonumber
\end{align}
By Proposition \ref{e:unifrom boundness for Euler approximate} again, we have
\begin{align}
&\sum_{|j|+k\leq 27}\Big|\Big<\partial_x^j\widetilde{Z}^k(\overline{u}^p\partial_xu^e(t,x,0))_{\Phi_p}, e^{2\phi_p(t,z)}\partial_x^j\widetilde{Z}^k \overline{u}_{\Phi_p}^p\Big>\Big|
\leq  C\|\overline{u}_\Phi^p\|^2_{\overline{H}^{27}_p}.\nonumber
\end{align}
Finally, we consider the transport term
\begin{align}
&\sum_{|j|+k\leq 27}\Big|\Big<\partial_x^j\widetilde{Z}^k\Big(\overline{u}^p\partial_x\overline{u}^p
-\int_0^z\partial_x\overline{u}^p(t,x,z')dz'\partial_z\overline{u}^p\Big)_{\Phi_p}, e^{2\phi_p(t,z)}\partial_x^j\widetilde{Z}^k \overline{u}_{\Phi_p}^p\Big>\Big|.\nonumber
\end{align}
First, there holds
\begin{align}
&\sum_{|j|+k\leq 27}\Big|\int_{\R^3_+}\partial_x^j\widetilde{Z}^k\big(\overline{u}^p\partial_x\overline{u}^p
\big)_{\Phi_p} e^{2\phi_p(t,z)}\partial_x^j\widetilde{Z}^k \overline{u}_{\Phi_p}^pdxdz\Big|\nonumber\\
&\quad\leq  C\Big(\|\overline{u}_{\Phi_p}^p\|_{\overline{H}^{27}_p}+1\Big)\|\overline{u}_{\Phi_p}^p\|^2_{\overline{H}^{27,\frac12}_p}
+\frac{1}{10}\big\|(\partial_z\overline{u}^p)_{\Phi_p}\big\|^2_{\overline{H}^{27}_p}.\nonumber
\end{align}
Then, a direct computation gives
\begin{align}
&\sum_{|j|+k\leq 27}\Big|\int_{\R^3_+}\partial_x^j\widetilde{Z}^k\Big(\int_0^z\partial_x\overline{u}^p(t,x,z')dz'\partial_z\overline{u}^p\Big)_{\Phi_p} e^{2\phi_p(t,z)}\partial_x^j\widetilde{Z}^k \overline{u}_{\Phi_p}^pdxdz\Big|\nonumber\\
&\quad\leq  C\|(\partial_z\overline{u}^p)_{\Phi_p}\|_{\overline{H}^{27}_p}\|\overline{u}_{\Phi_p}^p\|^2_{\overline{H}^{27}_p}
+C\|\overline{u}_{\Phi_p}^p\|^2_{\overline{H}^{27}_p}\|\overline{u}_{\Phi_p}^p\|^2_{\overline{H}^{27,\frac12}_p}.\nonumber
\end{align}
Thus, collecting these estimates, we arrive at
\begin{align}
&\frac12\frac{d}{dt}\|\overline{u}_{\Phi_p}^p\|^2_{\overline{H}^{27}_p}+\frac{\lambda_p}{4}\big\|\overline{u}^p_{\Phi_p}\big\|^2_{\overline{H}^{27,\frac12}_p}
+\frac{1}{10}\big\|(\partial_z\overline{u}^p)_{\Phi_p}\big\|^2_{\overline{H}^{27}_p}\nonumber\\
&\quad\leq C\big(1+\|(\overline{u}^p)_{\Phi_p}\|^2_{H^{27}_p}\big)^2+C\Big(\|\overline{u}_{\Phi_p}^p\|^2_{\overline{H}^{27}_p}+1\Big)\|\overline{u}_{\Phi_p}^p\|^2_{\overline{H}^{27,\frac12}_p}.\nonumber
\end{align}
With this, a continuous argument ensures that there exists $T_p> 0$ so that
$$\sup_{0\leq t\leq T_p}\|\overline{u}_{\Phi_p}^p\|^2_{\overline{H}^{27}_p}\leq C,$$
from which and  Proposition \ref{e:unifrom boundness for Euler approximate}, we infer that
$$\sup_{0\leq t\leq T_p}\|(u^p,v^p)_{\Phi_p}\|^2_{\overline{H}^{26}_p}\leq C.$$
For the second estimate, we can first show that
$$\sup_{0\leq t\leq T_p}\|\partial_t(u^p,v^p)_{\Phi_p}\|^2_{\overline{H}^{24}_p}\leq C,$$
then the desired estimate can be deduced by using the equation.
\end{proof}

\subsection{Proof of Lemma \ref{e:uniform boundness for approximate solution} and Lemma \ref{e:uniform boundness for error}.}

By the same arguments as in Proposition \ref{e:unifrom boundness for Euler approximate} and Proposition \ref{e:well posedness of prandtl equation}, we can prove the well-posedness of the linearized Euler equation (\ref{e:linearized Euler equation}) and linearized Prandtl equation (\ref{e:linearized prandtl equation}) with the associated initial-boundary conditions. Especially, there hold the following uniform bounds for the approximate solution:
\begin{align}
&\|(u_e^{(0)},v_e^{(0)})_{\Phi_e}\|^2_{H^{30}}\leq C,\quad \|(u_e^{(1)},v_e^{(1)})_{\Phi_e}\|^2_{H^{28}}\leq C,\quad \|\partial_t(u_e^{(0)},v_e^{(0)})_{\Phi_e}\|^2_{H^{29}}\leq C,\nonumber\\
& \|\partial_t(u_e^{(1)},v_e^{(1)})_{\Phi_e}\|^2_{H^{17}}\leq C,\quad
\|(u_p^{(0)},v_p^{(1)})_{\Phi_p}\|^2_{\overline{H}^{27}_p}\leq C,\quad \|(u_p^{(1)},v_p^{(2)})_{\Phi_p}\|^2_{\overline{H}^{20}_p}\leq C,\nonumber\\
& \|\partial_t(u_p^{(0)},v_p^{(1)})_{\Phi_p}\|^2_{\overline{H}^{24}_p}+\|\partial_{zz}(u_p^{(0)},v_p^{(1)})_{\Phi_p}\|^2_{\overline{H}^{24}_p}\leq C,\nonumber\\
& \|\partial_t(u_p^{(1)},v_p^{(2)})_{\Phi_p}\|^2_{\overline{H}^{17}_p}+\|\partial_{zz}(u_p^{(1)},v_p^{(2)})_{\Phi_p}\|^2_{\overline{H}^{17}_p}\leq C.\nonumber
\end{align}
With these bounds, Lemma \ref{e:uniform boundness for approximate solution} and Lemma \ref{e:uniform boundness for error} follow easily.

\bigskip

\noindent {\bf Acknowledgments.}
The authors are deeply grateful to Chao Wang for his valuable discussions and Yuxi Wang for careful reading of the manuscript.
M. Fei is partially supported by NSF of China under Grant 11301005 and AHNSF grant 1608085MA13. Z. Zhang is partially supported by NSF of China under Grant
11371039 and 11421101.

\end{document}